\documentclass[pdflatex,sn-mathphys-ay]{sn-jnl}

\usepackage{graphicx}%
\usepackage{multirow}%
\usepackage{amsmath,amssymb,amsfonts}%
\usepackage{amsthm}%
\usepackage{mathrsfs}%
\usepackage[title]{appendix}%
\usepackage{xcolor}%
\usepackage{textcomp}%
\usepackage{manyfoot}%
\usepackage{booktabs}%
\usepackage{algorithm}%
\usepackage{algorithmicx}%
\usepackage{algpseudocode}%
\usepackage{listings}%

\usepackage{mathtools}
\usepackage{aliascnt}
\usepackage{bm}
\usepackage{tikz}

\hypersetup{
  colorlinks=true,
  citecolor=blue,
  urlcolor=blue,
  linkcolor=blue
}
\usepackage[nameinlink,noabbrev]{cleveref}
\numberwithin{equation}{section}

\crefname{theorem}{Theorem}{Theorems}
\Crefname{theorem}{Theorem}{Theorems}
\crefname{proposition}{Proposition}{Propositions}
\Crefname{proposition}{Proposition}{Propositions}
\crefname{lemma}{Lemma}{Lemmas}
\Crefname{lemma}{Lemma}{Lemmas}
\crefname{corollary}{Corollary}{Corollaries}
\Crefname{corollary}{Corollary}{Corollaries}
\crefname{remark}{Remark}{Remarks}
\Crefname{remark}{Remark}{Remarks}
\crefname{section}{Section}{Sections}
\Crefname{section}{Section}{Sections}
\crefname{subsection}{Section}{Sections}
\Crefname{subsection}{Section}{Sections}
\crefformat{equation}{#2(#1)#3}
\Crefformat{equation}{#2(#1)#3}
\crefrangeformat{equation}{#3(#1)#4--#5(#2)#6}
\Crefrangeformat{equation}{#3(#1)#4--#5(#2)#6}
\crefmultiformat{equation}
  {#2(#1)#3}{ and~#2(#1)#3}{, #2(#1)#3}{, and~#2(#1)#3}
\Crefmultiformat{equation}
  {#2(#1)#3}{ and~#2(#1)#3}{, #2(#1)#3}{, and~#2(#1)#3}

\numberwithin{equation}{section}
\newcommand{\Z}{\mathbb Z}

\newcommand{\supp}{\operatorname{supp}}

\theoremstyle{thmstyleone}%
\newtheorem{theorem}{Theorem}
\newaliascnt{proposition}{theorem}
\newtheorem{proposition}[proposition]{Proposition}%
\aliascntresetthe{proposition}

\newaliascnt{corollary}{theorem}
\newtheorem{corollary}[corollary]{Corollary}
\aliascntresetthe{corollary}

\theoremstyle{thmstyletwo}%
\newtheorem{remark}{Remark}%
\newtheorem{lemma}{Lemma}

\theoremstyle{thmstylethree}%

\begin{document}

\title[Normal approximation for total magnetization in subcritical long-range Ising models]{Normal approximation for total magnetization in subcritical long-range Ising models}





\author[1]{\fnm{Song-Hao} \sur{Liu}}\email{liusonghao@dlut.edu.cn}

\author*[2]{\fnm{Qi-Man} \sur{Shao}}\email{shaoqm@sustech.edu.cn}

\author[3]{\fnm{Yang} \sur{Xie}}\email{xieyang@link.cuhk.edu.hk}

\author[4]{\fnm{Jing-Yu} \sur{Xu}}\email{12131253@mail.sustech.edu.cn}




\affil[1]{%
  \orgdiv{School of Mathematical Sciences},
  \orgname{Dalian University of Technology},
  \orgaddress{
    \city{Dalian},
    \country{China}
  }
}

\affil*[2]{%
  \orgdiv{Department of Statistics and Data Science, Shenzhen International
Center for Mathematics},
  \orgname{Southern University of Science and Technology},
  \orgaddress{
    \city{Shenzhen},
    \country{China}
  }
}

\affil[3]{%
  \orgdiv{Department of Statistics},
  \orgname{The Chinese University of Hong Kong},
  \orgaddress{
    \city{Hong Kong},
    \country{China}
  }
}
 
\affil[4]{%
  \orgdiv{Department of Statistics and Data Science},
  \orgname{Southern University of Science and Technology},
  \orgaddress{
    \city{Shenzhen},
    \country{China}
  }
}


\abstract{
\unboldmath

We establish Berry--Esseen bounds and 1-Wassershtein bounds for the total
spin of strictly subcritical zero-field Ising models with
summable interactions.  Let $J:\mathbb Z^d\to[0,\infty)$ be symmetric,
satisfy $J(0)=0$, and $\|J\|_1:=\sum_{x\in\mathbb Z^d}J(x)<\infty$.
For a nonempty finite set
$\Lambda\subset\mathbb Z^d$, consider zero-field Gibbs measure with free
boundary conditions, in which a spin configuration 
$\sigma\in\{-1,+1\}^{\Lambda}$ has probability proportional to
\[
 \exp\!\left\{
   \beta\sum_{x,y\in\Lambda}
   J(x-y)\sigma_x\sigma_y
 \right\}.
\]
Let $S_\Lambda:=\sum_{x\in\Lambda}\sigma_x$ be the total spin and $\beta_{\mathrm c}$ denote the critical inverse temperature.
For every fixed  $0\leq\beta<\beta_{\mathrm c}$, we prove that
\[
 d_{\mathrm K}\!\left(
   \frac{S_\Lambda}{\sqrt{\operatorname{Var}(S_\Lambda)}},G
 \right)
 \leq C_{\beta,J}|\Lambda|^{-1/2},
\]
where $d_{\mathrm K}$ is the Kolmogorov distance, $G\sim \mathcal{N}(0,1)$, and $C_{\beta,J}$ depends only on $\beta$ and $J$.
For the fully periodized model on
$(\mathbb Z/L\mathbb Z)^d$, the corresponding rate is $L^{-d/2}$,
uniformly in $L$.
For the classical nearest-neighbor model,
the Kolmogorov bound also holds for
$-\beta_{\mathrm c}<\beta<\beta_{\mathrm c}$, covering both
ferromagnetic and antiferromagnetic interactions. 
We also obtain $1$-Wasserstein bounds with the same rates in the
respective subcritical regimes.
Both the Kolmogorov and $1$-Wasserstein bounds hold with the same rate under the corresponding
infinite-volume free-boundary Gibbs measures.
}




\maketitle

\section{Introduction}\label{sec:introduction}

We study quantitative normal approximation for the total spin in
zero-field ferromagnetic Ising models on $\mathbb Z^d$ with general
summable interactions.
The interaction may extend over arbitrarily large distances, and
no particular spatial decay rate is assumed.
Our aim is to obtain quantitative normal approximation bounds
throughout the subcritical regime in both Kolmogorov and
$1$-Wasserstein distances.

\subsection{Model and notation}\label{subsec:intro-model}

Let $d\geq1$ be an integer.
At each site $x\in\mathbb Z^d$, the spin $\sigma_x$ takes values
in $\{-1,+1\}$. In particular, $\sigma_0$ denotes the spin at the origin
$0=(0,\ldots,0)$.
Let $J:\mathbb Z^d\to[0,\infty)$ satisfy
\begin{equation}
\label{eq:intro-summable-J}
J(0)=0,\qquad J(x)=J(-x),\qquad
\|J\|_1:=\sum_{x\in\mathbb Z^d}J(x)<\infty.
\end{equation}
The interaction strength between sites $x$ and $y$ is $J(x-y)$.
For a nonempty finite set $\Lambda\subset\mathbb Z^d$ and inverse
temperature $\beta\geq0$, the zero-field Gibbs measure with free
boundary conditions is
\begin{equation}
\label{eq:intro-long-range-measure}
\mu^{\mathrm f}_{\Lambda,\beta}(\sigma)
=\frac{1}{Z^{\mathrm f}_{\Lambda,\beta}}
\exp\!\left\{
\beta\sum_{x,y\in\Lambda}J(x-y)\sigma_x\sigma_y
\right\},
\qquad \sigma\in\{-1,+1\}^{\Lambda}.
\end{equation}
Here $\{-1,+1\}^{\Lambda}$ denotes the set of all spin
configurations on $\Lambda$.
A configuration $\sigma=(\sigma_x)_{x\in\Lambda}$ assigns
a spin $\sigma_x\in\{-1,+1\}$ to each site $x\in\Lambda$.
The quantity $Z^{\mathrm f}_{\Lambda,\beta}$ is the
normalizing constant.
By convention, the sum in \eqref{eq:intro-long-range-measure}
runs over unordered pairs of distinct sites in $\Lambda$,
with each pair counted once.
Free boundary conditions mean that interactions between $\Lambda$
and its complement are omitted.
The dependence of the measure on the fixed interaction $J$ is
suppressed in the notation.

Let $B_R=[-R,R]^d\cap\mathbb Z^d$.
For every finite $\Lambda\subset\mathbb Z^d$ and
$\eta\in\{-1,+1\}^{\Lambda}$, 
we denote by $\mu^\infty_\beta$ the infinite-volume free-boundary
Gibbs measure, defined by
\begin{equation}
\label{eq:intro-infinite-volume}
\mu^\infty_\beta(\eta)
=\lim_{R\to\infty}
\mu^{\mathrm f}_{B_R,\beta}(\eta)
\end{equation}
Here $\mu^{\mathrm f}_{B_R,\beta}(\eta)$ denotes the marginal
probability of the configuration $\eta$ on $\Lambda$ under
the finite-volume free-boundary measure on $B_R$.
The limit is taken with $\Lambda$ fixed as $B_R$ grows
to fill $\mathbb Z^d$.
Unlike the finite-volume free-boundary model on $\Lambda$, this
measure retains interactions between $\Lambda$ and its complement.

The susceptibility and its critical inverse temperature are defined by
\begin{equation}
\label{eq:free-susceptibility-critical-point}
\chi_J(\beta)
=\sum_{x\in\mathbb Z^d}
\mathbb E[\sigma_0\sigma_x],
\qquad
\beta_{\mathrm c}(J)
=\sup\{\beta\geq0:\chi_J(\beta)<\infty\}.
\end{equation}
Here the expectation is taken under the infinite-volume
free-boundary Gibbs measure $\mu^\infty_\beta$.
The ferromagnetic correlation inequalities in \citet{Griffiths1967a} imply that
$\chi_J(\beta)$ is nondecreasing in $\beta$.
Thus $\chi_J(\beta)<\infty$ for every
$0\leq\beta<\beta_{\mathrm c}(J)$, which we call the subcritical
regime.

For a nonempty finite set $\Lambda\subset\mathbb Z^d$, write
\begin{align}\label{def:total-spin-normalize}
  S_\Lambda=\sum_{x\in\Lambda}\sigma_x,\qquad
  W_\Lambda=\frac{S_\Lambda}{\sqrt{\operatorname{Var}(S_\Lambda)}}.  
\end{align}
Expectation and variance are taken under the Gibbs measure under
consideration.
Both $\mu^{\mathrm f}_{\Lambda,\beta}$ and $\mu^\infty_\beta$ are
invariant under reversing all spins, so $\mathbb E S_\Lambda=0$.
Let $G\sim\mathcal N(0,1)$ and let $\Phi$ be its distribution
function.
For real random variables $X$ and $Y$, the Kolmogorov distance is
\[
d_{\mathrm K}(X,Y)
=\sup_{t\in\mathbb R}
\big|\mathbb P(X\leq t)-\mathbb P(Y\leq t)\big|.
\]
If $\mathbb E|X|+\mathbb E|Y|<\infty$, their $1$-Wasserstein
distance is
\[
d_{\mathrm W}(X,Y)
=\sup_{\operatorname{Lip}(f)\leq1}
\big|\mathbb E f(X)-\mathbb E f(Y)\big|,
\]
where the supremum is over all functions $f:\mathbb R\to\mathbb R$
satisfying $|f(u)-f(v)|\leq|u-v|$ for all $u,v\in\mathbb R$.
We use the term Berry--Esseen bound for a bound on
$d_{\mathrm K}(W_\Lambda,G)$.

A basic example, used below to compare the assumptions of existing
results, is the power-law interaction
\begin{equation}
\label{eq:intro-power-law}
J(x)=a\|x\|_2^{-d-\alpha}\quad(x\neq0),
\qquad a>0,\quad\alpha>0.
\end{equation}
Here $\|\cdot\|_2$ denotes the Euclidean norm on $\mathbb R^d$. This interaction satisfies \eqref{eq:intro-summable-J} for every
$\alpha>0$.

\subsection{Previous results}\label{subsec:intro-previous-results}

At the qualitative level, the result of \citet[Theorem~2]{Newman1980} implies that,
for the zero-field long-range ferromagnetic Ising model with interaction $J$ satisfying \eqref{eq:intro-summable-J}, the central limit theorem holds for every fixed $0\leq\beta<\beta_{\mathrm c}(J)$:
\begin{equation}
\label{eq:intro-newman-clt}
W_{B_R}\xrightarrow[R\to\infty]{\mathrm d}G
\qquad\text{under }\mu^\infty_\beta.
\end{equation}

The quantitative question is how fast this convergence occurs
and under which assumptions a rate can be obtained.

\paragraph{Berry--Esseen bounds.}

For the zero-field long-range ferromagnetic Ising model with
$J\not\equiv0$ satisfying \eqref{eq:intro-summable-J},
the cluster-expansion estimates of \citet{NguyenFernandez2024},
combined with \citet[Theorem~2.4]{DoringJansenSchubert2022},
yield a Berry--Esseen bound throughout the explicit
sufficient interval
\begin{equation}
\label{eq:intro-cluster-temperature-simple}
0\leq\beta<
\frac{\log(1+1/(4e))}{\|J\|_1}.
\end{equation}
Specifically, for every fixed $\beta$ in this interval,
there exists $C_{\beta,J}<\infty$ such that
\[
d_{\mathrm K}(W_\Lambda,G)
\leq C_{\beta,J}|\Lambda|^{-1/2}
\]
for every nonempty finite $\Lambda\subset\mathbb Z^d$,
under both $\mu^{\mathrm f}_{\Lambda,\beta}$ and $\mu^\infty_\beta$.
This explicit sufficient interval is strictly contained in
the subcritical regime. Indeed, the standard susceptibility
estimate gives $\beta_{\mathrm c}(J)\geq\|J\|_1^{-1}$
\citep[Section~2.5]{DuminilCopinTassion2016}, and hence
\[
\frac{\log(1+1/(4e))}{\|J\|_1}
<
\frac{1}{\|J\|_1}
\leq\beta_{\mathrm c}(J).
\]

For the power-law interaction \eqref{eq:intro-power-law},
specializing the general result of
\citet[Theorem~4]{SanchezSegoviaSchneiderAlhambra2025}
to the zero-field long-range ferromagnetic Ising model with interaction $J$ satisfying \eqref{eq:intro-power-law} yields
\begin{align*}
 d_{\mathrm K}(W_\Lambda,G)\leq C_{\beta,J}|\Lambda|^{-1/2} 
\end{align*}
under $\mu^{\mathrm f}_{\Lambda,\beta}$ for every fixed
$\beta$ in the following volume-independent sufficient interval:
\begin{equation}
\label{eq:intro-power-law-temperature}
0\leq\beta<\frac{1}{32e\,g_Ju_{d,\alpha}},
\end{equation}
where
\begin{align*}
g_J&:=\max\left\{
\|J\|_1,\,
\sup_{x\neq0}(1+\|x\|_1)^{d+\alpha}J(x)
\right\},\\
u_{d,\alpha}&:=2^{d+\alpha}
\sum_{x\in\mathbb Z^d}(1+\|x\|_1)^{-d-\alpha}.
\end{align*}
Here $\|x\|_1:=\sum_{i=1}^d|x_i|$ for
$x=(x_1,\ldots,x_d)\in\mathbb R^d$.
This explicit sufficient interval is strictly contained in
the subcritical regime. Indeed, $g_J\geq\|J\|_1$ and
$u_{d,\alpha}\geq1$, while the standard susceptibility estimate
gives $\beta_{\mathrm c}(J)\geq\|J\|_1^{-1}$
\citep[Section~2.5]{DuminilCopinTassion2016}. Hence
\[
\frac{1}{32e\,g_Ju_{d,\alpha}}
\leq
\frac{1}{32e\,\|J\|_1}
<
\frac{1}{\|J\|_1}
\leq
\beta_{\mathrm c}(J).
\]

\paragraph{1-Wasserstein bounds.}
For general interactions $J\not\equiv0$ satisfying
\eqref{eq:intro-summable-J}, the cluster-expansion estimates
of \citet{NguyenFernandez2024}, combined with
\citet[Theorem~1]{Zaitsev2025Transport}, yield
\[
d_{\mathrm W}(W_\Lambda,G)
\leq C_{\beta,J}|\Lambda|^{-1/2}
\]
for every nonempty finite $\Lambda\subset\mathbb Z^d$
and every fixed $\beta$ in the sufficient interval
\eqref{eq:intro-cluster-temperature-simple},
under both $\mu^{\mathrm f}_{\Lambda,\beta}$ and $\mu^\infty_\beta$.

For the power-law interaction \eqref{eq:intro-power-law},
the same rate is available throughout the subcritical
regime under stronger decay assumptions.
Combining the correlation estimates of \citet{Aoun2021} with
the resulting finite-order mixing estimates for Ising spins and
\citet[Theorem~2.3(b)]{LiuAustern2026Mixing} gives
\begin{equation}
\label{eq:intro-previous-wasserstein}
d_{\mathrm W}(W_\Lambda,G)\leq C |\Lambda|^{-1/2},
\qquad \alpha>2d,\qquad
0<\beta<\beta_{\mathrm c}(J),
\end{equation}
under both $\mu^{\mathrm f}_{\Lambda,\beta}$ and $\mu^\infty_\beta$.
In dimension one, the same correlation estimates and
\citet[Theorem~3.1(b)]{DedeckerRio2008} also yield this rate
for interval sums under $\mu^\infty_\beta$ when $\alpha>2$.
The decay assumption $\alpha>2d$ used in these applications
is stronger than the summability condition $\alpha>0$
for the interaction \eqref{eq:intro-power-law}.

In summary, the results reviewed above do not establish
either Berry--Esseen or $1$-Wasserstein bounds of order
$|\Lambda|^{-1/2}$ for the standardized total spin of the
zero-field long-range ferromagnetic Ising model for every
interaction $J$ satisfying \eqref{eq:intro-summable-J}
and every fixed $0\leq\beta<\beta_{\mathrm c}(J)$.

In this paper, we fill this gap by establishing both Berry--Esseen and
$1$-Wasserstein bounds of order $|\Lambda|^{-1/2}$ for every
interaction $J$ satisfying \eqref{eq:intro-summable-J}
and every fixed $0\leq\beta<\beta_{\mathrm c}(J)$.

\subsection{Our contributions}\label{subsec:intro-contributions}

For every interaction $J$ satisfying
\eqref{eq:intro-summable-J} and for every fixed $0\leq\beta<\beta_{\mathrm c}(J)$, we prove
\[
d_{\mathrm K}(W_\Lambda,G)
\leq C_{\beta,J}|\Lambda|^{-1/2},
\qquad
d_{\mathrm W}(W_\Lambda,G)
\leq C_{\beta,J}|\Lambda|^{-1/2}
\]
for every nonempty finite $\Lambda\subset\mathbb Z^d$, under
both $\mu^{\mathrm f}_{\Lambda,\beta}$ and $\mu^\infty_\beta$,
see \cref{thm:free-long-range-new,thm:infinite-long-range-new}.

For the fully periodized interaction on $(\mathbb Z/L\mathbb Z)^d$ (see, \eqref{eq:periodized-J-new}),
we also obtain bounds of order $L^{-d/2}$ in both Kolmogorov and 1-Wasserstein distances for the standardized total spin over the whole torus,
throughout the same subcritical regime;
see \cref{thm:periodic-long-range-new}.
In particular, all three results apply to
\eqref{eq:intro-power-law} for every $a>0$ and $\alpha>0$.

We also give the Berry--Esseen and 1-Wasserstein bounds of order $|\Lambda|^{-1/2}$ for the zero-field nearest-neighbor Ising model with unit interaction $J_{\mathrm{nn}}$ for every fixed $|\beta|<\beta_{\mathrm c}(J_{\mathrm{nn}})$ 
(this covers the ferromagnetic case $\beta>0$, the
antiferromagnetic case $\beta<0$, and the independent case $\beta=0$)
, under free boundary conditions.

Our proof uses the Edwards--Sokal coupling
\citep{EdwardsSokal1988} to represent the total spin as a sum
of conditionally independent cluster signs weighted by cluster
sizes.
A fourth-cumulant identity controls fluctuations of the
conditional variance, while subcritical cluster-size moment
estimates provide the bounds needed for normal approximation.
This approach avoids additional assumptions on spatial decay.
For the periodized model, the proof further requires a
torus-to-lattice comparison and cluster-size estimates that
are uniform in the torus size.

The remaining of the paper is organized as follows.
\Cref{sec:applications-new} states the main results.
\Cref{sec:proof-main-results} develops the finite-graph normal
approximation bounds.
\Cref{sec:proof-applications} contains the proofs of the main
theorems.

\section{Main results}\label{sec:applications-new}

All results in this section concern zero external field.
We give Kolmogorov and $1$-Wasserstein bounds for the standardized total spin $W_{\Lambda}$ defined in \eqref{def:total-spin-normalize}.
We first consider ferromagnetic models with summable interactions.
And then we consider antiferromagnetic and ferromagnetic nearest-neighbor models.

\subsection{Long-range Ising model with free boundary conditions}
\label{subsec:main-ferromagnetic}

Let $d\geq1$ be an integer, and let $J$ satisfy
\eqref{eq:intro-summable-J}.
Write $\|J\|_1:=\sum_{x\in\mathbb Z^d}J(x)$.
We use the susceptibility threshold $\beta_{\mathrm c}(J)$
defined in \eqref{eq:free-susceptibility-critical-point}.

We consider the free-boundary Gibbs measure
$\mu^{\mathrm f}_{\Lambda,\beta}$ defined in
\eqref{eq:intro-long-range-measure}.
For a nonempty finite set $\Lambda\subset\mathbb Z^d$, write
\[
s_\Lambda:=\sqrt{\operatorname{Var}(S_\Lambda)},
\qquad
W_\Lambda=\frac{S_\Lambda}{s_\Lambda}.
\]
The variance is taken under $\mu^{\mathrm f}_{\Lambda,\beta}$.

\begin{theorem}
\label{thm:free-long-range-new}
Fix $0\leq\beta<\beta_{\mathrm c}(J)$.
There exists $C_{\beta,J}<\infty$ such that, for every nonempty
finite set $\Lambda\subset\mathbb Z^d$,
\begin{align}
d_{\mathrm K}(W_\Lambda,G)
&\leq C_{\beta,J}|\Lambda|^{-1/2},
\label{eq:free-long-range-BE-new}\\
d_{\mathrm W}(W_\Lambda,G)
&\leq C_{\beta,J}|\Lambda|^{-1/2}.
\label{eq:free-long-range-W-new}
\end{align}
Both bounds hold under $\mu^{\mathrm f}_{\Lambda,\beta}$.
The constant depends only on $\beta$ and $J$.
\end{theorem}

\begin{remark}
\label{rem:kolmogorov-rate}
For every fixed $J$ satisfying \eqref{eq:intro-summable-J}
and every fixed $0\leq\beta<\beta_{\mathrm c}(J)$, the order
$|\Lambda|^{-1/2}$ of the Kolmogorov bound in
\eqref{eq:free-long-range-BE-new} is optimal.
For every nonempty finite set $\Lambda\subset\mathbb Z^d$, we have
\begin{equation}
\label{eq:free-K-lower-bound}
d_{\mathrm K}(W_\Lambda,G)
\geq
\frac{3}{8\bigl(2\sqrt{\chi_J(\beta)}+1\bigr)}\,|\Lambda|^{-1/2}.
\end{equation}
This bound holds under $\mu^{\mathrm f}_{\Lambda,\beta}$.
We prove \eqref{eq:free-K-lower-bound} in
\cref{sec:free-long-range}.
\end{remark}

We also consider the infinite-volume free-boundary Gibbs measure
$\mu^\infty_\beta$, defined by \eqref{eq:intro-infinite-volume}.
In this case, the variance is taken under $\mu^\infty_\beta$. And we can also give the Berry--Esseen and $1$-Wasserstein bounds as follows.

\begin{theorem}
\label{thm:infinite-long-range-new}
Fix $0\leq\beta<\beta_{\mathrm c}(J)$.
There exists $C_{\beta,J}^{\infty}<\infty$ such that, for every
nonempty finite set $\Lambda\subset\mathbb Z^d$,
\begin{align}
d_{\mathrm K}(W_\Lambda,G)
&\leq C_{\beta,J}^{\infty}|\Lambda|^{-1/2},
\label{eq:infinite-long-range-BE-new}\\
d_{\mathrm W}(W_\Lambda,G)
&\leq C_{\beta,J}^{\infty}|\Lambda|^{-1/2}.
\label{eq:infinite-long-range-W-new}
\end{align}
Both bounds hold under $\mu^\infty_\beta$.
The constant depends only on $\beta$ and $J$.
\end{theorem}

\begin{remark}[Power-law interactions]
\label{rem:power-law-applications}
Let $a>0$ and $\alpha>0$. Consider the power-law interaction
\[
J(0)=0,
\qquad
J(x)=a\|x\|_2^{-d-\alpha},
\quad x\in\mathbb Z^d\setminus\{0\},
\]
as in \eqref{eq:intro-power-law}.
This interaction satisfies \eqref{eq:intro-summable-J}
for every $\alpha>0$.

Fix $0\leq\beta<\beta_{\mathrm c}(J)$.
Let $\Lambda\subset\mathbb Z^d$ be any nonempty finite set.
For the finite-volume model with free boundary conditions,
\cref{thm:free-long-range-new} gives Kolmogorov and
$1$-Wasserstein bounds of order $|\Lambda|^{-1/2}$
for the standardized total spin.
The same rates hold under the infinite-volume free-boundary
Gibbs measure by \cref{thm:infinite-long-range-new}.
\end{remark}

\begin{remark}[Boxes]
\label{rem:box-rates}
Let $J$ satisfy \eqref{eq:intro-summable-J}, and fix
$0\leq\beta<\beta_{\mathrm c}(J)$.
For an integer $n\geq1$, set
\[
Q_n:=\{1,\ldots,n\}^d.
\]
Since $|Q_n|=n^d$, the Kolmogorov and $1$-Wasserstein bounds
in \cref{thm:free-long-range-new} have order $n^{-d/2}$
under $\mu^{\mathrm f}_{Q_n,\beta}$.
Under $\mu^\infty_\beta$, \cref{thm:infinite-long-range-new}
gives order $n^{-d/2}$ in both distances.
\end{remark}

\subsection{Fully periodized models on tori}\label{subsec:main-periodized}
For an integer $L\geq1$, define
\[
\mathbb T_L^d:=(\mathbb Z/L\mathbb Z)^d,
\]
thus $|\mathbb T_L^d|=L^d$. For $x\in\mathbb Z^d$, define
\[
\bar x:=x+L\mathbb Z^d
=\{x+Lz:z\in\mathbb Z^d\}.
\]
The set $\bar x$ is a site of $\mathbb T_L^d$.
The lattice point $x$ is called a representative of $\bar x$.
Two lattice points $x,y$ represent the same torus site if and
only if $x_i-y_i$ is an integer multiple of $L$ for every
$i=1,\ldots,d$.
For distinct $\bar x,\bar y\in\mathbb T_L^d$, choose
representatives $x,y\in\mathbb Z^d$ and set
\begin{equation}
\label{eq:periodized-J-new}
J_L(\bar x-\bar y)
:=\sum_{u\in\mathbb Z^d}J(x-y+Lu).
\end{equation}
This sums the interaction strengths over all lattice displacements
that give the same torus displacement.
The series converges and does not depend on the chosen representatives.

The zero-field Gibbs measure on the torus is
\begin{equation}
\label{eq:periodic-measure-new}
\mu^{\mathrm{per}}_{L,\beta}(\sigma)
=\frac{1}{Z^{\mathrm{per}}_{L,\beta}}
\exp\!\left\{
\beta\sum_{\bar x,\bar y\in\mathbb T_L^d}
J_L(\bar x-\bar y)\sigma_{\bar x}\sigma_{\bar y}
\right\},
\qquad
\sigma\in\{-1,+1\}^{\mathbb T_L^d}.
\end{equation}
Here $\beta\geq0$, and $Z^{\mathrm{per}}_{L,\beta}$ is the
normalizing constant.
By convention, the sum in \eqref{eq:periodic-measure-new}
runs over unordered pairs of distinct sites of $\mathbb T_L^d$,
with each pair counted once.
Define
\[
S_L:=\sum_{\bar x\in\mathbb T_L^d}\sigma_{\bar x},
\qquad
s_L:=\sqrt{\operatorname{Var}(S_L)},
\qquad
W_L:=\frac{S_L}{s_L}.
\]
The variance is taken under $\mu^{\mathrm{per}}_{L,\beta}$.
This measure is invariant under reversing all spins, so
$\mathbb E S_L=0$.
In the following theorem, $\beta_{\mathrm c}(J)$ denotes the
susceptibility threshold of the infinite-volume model with
interaction $J$, defined in \eqref{eq:free-susceptibility-critical-point}.

\begin{theorem}
\label{thm:periodic-long-range-new}
Fix $0\leq\beta<\beta_{\mathrm c}(J)$.
There exists $C_{\beta,J}^{\mathrm{per}}<\infty$ such that,
for every integer $L\geq1$,
\begin{align}
d_{\mathrm K}(W_L,G)
&\leq C_{\beta,J}^{\mathrm{per}}L^{-d/2},
\label{eq:periodic-BE-new}\\
d_{\mathrm W}(W_L,G)
&\leq C_{\beta,J}^{\mathrm{per}}L^{-d/2}.
\label{eq:periodic-W-new}
\end{align}
Both bounds hold under $\mu^{\mathrm{per}}_{L,\beta}$.
The constant depends only on $\beta$ and $J$.
\end{theorem}

\subsection{Nearest-neighbor ferromagnetic and antiferromagnetic models}
\label{subsec:main-classical}

We now consider the classical nearest-neighbor Ising model.
For a nonempty finite set $\Lambda\subset\mathbb Z^d$, its Gibbs measure
with free boundary conditions is
\begin{equation}
\label{eq:intro-nn-measure}
 \mu^{\mathrm f}_{\Lambda,\beta}(\sigma)
 =\frac{1}{Z^{\mathrm f}_{\Lambda,\beta}}
 \exp\!\left\{
   \beta\sum_{\substack{x,y\in\Lambda\\ \|x-y\|_1=1}}
   \sigma_x\sigma_y
 \right\},
 \qquad \sigma\in\{-1,+1\}^{\Lambda}.
\end{equation}
By convention, the sum in \eqref{eq:intro-nn-measure}
runs over unordered nearest-neighbor pairs in $\Lambda$,
with each pair counted once.
Free boundary conditions mean that interactions between $\Lambda$ and
its complement are omitted. And the infinite-volume free-boundary Gibbs measure is also defined by \eqref{eq:intro-infinite-volume}.

We allow both signs of $\beta$: positive values give ferromagnetic
interactions, and negative values give antiferromagnetic interactions.
This measure is invariant under reversing all spins, so $\mathbb E S_\Lambda=0$ for both signs of $\beta$.
Write
\[
J_{\mathrm{nn}}(x):=\mathbf1_{\{\|x\|_1=1\}},
\qquad
\beta_{\mathrm c}:=\beta_{\mathrm c}(J_{\mathrm{nn}}),
\]
here $\beta_{\mathrm c}$ is the nearest-neighbor ferromagnetic
threshold used below.
We first consider the free-boundary measure $\mu^{\mathrm f}_{\Lambda,\beta}$.

\begin{theorem}
\label{thm:classical-ising}
Fix $-\beta_{\mathrm c}<\beta<\beta_{\mathrm c}$.
There exists $C_{\beta,d}<\infty$ such that, for every nonempty
finite set $\Lambda\subset\mathbb Z^d$, we have
$s_\Lambda^2>0$ and
\begin{align}
d_{\mathrm K}(W_\Lambda,G)
&\leq C_{\beta,d}|\Lambda|^{-1/2},
\label{eq:classical-ising-BE}\\
d_{\mathrm W}(W_\Lambda,G)
&\leq C_{\beta,d}|\Lambda|^{-1/2}.
\label{eq:classical-ising-W-bound}
\end{align}
Both bounds hold under $\mu^{\mathrm f}_{\Lambda,\beta}$.
The constant depends only on $\beta$ and $d$.
\end{theorem}

\begin{remark}
  For $0\leq\beta<\beta_{\mathrm c}$, both bounds in
\cref{thm:classical-ising} follow directly from
\cref{thm:free-long-range-new} with $J=J_{\mathrm{nn}}$.
\end{remark}

We next consider the zero-field nearest-neighbor model in
infinite volume $\mu^\infty_\beta$.

\begin{theorem}
\label{thm:classical-ising-infinite-wasserstein}
Fix $-\beta_{\mathrm c}<\beta<\beta_{\mathrm c}$.
There exists $C_{\beta,d}^{\infty}<\infty$ such that, for every
nonempty finite set $\Lambda\subset\mathbb Z^d$,
we have $s_\Lambda^2>0$ and
\begin{align}
d_{\mathrm K}(W_\Lambda,G)
&\leq C_{\beta,d}^{\infty}|\Lambda|^{-1/2},
\label{eq:classical-infinite-K-bound}\\
d_{\mathrm W}(W_\Lambda,G)
&\leq C_{\beta,d}^{\infty}|\Lambda|^{-1/2}.
\label{eq:classical-infinite-W-bound}
\end{align}
Both bounds hold under $\mu^\infty_\beta$.
The constant depends only on $\beta$ and $d$.
\end{theorem}

\begin{remark}
  For $0\leq\beta<\beta_{\mathrm c}$, both bounds in
\cref{thm:classical-ising-infinite-wasserstein}
follow directly from \cref{thm:infinite-long-range-new}
with $J=J_{\mathrm{nn}}$.
\end{remark}

\begin{remark}
\label{rem:classical-infinite-wasserstein-comparison}
For the zero-field nearest-neighbor ferromagnetic Ising model,
\citet[Theorem~2.1 and Corollary~2.7]{GoldsteinWiroonsri2018}
obtained a bound of order $|Q_n|^{-1/(2d+2)}$ in
$d_{\mathrm W}$ under $\mu^\infty_\beta$, for every fixed
$0<\beta<\beta_{\mathrm c}$.
For $\Lambda=Q_n$, the bound in
\eqref{eq:classical-infinite-W-bound} improves this rate to
$|Q_n|^{-1/2}$.
\end{remark}

\section{Finite-graph normal approximation for Ising linear statistics}
\label{sec:proof-main-results}

In this section, we use the Edwards--Sokal coupling to prove
normal approximation bounds for zero-field ferromagnetic Ising
models on finite graphs.
The coupling reduces the problem to normal approximation for
sums of conditionally independent random variables and to moment
estimates for random clusters.

\subsection{Finite-graph Ising model and Edwards--Sokal coupling}
\label{subsec:finite-graph-setup}

Let $\mathsf G=(\mathsf V,\mathsf E)$ be a finite simple graph,
where $\mathsf V$ is the vertex set and $\mathsf E$ is the edge set.
We assume that $\mathsf V$ is nonempty and $|\mathsf V|$ is the number of vertices.
For each edge $e=\{x,y\}\in\mathsf E$, let
$K_e=K_{xy}=K_{yx}\in[0,\infty)$.
The zero-field Ising measure is
\begin{equation}
\label{eq:finite-ising-new}
\mu_{\mathsf V,K}(\sigma)
=\frac{1}{Z^{\mathrm{Is}}_{\mathsf V,K}}
\exp\!\left\{ 
\sum_{\{x,y\}\in\mathsf E}K_{xy}\sigma_x\sigma_y
\right\},
\qquad \sigma\in\{-1,+1\}^{\mathsf V}.
\end{equation}
Here $Z^{\mathrm{Is}}_{\mathsf V,K}$ is the normalizing constant.
The inverse temperature is included in $K_e$.
For the free-boundary ferromagnet in
\eqref{eq:intro-long-range-measure}, $K_{xy}=\beta J(x-y)$.
For the periodized model in \eqref{eq:periodic-measure-new},
$K_{\bar x\bar y}=\beta J_L(\bar x-\bar y)$.

For each edge, set
\[
p_e:=1-e^{-2K_e}\in[0,1).
\]
An edge configuration is a vector
$\omega=(\omega_e)_{e\in\mathsf E}\in\{0,1\}^{\mathsf E}$.
An edge is open when $\omega_e=1$ and closed when $\omega_e=0$.
The connected components of the graph formed by the open edges
are called open clusters.
Let $\mathcal C(\omega)$ be their collection and
$k(\omega):=|\mathcal C(\omega)|$ be their number.
Isolated vertices are counted as singleton clusters.
Write $C(x)$ for the cluster containing $x$, and write
$x\leftrightarrow y$ when $x$ and $y$ belong to the same cluster.

The associated $q=2$ random-cluster measure, also called the
FK measure, is
\begin{equation}
\label{eq:finite-random-cluster-measure}
\phi_{\mathsf V,K,2}(\omega)
=\frac{1}{Z^{\mathrm{FK}}_{\mathsf V,K,2}}
2^{k(\omega)}
\prod_{e\in\mathsf E}
p_e^{\omega_e}(1-p_e)^{1-\omega_e}.
\end{equation}
Here $Z^{\mathrm{FK}}_{\mathsf V,K,2}$ is the normalizing constant;
see \citet{Grimmett2006}.
We use $0^0=1$ in this product.
Thus an edge with $p_e=0$ is closed almost surely.

The Edwards--Sokal coupling \citep{EdwardsSokal1988} is the
following joint measure on spins and edges:
\begin{equation}
\label{eq:finite-ES-joint-law}
\mathsf P^{\mathrm{ES}}_{\mathsf V,K}(\sigma,\omega)
=\frac{1}{Z^{\mathrm{ES}}_{\mathsf V,K}}
\prod_{e=\{x,y\}\in\mathsf E}
\bigl[(1-p_e)\delta(\omega_e,0)
+p_e\delta(\omega_e,1)\delta(\sigma_x,\sigma_y)\bigr].
\end{equation}
Here $\delta(u,v)=1$ if $u=v$ and $0$ otherwise, and
$Z^{\mathrm{ES}}_{\mathsf V,K}$ is the normalizing constant.
This is the form reviewed in
\citet[Section~2.2]{DingLiuXia2024}, with a separate parameter
$p_e$ for each edge.
A closed edge contributes $1-p_e$ to the product.
An open edge contributes $p_e$ if its endpoint spins agree
and zero otherwise.

For the spin and edge variables, $\mathbb P$, $\mathbb E$ and
$\operatorname{Var}$ denote probability, expectation and variance
under $\mathsf P^{\mathrm{ES}}_{\mathsf V,K}$ throughout this section.

\begin{proposition}[Edwards--Sokal coupling]
\label{prop:ES-identities}
The joint measure in \eqref{eq:finite-ES-joint-law} has the
following properties.
\begin{enumerate}
\item[(i)] The spin configuration has distribution
$\mu_{\mathsf V,K}$.
\item[(ii)] The edge configuration has distribution
$\phi_{\mathsf V,K,2}$.
\item[(iii)] Conditional on the spin configuration $\sigma$,
the edges are independent. An edge $e=\{x,y\}$ is open with
probability $p_e\mathbf1_{\{\sigma_x=\sigma_y\}}$.
\item[(iv)] Conditional on any edge configuration $\omega$
with $\phi_{\mathsf V,K,2}(\omega)>0$, the spins are constant
on each open cluster. Their common values
$(\varepsilon_C)_{C\in\mathcal C(\omega)}$ are independent,
and each equals $-1$ or $+1$ with probability $1/2$.
\item[(v)] For every $x,y\in\mathsf V$,
\begin{equation}
\label{eq:two-point-connectivity}
\mathbb E[\sigma_x\sigma_y]=\mathbb P(x\leftrightarrow y).
\end{equation}
\end{enumerate}
\end{proposition}

\begin{proof}[Proof of \cref{prop:ES-identities}]
\emph{(i)} Fix a spin configuration $\sigma$.
For an edge $e=\{x,y\}$, summing its factor in
\eqref{eq:finite-ES-joint-law} over $\omega_e=0,1$ gives
\[
\begin{aligned}
(1-p_e)+p_e\delta(\sigma_x,\sigma_y)
&=\begin{cases}
1,&\sigma_x=\sigma_y,\\
e^{-2K_e},&\sigma_x\ne\sigma_y
\end{cases}\\
&=e^{-K_e}e^{K_e\sigma_x\sigma_y}.
\end{aligned}
\]
Each edge variable appears in only one factor.
Summing over all edge configurations therefore gives
\[
\sum_{\omega\in\{0,1\}^{\mathsf E}}
\mathsf P^{\mathrm{ES}}_{\mathsf V,K}(\sigma,\omega)
=\frac{e^{-\sum_{e\in\mathsf E}K_e}}
{Z^{\mathrm{ES}}_{\mathsf V,K}}
\exp\!\left\{
\sum_{\{x,y\}\in\mathsf E}K_{xy}\sigma_x\sigma_y
\right\}.
\]
The sum of these probabilities over all $\sigma$ is one.
By the definition of $Z^{\mathrm{Is}}_{\mathsf V,K}$, this implies
\[
Z^{\mathrm{ES}}_{\mathsf V,K}
=e^{-\sum_{e\in\mathsf E}K_e}Z^{\mathrm{Is}}_{\mathsf V,K}.
\]
Substituting this identity into the preceding expression gives
the spin marginal $\mu_{\mathsf V,K}$, proving (i).

\emph{(ii)} Fix an edge configuration $\omega$.
Call a spin configuration compatible with $\omega$ if it is
constant on each open cluster.
Any other spin configuration has joint probability zero with
$\omega$, since some open edge has unequal endpoint spins.
Each cluster has two possible spin values, so there are
$2^{k(\omega)}$ compatible configurations.

For every compatible configuration, an open edge contributes
$p_e$ and a closed edge contributes $1-p_e$.
These factors do not depend on the signs chosen for the clusters.
Thus every compatible configuration has the same joint
probability with $\omega$, namely
\[
b(\omega):=
\frac{1}{Z^{\mathrm{ES}}_{\mathsf V,K}}
\prod_{e\in\mathsf E}
p_e^{\omega_e}(1-p_e)^{1-\omega_e}.
\]
Summing over the spin configurations gives
\[
\sum_{\sigma\in\{-1,+1\}^{\mathsf V}}
\mathsf P^{\mathrm{ES}}_{\mathsf V,K}(\sigma,\omega)
=2^{k(\omega)}b(\omega)
=\frac{2^{k(\omega)}}{Z^{\mathrm{ES}}_{\mathsf V,K}}
\prod_{e\in\mathsf E}
p_e^{\omega_e}(1-p_e)^{1-\omega_e}.
\]
The numerator is the FK weight.
Since the probabilities sum to one over $\omega$,
$Z^{\mathrm{ES}}_{\mathsf V,K}=Z^{\mathrm{FK}}_{\mathsf V,K,2}$.
Hence the edge marginal is $\phi_{\mathsf V,K,2}$, proving (ii).

\emph{(iii)} Fix $\sigma$ again.
For each edge $e=\{x,y\}$, write its closed and open weights as
\[
w_e(0):=1-p_e,
\qquad
w_e(1):=p_e\delta(\sigma_x,\sigma_y).
\]
Since $p_e<1$, we have
$w_e(0)+w_e(1)\geq1-p_e>0$.
For an edge configuration $t\in\{0,1\}^{\mathsf E}$, the joint
probability of $(\sigma,t)$ is proportional to
$\prod_e w_e(t_e)$.
Expanding the product on the right gives the identity
\[
\sum_{\eta\in\{0,1\}^{\mathsf E}}
\prod_{e\in\mathsf E}w_e(\eta_e)
=\prod_{e\in\mathsf E}\bigl(w_e(0)+w_e(1)\bigr).
\]
Dividing the weight of $t$ by the sum of the weights of all
edge configurations, we get
\[
\begin{aligned}
\mathbb P(\omega=t\mid\sigma)
&=\frac{\prod_{e\in\mathsf E}w_e(t_e)}
{\sum_{\eta\in\{0,1\}^{\mathsf E}}
 \prod_{e\in\mathsf E}w_e(\eta_e)}\\
&=\prod_{e\in\mathsf E}
\frac{w_e(t_e)}{w_e(0)+w_e(1)}.
\end{aligned}
\]
For each edge, the last factor sums to one over $t_e=0,1$.
The conditional distribution is therefore a product of
probability distributions on the individual edges.
This proves that the edges are independent given $\sigma$.
Their open probabilities are
\[
\mathbb P(\omega_e=1\mid\sigma)
=\frac{w_e(1)}{w_e(0)+w_e(1)}
=\begin{cases}
p_e,&\sigma_x=\sigma_y,\\
0,&\sigma_x\ne\sigma_y.
\end{cases}
\]
This is $p_e\mathbf1_{\{\sigma_x=\sigma_y\}}$, proving (iii).

\emph{(iv)} Fix $\omega$ with
$\phi_{\mathsf V,K,2}(\omega)>0$.
By the calculation in (ii), each compatible spin configuration
has joint probability $b(\omega)$ with $\omega$, and
$\phi_{\mathsf V,K,2}(\omega)=2^{k(\omega)}b(\omega)>0$.
Thus, for any compatible spin configuration $\tau$,
\[
\mathbb P(\sigma=\tau\mid\omega)
=\frac{\mathsf P^{\mathrm{ES}}_{\mathsf V,K}(\tau,\omega)}
{\phi_{\mathsf V,K,2}(\omega)}
=\frac{b(\omega)}{2^{k(\omega)}b(\omega)}
=2^{-k(\omega)}.
\]
Incompatible configurations have conditional probability zero.
Write the open clusters as $C_1,\ldots,C_k$, where $k=k(\omega)$.
For every choice $t_1,\ldots,t_k\in\{-1,+1\}$,
\[
\mathbb P\bigl(
\varepsilon_{C_1}=t_1,\ldots,\varepsilon_{C_k}=t_k
\mid\omega\bigr)
=2^{-k}=\prod_{i=1}^k\frac12.
\]
The cluster signs are therefore independent given $\omega$,
and each sign takes the two values with probability $1/2$.
This proves (iv).

\emph{(v)} Fix $x,y\in\mathsf V$ and condition on $\omega$.
If $x\leftrightarrow y$, the two spins are equal, so
$\sigma_x\sigma_y=1$.
Otherwise, they belong to different clusters.
By (iv), the two spins are conditionally independent and have
conditional mean zero, so
$\mathbb E(\sigma_x\sigma_y\mid\omega)=0$.
Consequently,
\[
\mathbb E(\sigma_x\sigma_y\mid\omega)
=\mathbf1_{\{x\leftrightarrow y\}}.
\]
Taking expectations gives
\[
\mathbb E[\sigma_x\sigma_y]
=\mathbb E\bigl[\mathbf1_{\{x\leftrightarrow y\}}\bigr]
=\mathbb P(x\leftrightarrow y),
\]
which proves (v).
\end{proof}

By (i) and (ii), expectations of spin functions are also
expectations under $\mu_{\mathsf V,K}$, and expectations of edge
functions are also expectations under $\phi_{\mathsf V,K,2}$.

\subsection{Weighted spin sums and conditional variance}
\label{subsec:weighted-spin-variance}

Let $a=(a_x)_{x\in\mathsf V}\in\mathbb R^{\mathsf V}$
be a fixed vector of weights. Define
\begin{equation}
\label{eq:weighted-statistic-scales}
T_a:=\sum_{x\in\mathsf V}a_x\sigma_x,
\qquad
s_a^2:=\operatorname{Var}(T_a),
\qquad
s_a:=\sqrt{s_a^2},
\qquad
A:=\max_{x\in\mathsf V}|a_x|.
\end{equation}
The variable $T_a$ is the weighted spin sum.
Taking $a_x=1$ at every vertex gives the total spin.
Spin-flip symmetry gives $\mathbb ET_a=0$.
The condition $a\ne0$ means that at least one coefficient $a_x$
is nonzero; other coefficients may be zero.
Under this condition, $s_a^2>0$.
Indeed, every spin configuration has positive probability,
and changing a spin with a nonzero coefficient changes $T_a$.

For each open cluster $C$, set
\[
a(C):=\sum_{x\in C}a_x,
\qquad
V_a:=\sum_{C\in\mathcal C(\omega)}a(C)^2.
\]
The cluster signs in \cref{prop:ES-identities} give, almost surely,
\begin{equation}
\label{eq:weighted-ES}
T_a=\sum_{C\in\mathcal C(\omega)}a(C)\varepsilon_C.
\end{equation}
Conditional on $\omega$, these summands are independent and
have mean zero. Hence
\begin{equation}
\label{eq:ES-conditional-moments}
\mathbb E(T_a\mid\omega)=0,
\qquad
\operatorname{Var}(T_a\mid\omega)=V_a,
\qquad
s_a^2=\mathbb EV_a.
\end{equation}
Thus $V_a$ is the conditional variance, and $s_a^2$ is its
expectation. Opposite weights in a cluster can cancel, so
$V_a$ may be zero even though $s_a^2>0$.

For every subset $H\subseteq\mathsf V$ and every $\ell\geq1$,
define
\[
R_\ell(H):=\sum_{C\in\mathcal C(\omega)}|C\cap H|^\ell,
\qquad
R_\ell:=R_\ell(\mathsf V).
\]
The clusters $C$ are those of the entire graph $\mathsf G$.
For the weighted sum $T_a$, take
\[
H:=\{x\in\mathsf V:a_x\ne0\}.
\]
Then $C\cap H$ consists of the vertices in $C$ with nonzero
weights.
Since $|a(C)|\leq A|C\cap H|$, we have
\begin{equation}
\label{eq:cluster-weight-bound}
\sum_{C\in\mathcal C(\omega)}|a(C)|^\ell
\leq A^\ell R_\ell(H)
\leq A^\ell R_\ell,
\qquad \ell\geq1.
\end{equation}
Counting each vertex of $H$ in its cluster gives
\begin{equation}
\label{eq:localized-cluster-sum-identity}
R_\ell(H)=\sum_{x\in H}|C(x)\cap H|^{\ell-1}.
\end{equation}
In particular,
\begin{equation}
\label{eq:cluster-sum-identity}
R_\ell=\sum_{x\in\mathsf V}|C(x)|^{\ell-1}.
\end{equation}
For the all-ones vector $\mathbf1=(1)_{x\in\mathsf V}$,
$a(C)=|C|$. Therefore
\begin{equation}
\label{eq:FG-basic}
V_{\mathbf1}=R_2\geq |\mathsf V|,
\qquad
s_{\mathbf1}^2=\mathbb ER_2.
\end{equation}
The inequality follows from
$\sum_C|C|^2\geq\sum_C|C|=|\mathsf V|$.

To prove the normal approximation bounds below, we use the
following lemma to control the fluctuations of $V_a$ around
its mean $s_a^2=\mathbb E V_a$.

\begin{lemma}[Fluctuation of the conditional variance]
\label{lem:conditional-variance-bound}
For every subset $H\subseteq\mathsf V$,
\begin{equation}
\label{eq:localized-magnetization-variance-fluctuation}
\operatorname{Var}(R_2(H))
\leq \frac23\mathbb ER_4(H).
\end{equation}
In particular,
\begin{equation}
\label{eq:magnetization-variance-fluctuation}
\operatorname{Var}(R_2)
\leq \frac23\mathbb ER_4.
\end{equation}
For every $a\in\mathbb R^{\mathsf V}$, with
$H=\{x\in\mathsf V:a_x\ne0\}$,
\begin{equation}
\label{eq:conditional-variance-bound}
\operatorname{Var}(V_a)
\leq \frac43 A^4\mathbb ER_4(H)
\leq \frac43 A^4\mathbb ER_4.
\end{equation}
\end{lemma}

\begin{proof}[Proof of \cref{lem:conditional-variance-bound}]
Since $\mathbb ET_a=0$, its fourth cumulant is
\[
\kappa_4(T_a):=
\mathbb E[T_a^4]-3\bigl(\mathbb E[T_a^2]\bigr)^2.
\]
Conditional on $\omega$, the cluster signs in
\eqref{eq:weighted-ES} are independent and symmetric. Hence
\[
\mathbb E(T_a^4\mid\omega)
=3V_a^2-2\sum_{C\in\mathcal C(\omega)}a(C)^4.
\]
Taking expectations and using $\mathbb E[T_a^2]=\mathbb EV_a$
yields
\begin{equation}
\label{eq:weighted-fourth-cumulant-identity}
\kappa_4(T_a)
=3\operatorname{Var}(V_a)
-2\mathbb E\sum_{C\in\mathcal C(\omega)}a(C)^4.
\end{equation}

For $x_1,x_2,x_3,x_4\in\mathsf V$, define the four-point cumulant
\[
\begin{aligned}
U_4(x_1,x_2,x_3,x_4)
:={}&\mathbb E[\sigma_{x_1}\sigma_{x_2}\sigma_{x_3}\sigma_{x_4}]
-\mathbb E[\sigma_{x_1}\sigma_{x_2}]
  \mathbb E[\sigma_{x_3}\sigma_{x_4}]\\
&-\mathbb E[\sigma_{x_1}\sigma_{x_3}]
  \mathbb E[\sigma_{x_2}\sigma_{x_4}]
-\mathbb E[\sigma_{x_1}\sigma_{x_4}]
  \mathbb E[\sigma_{x_2}\sigma_{x_3}].
\end{aligned}
\]
The zero-field four-point inequality of
\citet[p.~91, Remark~(i)]{Lebowitz1974} gives $U_4\leq0$
when the four vertices are distinct.
If a vertex is repeated, symmetry allows us to write the
arguments as $(x,x,y,z)$. Since $\sigma_x^2=1$,
\[
U_4(x,x,y,z)
=-2\mathbb E[\sigma_x\sigma_y]\mathbb E[\sigma_x\sigma_z]
\leq0,
\]
where the inequality follows from \eqref{eq:two-point-connectivity}.
Thus
\begin{equation}
\label{eq:U4-sign}
U_4(x_1,x_2,x_3,x_4)\leq0
\qquad\text{for all }x_1,x_2,x_3,x_4\in\mathsf V.
\end{equation}
Expanding the moments in the definition of $\kappa_4(T_a)$ gives
\begin{equation}
\label{eq:cumulant-expansion}
\kappa_4(T_a)
=\sum_{x_1,x_2,x_3,x_4\in\mathsf V}
a_{x_1}a_{x_2}a_{x_3}a_{x_4}U_4(x_1,x_2,x_3,x_4).
\end{equation}

Fix $H\subseteq\mathsf V$ and set
$c_x=\mathbf1_{\{x\in H\}}$ for $x\in\mathsf V$.
Then $c(C)=|C\cap H|$ and $V_c=R_2(H)$.
Since all $c_x$ are nonnegative, \eqref{eq:U4-sign} and
\eqref{eq:cumulant-expansion} give $\kappa_4(T_c)\leq0$.
Applying \eqref{eq:weighted-fourth-cumulant-identity} to $T_c$,
we obtain
\[
3\operatorname{Var}(R_2(H))
=\kappa_4(T_c)+2\mathbb ER_4(H)
\leq2\mathbb ER_4(H).
\]
This proves \eqref{eq:localized-magnetization-variance-fluctuation}.
Taking $H=\mathsf V$ gives
\eqref{eq:magnetization-variance-fluctuation}.

Now let $a\in\mathbb R^{\mathsf V}$ and take
$H=\{x\in\mathsf V:a_x\ne0\}$, with $c$ as above.
Using \eqref{eq:cumulant-expansion}, $|a_x|\leq A$, and
\eqref{eq:U4-sign}, we get
\begin{align*}
|\kappa_4(T_a)|
&\leq A^4\sum_{x_1,x_2,x_3,x_4\in H}
|U_4(x_1,x_2,x_3,x_4)|\\
&=-A^4\kappa_4(T_c)\\
&=A^4\bigl(2\mathbb ER_4(H)
-3\operatorname{Var}(R_2(H))\bigr)\\
&\leq2A^4\mathbb ER_4(H).
\end{align*}
Also, \eqref{eq:cluster-weight-bound} gives
$\sum_Ca(C)^4\leq A^4R_4(H)$.
Therefore \eqref{eq:weighted-fourth-cumulant-identity} implies
\begin{align*}
3\operatorname{Var}(V_a)
&=\kappa_4(T_a)+2\mathbb E\sum_Ca(C)^4\\
&\leq|\kappa_4(T_a)|+2A^4\mathbb ER_4(H)\\
&\leq4A^4\mathbb ER_4(H).
\end{align*}
Together with $R_4(H)\leq R_4$, this proves
\eqref{eq:conditional-variance-bound}.
\end{proof}

\subsection{Kolmogorov bounds}
\label{subsec:weighted-spin-BE}

We first recall the classical Berry--Esseen inequality
\citep{Berry1941,Esseen1942}; see also
\citet[Section~3.4, (3.27)]{ChenGoldsteinShao2011}.
We use the following form with numerical constant one.

\begin{lemma}
\label{lem:classical-berry-esseen}
Let $\xi_1,\ldots,\xi_n$ be independent real random variables
with $\mathbb E\xi_i=0$ and $\mathbb E|\xi_i|^3<\infty$
for $i=1,\ldots,n$.
If $v:=\sum_{i=1}^n\mathbb E\xi_i^2>0$, then
\begin{equation}
\label{eq:classical-BE-constant-one}
d_{\mathrm K}\!\left(\frac{\sum_{i=1}^n\xi_i}{\sqrt v},G\right)
\leq\frac{1}{v^{3/2}}\sum_{i=1}^n\mathbb E|\xi_i|^3.
\end{equation}
\end{lemma}

We apply the classical Berry--Esseen inequality (\cref{lem:classical-berry-esseen}) after conditioning
on $\omega$. The conditional variance estimate then controls
the effect of replacing $V_a$ by $s_a^2$.

\begin{proposition}[Finite-graph Berry--Esseen bounds]
\label{prop:weighted-ising-normal-approximation}
For the total spin
\[
S:=T_{\mathbf1}=\sum_{x\in\mathsf V}\sigma_x,
\qquad s:=s_{\mathbf1},
\]
we have
\begin{equation}
\label{eq:finite-graph-magnetization-bound}
d_{\mathrm K}\!\left(\frac{S}{s},G\right)
\leq \frac{\mathbb ER_3}{|\mathsf V|^{3/2}}
+\frac{2}{|\mathsf V|}\sqrt{\frac23\mathbb ER_4}.
\end{equation}
For every nonzero $a\in\mathbb R^{\mathsf V}$, with
$H=\{x\in\mathsf V:a_x\ne0\}$,
\begin{equation}
\label{eq:signed-intrinsic-bound-new}
d_{\mathrm K}\!\left(\frac{T_a}{s_a},G\right)
\leq
\frac{2^{3/2}A^3\mathbb ER_3(H)}{s_a^3}
+\frac{4A^2}{\sqrt3}
\frac{\sqrt{\mathbb ER_4(H)}}{s_a^2}
+\frac{16A^4}{3}\frac{\mathbb ER_4(H)}{s_a^4}.
\end{equation}
\end{proposition}

\begin{proof}[Proof of \cref{prop:weighted-ising-normal-approximation}]
Fix $a\ne0$, and take $G$ to be independent of the
Edwards--Sokal coupling. Set
\[
\Theta_a:=
\begin{cases}
\displaystyle
\frac{\sum_{C\in\mathcal C(\omega)}|a(C)|^3}{V_a^{3/2}},
& V_a>0,\\[2mm]
0,& V_a=0.
\end{cases}
\]
Conditional on $\omega$ with $V_a>0$, the representation
\eqref{eq:weighted-ES} is a sum of independent centered variables.
Their variances sum to $V_a$, and their third absolute moments
sum to $\sum_C|a(C)|^3$.
Applying \cref{lem:classical-berry-esseen} with $v=V_a$
and rescaling both variables by $\sqrt{V_a}/s_a$ bounds
the conditional Kolmogorov distance between
$T_a/s_a$ and $\sqrt{V_a}G/s_a$ by $\min\{1,\Theta_a\}$.
If $V_a=0$, then every $a(C)$ is zero, so both variables vanish.
Averaging over $\omega$ gives
\begin{equation}
\label{eq:conditional-BE-mixture}
d_{\mathrm K}\!\left(
\frac{T_a}{s_a},\frac{\sqrt{V_a}}{s_a}G
\right)
\leq\mathbb E\min\{1,\Theta_a\}.
\end{equation}

We next compare $\sqrt{V_a}G/s_a$ with $G$.
For every $u\geq0$,
\begin{equation}
\label{eq:gaussian-scale-comparison}
d_{\mathrm K}(\sqrt u\,G,G)\leq2|u-1|.
\end{equation}
Applying this bound for each fixed $\omega$ and then averaging,
we obtain
\[
\begin{aligned}
d_{\mathrm K}\!\left(\frac{\sqrt{V_a}}{s_a}G,G\right)
&\leq2\mathbb E\left|\frac{V_a}{s_a^2}-1\right|
\leq\frac{2}{s_a^2}\sqrt{\operatorname{Var}(V_a)}.
\end{aligned}
\]
The last step uses $\mathbb EV_a=s_a^2$ and Cauchy--Schwarz.
Together with \eqref{eq:conditional-BE-mixture} and the triangle
inequality, this gives
\begin{equation}
\label{eq:weighted-master}
d_{\mathrm K}\!\left(\frac{T_a}{s_a},G\right)
\leq\mathbb E\min\{1,\Theta_a\}
+\frac{2}{s_a^2}\sqrt{\operatorname{Var}(V_a)}.
\end{equation}

For the total spin, take $a=\mathbf1$.
Then $V_{\mathbf1}=R_2$ and
$\Theta_{\mathbf1}=R_3/R_2^{3/2}$.
By \eqref{eq:FG-basic}, $R_2\geq |\mathsf V|$ and $s^2\geq |\mathsf V|$.
Consequently,
\[
\mathbb E\min\{1,\Theta_{\mathbf1}\}
\leq \frac{\mathbb ER_3}{|\mathsf V|^{3/2}}.
\]
Also, \eqref{eq:magnetization-variance-fluctuation} gives
\[
\frac{2}{s^2}\sqrt{\operatorname{Var}(R_2)}
\leq\frac{2}{|\mathsf V|}\sqrt{\frac23\mathbb ER_4}.
\]
Substituting these two estimates into \eqref{eq:weighted-master}
proves \eqref{eq:finite-graph-magnetization-bound}.

For general weights, $V_a$ can be small because weights within
a cluster can cancel. Define the event
\[
\mathcal B:=\{V_a<s_a^2/2\}.
\]
On $\mathcal B^{\mathrm c}$, we have $V_a\geq s_a^2/2$.
Using \eqref{eq:cluster-weight-bound}, we get
\[
\Theta_a
\leq\frac{2^{3/2}}{s_a^3}\sum_C|a(C)|^3
\leq\frac{2^{3/2}A^3R_3(H)}{s_a^3}.
\]
On $\mathcal B$, we use
$\min\{1,\Theta_a\}\leq1$.
Thus
\begin{equation}
\label{eq:conditional-BE-general-bound}
\mathbb E\min\{1,\Theta_a\}
\leq\frac{2^{3/2}A^3\mathbb ER_3(H)}{s_a^3}
+\mathbb P(\mathcal B).
\end{equation}
Since $\mathbb EV_a=s_a^2$, Chebyshev's inequality and
\eqref{eq:conditional-variance-bound} yield
\begin{equation}
\label{eq:small-conditional-variance-bound}
\mathbb P(\mathcal B)
\leq\frac{4\operatorname{Var}(V_a)}{s_a^4}
\leq\frac{16A^4}{3s_a^4}\mathbb ER_4(H).
\end{equation}
The same variance estimate gives
\begin{equation}
\label{eq:variance-mixture-bound}
\frac{2}{s_a^2}\sqrt{\operatorname{Var}(V_a)}
\leq\frac{4A^2}{\sqrt3}
\frac{\sqrt{\mathbb ER_4(H)}}{s_a^2}.
\end{equation}
Substituting \eqref{eq:conditional-BE-general-bound},
\eqref{eq:small-conditional-variance-bound}, and
\eqref{eq:variance-mixture-bound} into \eqref{eq:weighted-master}
proves \eqref{eq:signed-intrinsic-bound-new}.
\end{proof}

\subsection{Wasserstein bounds}
\label{subsec:weighted-spin-W}

We now use the same conditional representation to obtain a bound
in the $1$-Wasserstein distance.

\begin{proposition}[Finite-graph Wasserstein bound]
\label{prop:weighted-ising-wasserstein}
For every nonzero $a\in\mathbb R^{\mathsf V}$, with
$H=\{x\in\mathsf V:a_x\ne0\}$,
\begin{equation}
\label{eq:finite-graph-wasserstein-bound}
\begin{aligned}
d_{\mathrm W}\!\left(\frac{T_a}{s_a},G\right)
\leq{}&\frac{6A^3\mathbb ER_3(H)}{s_a^3}
+2\sqrt{\frac{2}{3\pi}}\,
\frac{A^2\sqrt{\mathbb ER_4(H)}}{s_a^2}\\
&+\frac{16\sqrt2}{3}\,
\frac{A^4\mathbb ER_4(H)}{s_a^4}.
\end{aligned}
\end{equation}
\end{proposition}

\begin{proof}[Proof of \cref{prop:weighted-ising-wasserstein}]
We use the independent-sum bound in
\citet[Corollary~2.5]{Chen2015}.
Let $\xi_1,\ldots,\xi_k$ be independent random variables with
mean zero and finite third absolute moments, and set
$v:=\sum_{i=1}^k\mathbb E\xi_i^2>0$.
Then
\begin{equation}
\label{eq:weighted-W-independent-sum}
d_{\mathrm W}\!\left(\frac{\sum_{i=1}^k\xi_i}{\sqrt v},G\right)
\leq\frac{3}{v^{3/2}}\sum_{i=1}^k\mathbb E|\xi_i|^3.
\end{equation}

Take $G$ to be independent of the Edwards--Sokal coupling.
Fix a function $f$ with $\operatorname{Lip}(f)\leq1$.
Conditional on $\omega$ with $V_a>0$, apply
\eqref{eq:weighted-W-independent-sum} to the summands
$a(C)\varepsilon_C$ in \eqref{eq:weighted-ES}.
Rescaling by $\sqrt{V_a}/s_a$ gives
\begin{equation}
\label{eq:weighted-W-conditional-comparison}
\left|\mathbb E\!\left[
f\!\left(\frac{T_a}{s_a}\right)
-f\!\left(\frac{\sqrt{V_a}}{s_a}G\right)
\,\middle|\,\omega\right]\right|
\leq\frac{3}{s_aV_a}\sum_C|a(C)|^3.
\end{equation}
If $V_a=0$, both arguments of $f$ are zero, so the conditional
difference is zero.

As in the Kolmogorov proof, define
\[
\mathcal B:=\{V_a<s_a^2/2\}.
\]
On $\mathcal B^{\mathrm c}$, the right-hand side of
\eqref{eq:weighted-W-conditional-comparison} is at most
\[
\frac{6}{s_a^3}\sum_C|a(C)|^3
\leq\frac{6A^3}{s_a^3}R_3(H).
\]
On $\mathcal B$, the Lipschitz property gives
\[
\begin{aligned}
\left|\mathbb E\!\left[
f\!\left(\frac{T_a}{s_a}\right)
-f\!\left(\frac{\sqrt{V_a}}{s_a}G\right)
\,\middle|\,\omega\right]\right|
&\leq\frac{\mathbb E(|T_a|\mid\omega)}{s_a}
+\frac{\sqrt{V_a}}{s_a}\mathbb E|G|\\
&\leq\frac{2\sqrt{V_a}}{s_a}
\leq\sqrt2.
\end{aligned}
\]
Here we used $\mathbb E(T_a^2\mid\omega)=V_a$ and
$\mathbb E|G|\leq1$.
The probability estimate \eqref{eq:small-conditional-variance-bound}
gives
\[
\mathbb P(\mathcal B)
\leq\frac{16A^4}{3s_a^4}\mathbb ER_4(H).
\]
Averaging over $\omega$ and taking the supremum over $f$ yields
\begin{equation}
\label{eq:weighted-W-mixture-comparison}
d_{\mathrm W}\!\left(
\frac{T_a}{s_a},\frac{\sqrt{V_a}}{s_a}G\right)
\leq\frac{6A^3\mathbb ER_3(H)}{s_a^3}
+\frac{16\sqrt2}{3}\,
\frac{A^4\mathbb ER_4(H)}{s_a^4}.
\end{equation}

It remains to compare $\sqrt{V_a}G/s_a$ with $G$.
Using the same $G$ in both variables and its independence from
$V_a$, we obtain
\[
d_{\mathrm W}\!\left(\frac{\sqrt{V_a}}{s_a}G,G\right)
\leq\sqrt{\frac2\pi}\,
\mathbb E\left|\frac{\sqrt{V_a}}{s_a}-1\right|.
\]
For every $u\geq0$, $|\sqrt u-1|\leq|u-1|$.
Since $\mathbb EV_a=s_a^2$, Cauchy--Schwarz gives
\[
\mathbb E\left|\frac{\sqrt{V_a}}{s_a}-1\right|
\leq\frac{\mathbb E|V_a-s_a^2|}{s_a^2}
\leq\frac{\sqrt{\operatorname{Var}(V_a)}}{s_a^2}.
\]
Combining these inequalities with
\eqref{eq:conditional-variance-bound}, we get
\begin{equation}
\label{eq:weighted-W-random-scale}
d_{\mathrm W}\!\left(\frac{\sqrt{V_a}}{s_a}G,G\right)
\leq2\sqrt{\frac{2}{3\pi}}\,
\frac{A^2\sqrt{\mathbb ER_4(H)}}{s_a^2}.
\end{equation}
The triangle inequality, \eqref{eq:weighted-W-mixture-comparison},
and \eqref{eq:weighted-W-random-scale} prove
\eqref{eq:finite-graph-wasserstein-bound}.
\end{proof}

We now express \eqref{eq:signed-intrinsic-bound-new} and
\eqref{eq:finite-graph-wasserstein-bound} in terms of cluster
moments and a lower bound on the variance.

\begin{corollary}[Normal approximation from cluster moments]
\label{cor:cluster-moment-normal-approximation}
Let $a\ne0$ satisfy $A\leq1$, and set $M:=|H|$.
Suppose that there are constants $b_2,b_3>0$ such that
\begin{equation}
\label{eq:cluster-moment-assumptions}
\mathbb E|C(x)|^2\leq b_2,
\qquad
\mathbb E|C(x)|^3\leq b_3
\qquad\text{for every }x\in H.
\end{equation}
The expectations are taken under $\phi_{\mathsf V,K,2}$.
Assume also that $s_a^2\geq cM$ for a constant $c>0$.
Then there exists $C=C(b_2,b_3,c)<\infty$ such that
\begin{equation}
\label{eq:cluster-moment-normal-rates}
\begin{aligned}
d_{\mathrm K}\!\left(\frac{T_a}{s_a},G\right)
&\leq CM^{-1/2},\\
d_{\mathrm W}\!\left(\frac{T_a}{s_a},G\right)
&\leq CM^{-1/2}.
\end{aligned}
\end{equation}
Both bounds hold under $\mu_{\mathsf V,K}$.
The constant $C$ depends only on $b_2,b_3,c$.
\end{corollary}

\begin{proof}[Proof of \cref{cor:cluster-moment-normal-approximation}]
By \eqref{eq:localized-cluster-sum-identity},
$|C(x)\cap H|\leq|C(x)|$, and
\eqref{eq:cluster-moment-assumptions}, we have
\begin{align*}
\mathbb ER_3(H)
&=\sum_{x\in H}\mathbb E|C(x)\cap H|^2
\leq\sum_{x\in H}\mathbb E|C(x)|^2
\leq b_2M,\\
\mathbb ER_4(H)
&=\sum_{x\in H}\mathbb E|C(x)\cap H|^3
\leq\sum_{x\in H}\mathbb E|C(x)|^3
\leq b_3M.
\end{align*}
Since $A\leq1$ and $s_a^2\geq cM$, it follows that
\begin{equation}
\label{eq:cluster-moment-ratios}
\begin{aligned}
\frac{A^3\mathbb ER_3(H)}{s_a^3}
&\leq\frac{b_2}{c^{3/2}}M^{-1/2},\\
\frac{A^2\sqrt{\mathbb ER_4(H)}}{s_a^2}
&\leq\frac{\sqrt{b_3}}{c}M^{-1/2},\\
\frac{A^4\mathbb ER_4(H)}{s_a^4}
&\leq\frac{b_3}{c^2}M^{-1}.
\end{aligned}
\end{equation}
Because $a\ne0$, we have $M\geq1$, so $M^{-1}\leq M^{-1/2}$.
Substituting \eqref{eq:cluster-moment-ratios} into
\eqref{eq:signed-intrinsic-bound-new} and
\eqref{eq:finite-graph-wasserstein-bound} proves
\eqref{eq:cluster-moment-normal-rates}.
\end{proof}

For the total spin, $a=\mathbf1$, $H=\mathsf V$, and $M=|\mathsf V|$.
By \eqref{eq:FG-basic}, $s_{\mathbf1}^2\geq |\mathsf V|$, so
\cref{cor:cluster-moment-normal-approximation} applies with
$c=1$ once \eqref{eq:cluster-moment-assumptions} is verified.
In \cref{sec:proof-applications}, we establish the cluster
moment and variance estimates needed to apply
\cref{cor:cluster-moment-normal-approximation} to the models
in \cref{sec:applications-new}.
For the infinite-volume measures, we keep the finite observation
set fixed and enlarge the surrounding boxes.
We obtain both bounds with constants independent of the boxes,
and then pass to the infinite-volume limit.

\section{Proofs of the main theorems}\label{sec:proof-applications}

We prove the Kolmogorov and $1$-Wasserstein bounds stated in
\cref{sec:applications-new}.
For the free-boundary and fully periodized ferromagnetic models,
we verify the cluster moment and variance assumptions of
\cref{cor:cluster-moment-normal-approximation}.
We then apply \cref{lem:fixed-set-limit} to obtain the
infinite-volume ferromagnetic bounds.
The observation set remains fixed while the surrounding boxes grow.
For the nearest-neighbor model,
\cref{prop:nearest-neighbor-local-approximation} treats both signs
of $\beta$ and gives the estimates used in the finite-volume
and infinite-volume proofs.

\subsection{Free-boundary model}
\label{sec:free-long-range}

Let $J$ satisfy \eqref{eq:intro-summable-J}, and let $\beta\geq0$.
For a nonempty finite set $\Lambda\subset\mathbb Z^d$, define
\begin{equation}
\label{eq:finite-free-edge-set}
E_\Lambda
:=\bigl\{\{x,y\}:x,y\in\Lambda,\ x\ne y,\ J(x-y)>0\bigr\}.
\end{equation}
With $\mathsf V=\Lambda$, $\mathsf E=E_\Lambda$, and
$K_{xy}=\beta J(x-y)$, the measure
$\mu^{\mathrm f}_{\Lambda,\beta}$ is the Ising measure in
\eqref{eq:finite-ising-new}.
Write $\phi^{\mathrm f}_{\Lambda,\beta,2}$ for its FK measure
defined by \eqref{eq:finite-random-cluster-measure}.
The edge parameters are $p_{xy}=1-e^{-2\beta J(x-y)}$.
For $x\in\Lambda$, let $C_\Lambda(x)$ be the open cluster
containing $x$.

For the infinite-volume model, define the weighted graph
\begin{equation}
\label{eq:support-weighted-graph}
\begin{aligned}
\mathsf G_J&=(\mathbb Z^d,E_J,w),\\
E_J&:=\bigl\{\{x,y\}:x,y\in\mathbb Z^d,\ x\ne y,\
J(x-y)>0\bigr\},\\
w_{\{x,y\}}&:=2J(x-y),
\qquad \{x,y\}\in E_J.
\end{aligned}
\end{equation}
Since $w_{\{x,y\}}=2J(x-y)$, the FK edge parameter satisfies
\[
p_{xy}=1-e^{-2\beta J(x-y)}
      =1-e^{-\beta w_{\{x,y\}}}.
\]
By \eqref{eq:intro-summable-J},
\begin{equation}
\label{eq:support-weighted-degree}
\sum_{\substack{y\in\mathbb Z^d\\\{x,y\}\in E_J}}
w_{\{x,y\}}
=2\sum_{z\in\mathbb Z^d}J(z)
=2\|J\|_1<\infty
\qquad(x\in\mathbb Z^d).
\end{equation}

As $R\to\infty$, the free FK measures
$\phi^{\mathrm f}_{B_R,\beta,2}$ converge to a measure on
$\{0,1\}^{E_J}$, denoted by $\phi^{\mathrm f}_{\beta,2}$.
Here convergence means convergence of the joint distribution
of every fixed finite set of edge variables.
The limit exists; see \citet[Section~1.2]{Hutchcroft2020}.
Replacing the boxes by any increasing sequence of finite sets
whose union is $\mathbb Z^d$ gives the same limit.
For $x\in\mathbb Z^d$, let $C(x)$ be the open cluster
containing $x$ under $\phi^{\mathrm f}_{\beta,2}$.
For real $r\geq1$, set
\begin{equation}
\label{eq:free-cluster-moment-definition}
M_r(\beta,J):=\mathbb E|C(0)|^r\in[1,\infty].
\end{equation}
The expectation in \eqref{eq:free-cluster-moment-definition}
is taken under $\phi^{\mathrm f}_{\beta,2}$.

For every $0\leq\beta<\beta_{\mathrm c}(J)$, we will prove
\begin{equation}
\label{eq:free-moment-obligations}
\sup_{\substack{\varnothing\ne\Lambda\subset\mathbb Z^d
                 \text{ finite}\\x\in\Lambda}}
\mathbb E|C_\Lambda(x)|^r
\leq M_r(\beta,J)<\infty,
\qquad r=1,2,3.
\end{equation}
The expectation in \eqref{eq:free-moment-obligations}
is taken under $\phi^{\mathrm f}_{\Lambda,\beta,2}$.
The cases $r=2,3$ verify the cluster moment assumptions of
\cref{cor:cluster-moment-normal-approximation} for the total spin.
We also use the case $r=1$ to prove the variance upper bound
$s_\Lambda^2\leq|\Lambda|\chi_J(\beta)$ required in
\cref{rem:kolmogorov-rate}.

To prove \eqref{eq:free-moment-obligations}, we first compare
$\beta_{\mathrm c}(J)$ with the free FK percolation threshold,
defined by
\begin{equation}
\label{eq:free-FK-critical-parameter}
\beta_{\mathrm c}^{\mathrm{FK},\mathrm f}
:=\inf\bigl\{\beta\geq0:
\phi^{\mathrm f}_{\beta,2}(|C(0)|=\infty)>0\bigr\},
\end{equation}
where $\inf\varnothing=\infty$.

\begin{proposition}[Free FK threshold and exponential cluster-size decay]
\label{prop:critical-equivalence}
Let $J$ satisfy \eqref{eq:intro-summable-J}.
Then
\begin{equation}
\label{eq:free-FK-critical-parameter-equality}
\beta_{\mathrm c}^{\mathrm{FK},\mathrm f}=\beta_{\mathrm c}(J).
\end{equation}
For every $\beta\geq0$,
\begin{equation}
\label{eq:infinite-free-ES-identity}
\chi_J(\beta)=\mathbb E|C(0)|.
\end{equation}
The expectation is taken under $\phi^{\mathrm f}_{\beta,2}$,
and the equality holds in $[0,\infty]$.
For every $0\leq\beta<\beta_{\mathrm c}(J)$, there exist
constants $A_{\beta,J}<\infty$ and $a_{\beta,J}>0$ such that
\begin{equation}
\label{eq:hutchcroft-tail}
\phi^{\mathrm f}_{\beta,2}(|C(0)|\geq m)
\leq A_{\beta,J}e^{-a_{\beta,J}m}
\qquad\text{for every integer }m\geq1.
\end{equation}
\end{proposition}

\begin{proof}[Proof of \cref{prop:critical-equivalence}]
If $J\equiv0$, every FK cluster is a singleton and
$\chi_J(\beta)=1$ for every $\beta\geq0$.
Thus \eqref{eq:infinite-free-ES-identity} holds, and
$\beta_{\mathrm c}^{\mathrm{FK},\mathrm f}
=\beta_{\mathrm c}(J)=\infty$.
The bound \eqref{eq:hutchcroft-tail} holds with
$A_{\beta,J}=e$ and $a_{\beta,J}=1$.
We now assume $J\not\equiv0$.

Fix $\beta\geq0$ and $x\in\mathbb Z^d$.
For every box $B_R$ containing $0$ and $x$,
\eqref{eq:two-point-connectivity} gives
\[
\mathbb E[\sigma_0\sigma_x]
=\phi^{\mathrm f}_{B_R,\beta,2}(0\leftrightarrow x).
\]
Here the expectation is taken under $\mu^{\mathrm f}_{B_R,\beta}$.

We next justify the limit of the FK connectivity probability.
For finite sets $\Lambda\subseteq\Delta$, condition all edges
in $E_\Delta\setminus E_\Lambda$ to be closed.
Under $\phi^{\mathrm f}_{\Delta,\beta,2}$, the conditional law
on $E_\Lambda$ is $\phi^{\mathrm f}_{\Lambda,\beta,2}$.
The conditioning event has positive probability since $p_e<1$.
An event is called increasing if, whenever it occurs for an
edge configuration, it also occurs after any closed edges are opened.
Positive association of the FK measure
\citep[Section~3.2]{Grimmett2006} implies that conditioning
edges to be closed cannot increase the probability of an
increasing event.
Thus the free FK measures increase under enlargement of the
vertex set.

Extend each configuration on $E_{B_R}$ to $E_J$ by closing
all edges outside $E_{B_R}$.
We can therefore construct configurations $\omega^{(R)}$
on one probability space, with law
$\phi^{\mathrm f}_{B_R,\beta,2}$ on $E_{B_R}$, such that
\[
\omega^{(R)}_e\leq\omega^{(R+1)}_e
\qquad(e\in E_J).
\]
The limiting configuration
$\omega=(\lim_{R\to\infty}\omega^{(R)}_e)_{e\in E_J}$
has law $\phi^{\mathrm f}_{\beta,2}$.
Every open path from $0$ to $x$ in $\omega$ uses finitely
many edges.
The edges of that path are all open in $\omega^{(R)}$ for
all sufficiently large $R$.
Consequently,
\[
\lim_{R\to\infty}
\phi^{\mathrm f}_{B_R,\beta,2}(0\leftrightarrow x)
=\phi^{\mathrm f}_{\beta,2}(0\leftrightarrow x).
\]
Taking the free limits in \eqref{eq:two-point-connectivity}
and using \eqref{eq:intro-infinite-volume}, we obtain
\[
\mathbb E[\sigma_0\sigma_x]
=\phi^{\mathrm f}_{\beta,2}(0\leftrightarrow x),
\]
where the expectation is now taken under $\mu^\infty_\beta$.
By \eqref{eq:free-susceptibility-critical-point} and Tonelli's
theorem,
\begin{align}
\chi_J(\beta)
&=\sum_{x\in\mathbb Z^d}
\phi^{\mathrm f}_{\beta,2}(0\leftrightarrow x)\notag\\
&=\mathbb E\sum_{x\in\mathbb Z^d}
\mathbf1_{\{0\leftrightarrow x\}}
=\mathbb E|C(0)|.
\label{eq:infinite-susceptibility-calculation}
\end{align}
The expectations in \eqref{eq:infinite-susceptibility-calculation}
are taken under $\phi^{\mathrm f}_{\beta,2}$.
This proves \eqref{eq:infinite-free-ES-identity}, including
the case where both sides are infinite.

Translations of $\mathbb Z^d$ preserve $E_J$ and the weights
$w_{\{x,y\}}=2J(x-y)$.
Hence $\mathsf G_J$ is an infinite transitive weighted graph.
The total weight incident to each vertex is finite by
\eqref{eq:support-weighted-degree}.
\citet[Theorem~1.5]{Hutchcroft2020} therefore applies with
$q=2$ and free boundary conditions.
For every $0\leq\beta<\beta_{\mathrm c}^{\mathrm{FK},\mathrm f}$,
it gives
\[
\phi^{\mathrm f}_{\beta,2}(|C(0)|\geq m)
\leq A_{\beta,J}e^{-a_{\beta,J}m}
\qquad(m\geq1)
\]
for some $A_{\beta,J}<\infty$ and $a_{\beta,J}>0$.
In particular,
\[
\mathbb E|C(0)|
=\sum_{m\geq1}\phi^{\mathrm f}_{\beta,2}(|C(0)|\geq m)
\leq A_{\beta,J}\sum_{m\geq1}e^{-a_{\beta,J}m}
<\infty,
\]
with expectation under $\phi^{\mathrm f}_{\beta,2}$.
By \eqref{eq:infinite-free-ES-identity},
$\chi_J(\beta)<\infty$ for every
$0\leq\beta<\beta_{\mathrm c}^{\mathrm{FK},\mathrm f}$.
The definition of $\beta_{\mathrm c}(J)$ therefore gives
\begin{equation}
\label{eq:free-FK-below-susceptibility-threshold}
\beta_{\mathrm c}^{\mathrm{FK},\mathrm f}
\leq\beta_{\mathrm c}(J).
\end{equation}
If $\beta_{\mathrm c}^{\mathrm{FK},\mathrm f}=0$, this inequality
follows directly from $\beta_{\mathrm c}(J)\geq0$.

Conversely, fix $0\leq\beta<\beta_{\mathrm c}(J)$.
The monotonicity of $\chi_J$ stated after
\eqref{eq:free-susceptibility-critical-point} gives
$\chi_J(\beta)<\infty$.
By \eqref{eq:infinite-free-ES-identity},
$\mathbb E|C(0)|<\infty$ under $\phi^{\mathrm f}_{\beta,2}$, so
\[
\phi^{\mathrm f}_{\beta,2}(|C(0)|=\infty)=0.
\]
The probability
$\phi^{\mathrm f}_{\beta,2}(|C(0)|=\infty)$ is nondecreasing
in $\beta$; see \citet[Section~1.2]{Hutchcroft2020}.
Hence
\[
\phi^{\mathrm f}_{t,2}(|C(0)|=\infty)=0
\qquad(0\leq t\leq\beta).
\]
By \eqref{eq:free-FK-critical-parameter},
$\beta\leq\beta_{\mathrm c}^{\mathrm{FK},\mathrm f}$.
Since this holds for every $0\leq\beta<\beta_{\mathrm c}(J)$,
\begin{equation}
\label{eq:susceptibility-below-free-FK-threshold}
\beta_{\mathrm c}(J)
\leq\beta_{\mathrm c}^{\mathrm{FK},\mathrm f}.
\end{equation}
If $\beta_{\mathrm c}(J)=0$, this inequality follows from
$\beta_{\mathrm c}^{\mathrm{FK},\mathrm f}\geq0$.
Both threshold comparisons remain valid in $[0,\infty]$.
Indeed, if $\beta_{\mathrm c}^{\mathrm{FK},\mathrm f}=\infty$,
then $\chi_J(\beta)<\infty$ for every finite $\beta\geq0$,
so $\beta_{\mathrm c}(J)=\infty$.
If $\beta_{\mathrm c}(J)=\infty$, then
$\beta\leq\beta_{\mathrm c}^{\mathrm{FK},\mathrm f}$ for
every finite $\beta\geq0$, so
$\beta_{\mathrm c}^{\mathrm{FK},\mathrm f}=\infty$.
Combining
\eqref{eq:free-FK-below-susceptibility-threshold} and
\eqref{eq:susceptibility-below-free-FK-threshold}
proves \eqref{eq:free-FK-critical-parameter-equality}.
Substituting
$\beta_{\mathrm c}^{\mathrm{FK},\mathrm f}=\beta_{\mathrm c}(J)$
into the parameter range of
\citet[Theorem~1.5]{Hutchcroft2020} proves
\eqref{eq:hutchcroft-tail}.
\end{proof}

We next prove \eqref{eq:free-moment-obligations} by comparing
the finite-volume cluster moments with $M_r(\beta,J)$.

\begin{proposition}[Comparison of finite-volume and infinite-volume cluster moments]
\label{prop:cluster-moments}
Let $J$ satisfy \eqref{eq:intro-summable-J}, and fix
$0\leq\beta<\beta_{\mathrm c}(J)$.
For every real $r\geq1$, every nonempty finite set
$\Lambda\subset\mathbb Z^d$, and every $x\in\Lambda$,
\begin{equation}
\label{eq:finite-moment-comparison}
\mathbb E|C_\Lambda(x)|^r\leq M_r(\beta,J)<\infty.
\end{equation}
The expectation is taken under $\phi^{\mathrm f}_{\Lambda,\beta,2}$,
and $M_r(\beta,J)$ is defined in
\eqref{eq:free-cluster-moment-definition}.
\end{proposition}

\begin{proof}[Proof of \cref{prop:cluster-moments}]
If $J\equiv0$, then $|C_\Lambda(x)|=1$ and $M_r(\beta,J)=1$,
so \eqref{eq:finite-moment-comparison} holds.
Assume $J\not\equiv0$.
Fix a nonempty finite set $\Lambda\subset\mathbb Z^d$ and
$x\in\Lambda$.
We first prove that, for every integer $m\geq1$,
\begin{equation}
\label{eq:finite-tail-comparison}
\phi^{\mathrm f}_{\Lambda,\beta,2}(|C_\Lambda(x)|\geq m)
\leq\phi^{\mathrm f}_{\beta,2}(|C(0)|\geq m).
\end{equation}

For $\omega\in\{0,1\}^{E_J}$, let $C_\Lambda^\omega(x)$
be the cluster of $x$ formed using only the open edges in
$E_\Lambda$.
Define
\begin{equation}
\label{eq:finite-cluster-cylinder-event}
A_{\Lambda,x,m}
:=\{\omega:|C_\Lambda^\omega(x)|\geq m\}.
\end{equation}
Opening additional edges cannot decrease $|C_\Lambda^\omega(x)|$,
so $A_{\Lambda,x,m}$ is an increasing event.
It depends only on the finitely many edge variables in $E_\Lambda$.
For a finite set $\Delta$, extend each configuration on $E_\Delta$
to $E_J$ by setting $\omega_e=0$ for $e\notin E_\Delta$.

For every $R$ with $\Lambda\subseteq B_R$, the domain
monotonicity established in the proof of
\cref{prop:critical-equivalence} gives the inequality below in
\eqref{eq:finite-domain-monotonicity}.
The joint law of the edge variables in $E_\Lambda$ under
$\phi^{\mathrm f}_{B_R,\beta,2}$ converges to their joint law
under $\phi^{\mathrm f}_{\beta,2}$.
Since $A_{\Lambda,x,m}$ depends only on these finitely many
edge variables, its probability converges as follows:
\begin{equation}
\label{eq:finite-domain-monotonicity}
\phi^{\mathrm f}_{\Lambda,\beta,2}(A_{\Lambda,x,m})
\leq\phi^{\mathrm f}_{B_R,\beta,2}(A_{\Lambda,x,m})
\xrightarrow[R\to\infty]{}
\phi^{\mathrm f}_{\beta,2}(A_{\Lambda,x,m}).
\end{equation}

The cluster $C_\Lambda^\omega(x)$ is contained in the full
cluster $C(x)$, so
\begin{equation}
\label{eq:cylinder-event-inclusion}
A_{\Lambda,x,m}\subseteq\{|C(x)|\geq m\}.
\end{equation}
Translations preserve the interaction strengths $J(x-y)$.
The free FK limit is independent of the increasing sequence
of finite sets used to construct it.
Hence $\phi^{\mathrm f}_{\beta,2}$ is translation invariant.
Combining \eqref{eq:finite-domain-monotonicity} and
\eqref{eq:cylinder-event-inclusion}, we obtain
\begin{align*}
\phi^{\mathrm f}_{\Lambda,\beta,2}(|C_\Lambda(x)|\geq m)
&=\phi^{\mathrm f}_{\Lambda,\beta,2}(A_{\Lambda,x,m})\\
&\leq\phi^{\mathrm f}_{\beta,2}(A_{\Lambda,x,m})\\
&\leq\phi^{\mathrm f}_{\beta,2}(|C(x)|\geq m)\\
&=\phi^{\mathrm f}_{\beta,2}(|C(0)|\geq m).
\end{align*}
This proves \eqref{eq:finite-tail-comparison}.

For every real $r\geq1$ and every random variable $Y$ taking
values in $\{0,1,2,\ldots\}\cup\{\infty\}$,
\begin{equation}
\label{eq:tail-sum-identity}
Y^r=\sum_{m=1}^{\infty}
\bigl(m^r-(m-1)^r\bigr)\mathbf1_{\{Y\geq m\}}.
\end{equation}
When $Y=k<\infty$, the sum equals
$\sum_{m=1}^k(m^r-(m-1)^r)=k^r$.
When $Y=\infty$, the first $k$ terms sum to $k^r$, which
tends to infinity as $k\to\infty$.
Taking expectations in \eqref{eq:tail-sum-identity} by
Tonelli's theorem and applying \eqref{eq:finite-tail-comparison}
gives
\begin{align*}
\mathbb E|C_\Lambda(x)|^r
&=\sum_{m=1}^{\infty}\bigl(m^r-(m-1)^r\bigr)
\phi^{\mathrm f}_{\Lambda,\beta,2}(|C_\Lambda(x)|\geq m)\\
&\leq\sum_{m=1}^{\infty}\bigl(m^r-(m-1)^r\bigr)
\phi^{\mathrm f}_{\beta,2}(|C(0)|\geq m)\\
&=M_r(\beta,J).
\end{align*}
The expectation is taken under $\phi^{\mathrm f}_{\Lambda,\beta,2}$.

To prove $M_r(\beta,J)<\infty$, use
$m^r-(m-1)^r\leq r m^{r-1}$ for $m\geq1$ and
the exponential bound \eqref{eq:hutchcroft-tail} in
\eqref{eq:tail-sum-identity}:
\[
M_r(\beta,J)
\leq A_{\beta,J}r\sum_{m=1}^{\infty}
m^{r-1}e^{-a_{\beta,J}m}<\infty.
\]
This proves \eqref{eq:finite-moment-comparison} for every real
$r\geq1$.
\end{proof}

\begin{proof}[Proof of \cref{thm:free-long-range-new}]
Fix $0\leq\beta<\beta_{\mathrm c}(J)$ and a nonempty finite set
$\Lambda\subset\mathbb Z^d$.
For the total spin under $\mu^{\mathrm f}_{\Lambda,\beta}$,
take $a=\mathbf1$ on $\Lambda$.
Then $T_a=S_\Lambda$, $s_a=s_\Lambda$, $A=1$, $H=\Lambda$,
and $M=|\Lambda|$.
By \cref{prop:cluster-moments}, for every $x\in\Lambda$,
\[
\mathbb E|C_\Lambda(x)|^2\leq M_2(\beta,J),
\qquad
\mathbb E|C_\Lambda(x)|^3\leq M_3(\beta,J).
\]
All cluster expectations in this proof are taken under
$\phi^{\mathrm f}_{\Lambda,\beta,2}$.
The constants $M_2(\beta,J)$ and $M_3(\beta,J)$ are finite
and depend only on $\beta$ and $J$.

By \eqref{eq:free-cluster-moment-definition} and
\eqref{eq:infinite-free-ES-identity},
$M_1(\beta,J)=\chi_J(\beta)$.
Using \eqref{eq:FG-basic}, \eqref{eq:cluster-sum-identity},
and \eqref{eq:finite-moment-comparison}, we obtain
\begin{equation}
\label{eq:free-variance-bounds}
|\Lambda|\leq s_\Lambda^2
=\sum_{x\in\Lambda}\mathbb E|C_\Lambda(x)|
\leq |\Lambda|M_1(\beta,J)
=|\Lambda|\chi_J(\beta).
\end{equation}
We can therefore apply
\cref{cor:cluster-moment-normal-approximation} with
\[
b_2=M_2(\beta,J),
\qquad
b_3=M_3(\beta,J),
\qquad
c=1.
\]
This proves \eqref{eq:free-long-range-BE-new} and
\eqref{eq:free-long-range-W-new}, with a constant depending
only on $\beta$ and $J$.
\end{proof}

\begin{proof}[Proof of \cref{rem:kolmogorov-rate}]
Fix $0\leq\beta<\beta_{\mathrm c}(J)$ and a nonempty finite set
$\Lambda\subset\mathbb Z^d$.
All probabilities and expectations in this proof are taken
under $\mu^{\mathrm f}_{\Lambda,\beta}$.
By \eqref{eq:free-variance-bounds},
\[
0<|\Lambda|\leq s_\Lambda^2\leq |\Lambda|\chi_J(\beta),
\qquad \chi_J(\beta)<\infty.
\]
Since $\mathbb E S_\Lambda=0$, Chebyshev's inequality gives
\[
\mathbb P\bigl(|S_\Lambda|\leq2s_\Lambda\bigr)\geq\frac34.
\]
The total spin takes values in $\{-|\Lambda|,-|\Lambda|+2,\ldots,|\Lambda|\}$.
Thus the interval $[-2s_\Lambda,2s_\Lambda]$ contains at most
$2s_\Lambda+1$ possible values, and hence
\[
p_\Lambda
:=\max_{t\in\mathbb R}\mathbb P(S_\Lambda=t)
\geq\frac{3}{4(2s_\Lambda+1)}.
\]
Choose $u_\Lambda$ such that
$\mathbb P(W_\Lambda=u_\Lambda)=p_\Lambda$.
The continuity of $\Phi$ gives
\begin{align*}
p_\Lambda
&=\mathbb P(W_\Lambda\leq u_\Lambda)
 -\mathbb P(W_\Lambda<u_\Lambda)\\
&\leq
\bigl|\mathbb P(W_\Lambda\leq u_\Lambda)-\Phi(u_\Lambda)\bigr|
+\bigl|\mathbb P(W_\Lambda<u_\Lambda)-\Phi(u_\Lambda)\bigr|\\
&\leq2d_{\mathrm K}(W_\Lambda,G).
\end{align*}
Therefore
\[
d_{\mathrm K}(W_\Lambda,G)
\geq\frac{p_\Lambda}{2}
\geq\frac{3}{8(2s_\Lambda+1)}.
\]
Finally, $s_\Lambda\leq\sqrt{|\Lambda|\chi_J(\beta)}$ and $|\Lambda|\geq1$ imply
\[
d_{\mathrm K}(W_\Lambda,G)
\geq
\frac{3}{8\bigl(2\sqrt{\chi_J(\beta)}+1\bigr)}\,|\Lambda|^{-1/2}.
\]
This proves \eqref{eq:free-K-lower-bound}.
\end{proof}

\subsection{Periodized model}
\label{sec:periodic-models}

Let $J$ satisfy \eqref{eq:intro-summable-J}.
For a bounded function $F$ of finitely many spins, write
$\langle F\rangle_\beta^{\mathrm f}$ for its expectation under
the infinite-volume free-boundary Gibbs measure $\mu^\infty_\beta$
with interaction $J$.

\subsubsection{Reduction to uniform torus estimates}
\label{subsubsec:periodic-reduction}

We apply \cref{cor:cluster-moment-normal-approximation} to the
periodized model in \eqref{eq:periodic-measure-new}.
For an integer $L\geq1$, define
\begin{equation}
\label{eq:periodic-FK-edge-set}
E_L:=\bigl\{\{\bar x,\bar y\}:\bar x,\bar y\in\mathbb T_L^d,
\ \bar x\ne\bar y,\ J_L(\bar x-\bar y)>0\bigr\}.
\end{equation}
For $\beta\geq0$, the Gibbs measure
$\mu^{\mathrm{per}}_{L,\beta}$ is the finite-graph Ising measure
in \eqref{eq:finite-ising-new} with
\[
\mathsf V=\mathbb T_L^d,
\qquad
\mathsf E=E_L,
\qquad
K_{\bar x\bar y}=\beta J_L(\bar x-\bar y).
\]
The FK edge parameters are
\begin{equation}
\label{eq:periodic-FK-edge-parameters}
p_{\bar x\bar y}:=1-e^{-2\beta J_L(\bar x-\bar y)},
\qquad \{\bar x,\bar y\}\in E_L.
\end{equation}
Write
\begin{equation}
\label{eq:periodic-FK-measure}
\phi^{\mathrm{per}}_{L,\beta,2}
:=\phi_{\mathbb T_L^d,K,2}
\end{equation}
for the FK measure defined in
\eqref{eq:finite-random-cluster-measure}.
For $\bar x\in\mathbb T_L^d$, let $C_L(\bar x)$ be
the open cluster containing $\bar x$.
For a function $F$ of the torus spins, write
$\langle F\rangle_{L,\beta}^{\mathrm{per}}$ for its expectation
under $\mu^{\mathrm{per}}_{L,\beta}$.

By \cref{prop:ES-identities}, the spin and edge marginals of
\eqref{eq:finite-ES-joint-law} are
$\mu^{\mathrm{per}}_{L,\beta}$ and
$\phi^{\mathrm{per}}_{L,\beta,2}$.
Spin-flip symmetry, translation invariance, and
\eqref{eq:two-point-connectivity} give
\begin{equation}
\label{eq:periodic-susceptibility-FK}
\begin{aligned}
\frac{s_L^2}{L^{d}}
&=\sum_{\bar x\in\mathbb T_L^d}
\bigl\langle\sigma_{\bar0}\sigma_{\bar x}\bigr\rangle_{L,\beta}^{\mathrm{per}}\\
&=\sum_{\bar x\in\mathbb T_L^d}
\phi^{\mathrm{per}}_{L,\beta,2}(\bar0\leftrightarrow\bar x)
=\mathbb E|C_L(\bar0)|.
\end{aligned}
\end{equation}
The expectation in \eqref{eq:periodic-susceptibility-FK}
is taken under $\phi^{\mathrm{per}}_{L,\beta,2}$.

The following proposition compares the torus two-point function
with a sum of lattice two-point functions.

\begin{proposition}
\label{prop:torus-to-lattice-comparison}
Fix $0\leq\beta<\beta_{\mathrm c}(J)$.
For every integer $L\geq1$ and every
$\bar x\in\mathbb T_L^d$ with representative $x\in\mathbb Z^d$,
\begin{equation}
\label{eq:torus-two-point-comparison}
\bigl\langle\sigma_{\bar0}\sigma_{\bar x}\bigr\rangle_{L,\beta}^{\mathrm{per}}
\leq
\sum_{u\in\mathbb Z^d}
\bigl\langle\sigma_0\sigma_{x+Lu}\bigr\rangle_{\beta}^{\mathrm f}.
\end{equation}
Consequently,
\begin{equation}
\label{eq:periodic-susceptibility-new}
\frac{s_L^2}{L^{d}}\leq\chi_J(\beta).
\end{equation}
\end{proposition}

To control the second and third cluster-size moments, we use
the following exponential tail bound.

\begin{proposition}[Uniform torus cluster-size tail]
\label{prop:periodic-cluster-tail}
Fix $0\leq\beta<\beta_0<\beta_{\mathrm c}(J)$.
There exist constants $A_0,a_0>0$, depending only on
$\beta,\beta_0,J$, such that, for all integers $L,m\geq1$,
\begin{equation}
\label{eq:periodic-volume-tail-new}
\phi^{\mathrm{per}}_{L,\beta,2}
\bigl(|C_L(\bar0)|\geq m\bigr)
\leq A_0e^{-a_0m}.
\end{equation}
\end{proposition}

\begin{proof}[Proof of \cref{thm:periodic-long-range-new}]
Fix $0\leq\beta<\beta_{\mathrm c}(J)$ and choose
$\beta_0\in(\beta,\beta_{\mathrm c}(J))$, depending only on
$\beta$ and $J$.
For every integer $L\geq1$ and every real $r\geq1$,
\eqref{eq:tail-sum-identity} and
\eqref{eq:periodic-volume-tail-new} give
\begin{align*}
\mathbb E|C_L(\bar0)|^r
&=\sum_{m=1}^{\infty}\bigl(m^r-(m-1)^r\bigr)
\phi^{\mathrm{per}}_{L,\beta,2}
\bigl(|C_L(\bar0)|\geq m\bigr)\\
&\leq A_0r\sum_{m=1}^{\infty}m^{r-1}e^{-a_0m}<\infty.
\end{align*}
All cluster expectations in this proof are taken under
$\phi^{\mathrm{per}}_{L,\beta,2}$.
Since $A_0$ and $a_0$ are independent of $L$, we obtain
\begin{equation}
\label{eq:periodic-uniform-moments}
\mathcal M_r^{\mathrm{per}}
:=\sup_{L\geq1}\mathbb E|C_L(\bar0)|^r<\infty,
\qquad r\geq1.
\end{equation}
The quantities $\mathcal M_r^{\mathrm{per}}$ depend only on
$r$, $\beta$, and $J$.
By translation invariance of $\phi^{\mathrm{per}}_{L,\beta,2}$,
for every $\bar x\in\mathbb T_L^d$ and $r=2,3$,
\[
\mathbb E|C_L(\bar x)|^r
=\mathbb E|C_L(\bar0)|^r
\leq\mathcal M_r^{\mathrm{per}}.
\]

Fix an integer $L\geq1$.
For the total spin under $\mu^{\mathrm{per}}_{L,\beta}$,
take $a=\mathbf1$ on $\mathbb T_L^d$.
Then $T_a=S_L$, $s_a=s_L$, $A=1$, $H=\mathbb T_L^d$,
and $M=L^{d}$.
By \eqref{eq:FG-basic} and
\eqref{eq:periodic-susceptibility-new},
\[
L^{d}\leq s_L^2\leq L^{d}\chi_J(\beta).
\]
Applying \cref{cor:cluster-moment-normal-approximation} with
\[
b_2=\mathcal M_2^{\mathrm{per}},
\qquad
b_3=\mathcal M_3^{\mathrm{per}},
\qquad
c=1
\]
gives both bounds in \eqref{eq:periodic-BE-new} and
\eqref{eq:periodic-W-new}.
The constant depends only on $\beta$ and $J$.
\end{proof}

We prove \cref{prop:torus-to-lattice-comparison} in
\cref{subsubsec:periodic-comparison} and
\cref{prop:periodic-cluster-tail} in
\cref{subsubsec:periodic-tail}.

\subsubsection{Torus-to-lattice comparison}
\label{subsubsec:periodic-comparison}

We first apply \cref{prop:finite-range-unfolding} to finite-range
truncations of $J$.
We then let the truncation range tend to infinity to prove
\cref{prop:torus-to-lattice-comparison}.

For $z=(z_1,\ldots,z_d)\in\mathbb Z^d$, write
$\|z\|_\infty:=\max_{1\leq i\leq d}|z_i|$.
For $R\geq1$, define the truncated interaction and its
periodization by
\begin{align}
J^{(R)}(z)
&:=J(z)\mathbf1_{\{\|z\|_\infty\leq R\}},
\qquad z\in\mathbb Z^d,\notag\\
J_L^{(R)}(\bar z)
&:=\sum_{u\in\mathbb Z^d}J^{(R)}(z+Lu),
\qquad \bar z\in\mathbb T_L^d\setminus\{\bar0\}.
\label{eq:periodized-truncation}
\end{align}
In the definition of $J_L^{(R)}(\bar z)$, $z$ is any
representative of $\bar z$.
Write $\langle\cdot\rangle_{L,\beta}^{\mathrm{per},(R)}$
for expectation under the zero-field torus Gibbs measure in
\eqref{eq:periodic-measure-new}, with $J_L$ replaced by
$J_L^{(R)}$.
For $y\in\mathbb Z^d$, define the lattice two-point functions
\begin{equation}
\label{eq:truncated-two-point-functions}
\tau_\beta^{(R)}(y):=\mathbb E[\sigma_0\sigma_y],
\qquad
\tau_\beta(y):=\langle\sigma_0\sigma_y\rangle_\beta^{\mathrm f}.
\end{equation}
Here $\mathbb E$ is taken under the zero-field infinite-volume
free-boundary Gibbs measure with interaction $J^{(R)}$
and inverse temperature $\beta$.

\begin{proposition}[Finite-range unfolding]
\label{prop:finite-range-unfolding}
Let $\beta\geq0$ and $R\geq1$.
For every integer $L\geq1$ and every
$\bar x\in\mathbb T_L^d$ with representative $x\in\mathbb Z^d$,
\begin{equation}
\label{eq:finite-range-unfolding}
\bigl\langle\sigma_{\bar0}\sigma_{\bar x}\bigr\rangle_{L,\beta}^{\mathrm{per},(R)}
\leq\sum_{u\in\mathbb Z^d}\tau_\beta^{(R)}(x+Lu).
\end{equation}
\end{proposition}

The sum in \eqref{eq:finite-range-unfolding} may be infinite.
The proof of \cref{prop:finite-range-unfolding} is given in
\cref{subsubsec:periodic-unfolding}.

\begin{proof}[Proof of \cref{prop:torus-to-lattice-comparison}]
Fix $0\leq\beta<\beta_{\mathrm c}(J)$, an integer $L\geq1$,
and $\bar x\in\mathbb T_L^d$ with representative $x\in\mathbb Z^d$.
The lattice correlations $\tau_\beta(y)$ are nonnegative by
\eqref{eq:two-point-connectivity} and the free-boundary limit.
If $\bar x=\bar0$, then
$\langle\sigma_{\bar0}\sigma_{\bar x}\rangle_{L,\beta}^{\mathrm{per}}=1$.
Since $x\in L\mathbb Z^d$, the term $u=-x/L$ in
\eqref{eq:torus-two-point-comparison} equals $\tau_\beta(0)=1$.
Thus \eqref{eq:torus-two-point-comparison} holds when
$\bar x=\bar0$.
We now assume $\bar x\ne\bar0$.

We first compare $\tau_\beta^{(R)}(y)$ and $\tau_\beta(y)$.
Consider the free FK measure with interaction $J$ in a box
$B_M$ containing $0$ and $y$.
Condition all edges $\{u,v\}$ with $\|u-v\|_\infty>R$
to be closed.
The remaining edges have the free FK law with interaction
$J^{(R)}$.
By positive association of the FK measure
\citep[Section~3.2]{Grimmett2006}, conditioning edges to be
closed cannot increase connection probabilities.
Using \eqref{eq:two-point-connectivity} in $B_M$ and taking
the free-boundary limits as $M\to\infty$ gives
\[
\tau_\beta^{(R)}(y)\leq\tau_\beta(y),
\qquad y\in\mathbb Z^d.
\]
Applying \eqref{eq:finite-range-unfolding} and the inequality
$\tau_\beta^{(R)}(y)\leq\tau_\beta(y)$ gives
\begin{equation}
\label{eq:truncated-two-point-bound}
\bigl\langle\sigma_{\bar0}\sigma_{\bar x}\bigr\rangle_{L,\beta}^{\mathrm{per},(R)}
\leq\sum_{u\in\mathbb Z^d}\tau_\beta^{(R)}(x+Lu)
\leq\sum_{u\in\mathbb Z^d}\tau_\beta(x+Lu).
\end{equation}

Monotone convergence applies to the nonnegative sum in
\eqref{eq:periodized-truncation}.
On a finite torus, the Gibbs weights are continuous in the
pair couplings and their sum is positive.
Therefore, as $R\to\infty$,
\begin{equation}
\label{eq:periodic-truncation-convergence}
\begin{aligned}
J_L^{(R)}(\bar z)&\uparrow J_L(\bar z)
\qquad(\bar z\ne\bar0),\\
\bigl\langle\sigma_{\bar0}\sigma_{\bar x}\bigr\rangle_{L,\beta}^{\mathrm{per},(R)}
&\longrightarrow
\bigl\langle\sigma_{\bar0}\sigma_{\bar x}\bigr\rangle_{L,\beta}^{\mathrm{per}}
\end{aligned}
\end{equation}
The upper bound $\sum_{u\in\mathbb Z^d}\tau_\beta(x+Lu)$
in \eqref{eq:truncated-two-point-bound} is independent of $R$.
Letting $R\to\infty$ and using
\eqref{eq:periodic-truncation-convergence} proves
\eqref{eq:torus-two-point-comparison}.

To sum \eqref{eq:torus-two-point-comparison}, take the set
of representatives
\begin{equation}
\label{eq:torus-fundamental-domain}
\mathcal F_L:=\{0,1,\ldots,L-1\}^d.
\end{equation}
Every $y\in\mathbb Z^d$ has a unique representation
$y=x+Lu$ with $x\in\mathcal F_L$ and $u\in\mathbb Z^d$.
Summing over the torus sites and using Tonelli's theorem gives
\begin{align}
\sum_{\bar x\in\mathbb T_L^d}
\bigl\langle\sigma_{\bar0}\sigma_{\bar x}\bigr\rangle_{L,\beta}^{\mathrm{per}}
&\leq\sum_{x\in\mathcal F_L}\sum_{u\in\mathbb Z^d}
\tau_\beta(x+Lu)\notag\\
&=\sum_{y\in\mathbb Z^d}\tau_\beta(y)
=\chi_J(\beta).
\label{eq:periodic-correlation-sum}
\end{align}
Combining \eqref{eq:periodic-correlation-sum} with
\eqref{eq:periodic-susceptibility-FK} proves
\eqref{eq:periodic-susceptibility-new}.
\end{proof}

\subsubsection{Uniform cluster-size tails on the torus}
\label{subsubsec:periodic-tail}

To prove \cref{prop:periodic-cluster-tail}, we first derive a
cluster-size tail bound for finite random-cluster models.
Let $\mathsf G=(\mathsf V,\mathsf E)$ be a finite graph with
vertex set $\mathsf V$ and nonempty edge set $\mathsf E$.
Assign a positive weight $\mathcal J_e$ to each edge
$e\in\mathsf E$.
Assume that, for any two vertices $v,w\in\mathsf V$, there is
a graph automorphism that sends $v$ to $w$ and preserves all
edge weights.

Fix $q\geq1$.
For $t>0$, define
\begin{equation}
\label{eq:finite-RC-parametrization}
\phi_{\mathsf G,t,q}(\omega)
:=\frac{1}{Z_{\mathsf G,t,q}}q^{k(\omega)}
\prod_{e\in\mathsf E}
\bigl(e^{t\mathcal J_e}-1\bigr)^{\omega_e},
\qquad \omega\in\{0,1\}^{\mathsf E}.
\end{equation}
Here $Z_{\mathsf G,t,q}$ is the normalizing constant, and
$k(\omega)$ is the number of open clusters.
This is the random-cluster measure with edge parameters
$p_e=1-e^{-t\mathcal J_e}$.
At $t=0$, let $\phi_{\mathsf G,0,q}$ assign probability one
to the configuration in which every edge is closed.

For $v\in\mathsf V$, let $C(v)$ be its open cluster and set
\begin{equation}
\label{eq:finite-chi-C}
\chi_{\mathsf G}(t):=\mathbb E|C(v)|,
\qquad
\mathfrak C_{\mathsf G}(t)
:=\max_{e\in\mathsf E}
\frac{e^{t\mathcal J_e}-1}{\mathcal J_e}.
\end{equation}
The expectation is taken under $\phi_{\mathsf G,t,q}$.
The assumption on graph automorphisms implies that
$\mathbb E|C(v)|$ does not depend on $v$.
For the torus application, we will take $q=2$, $t=\beta$, and
$\mathcal J_{\{\bar x,\bar y\}}=2J_L(\bar x-\bar y)$.
The case of a torus with no edges will be treated separately.

\begin{lemma}[Weighted derivative formula]
\label{lem:weighted-derivative}
Let $F:\{0,1\}^{\mathsf E}\to\mathbb R$ be independent of $t$.
For every $t>0$,
\begin{equation}
\label{eq:exact-derivative}
\frac{d}{dt}\mathbb EF
=\sum_{e\in\mathsf E}
\frac{\mathcal J_e e^{t\mathcal J_e}}
     {e^{t\mathcal J_e}-1}
\operatorname{Cov}(F,\omega_e).
\end{equation}
The expectation and covariance are taken under
$\phi_{\mathsf G,t,q}$.
If $F$ is increasing, then
\begin{equation}
\label{eq:derivative-lower}
\frac{d}{dt}\mathbb EF
\geq\frac{1}{\mathfrak C_{\mathsf G}(t)}
\sum_{e\in\mathsf E}\operatorname{Cov}(F,\omega_e).
\end{equation}
\end{lemma}

\begin{proof}[Proof of \cref{lem:weighted-derivative}]
Write
\[
w_t(\omega):=q^{k(\omega)}
\prod_{e\in\mathsf E}
\bigl(e^{t\mathcal J_e}-1\bigr)^{\omega_e}
\]
for the unnormalized weight in
\eqref{eq:finite-RC-parametrization}.
For $t>0$, logarithmic differentiation gives
\[
\frac{d}{dt}\log w_t(\omega)
=\sum_{e\in\mathsf E}
\omega_e\frac{\mathcal J_e e^{t\mathcal J_e}}
                  {e^{t\mathcal J_e}-1}.
\]
Since $Z_{\mathsf G,t,q}=\sum_\omega w_t(\omega)$ and
there are finitely many configurations,
\[
\frac{d}{dt}\log Z_{\mathsf G,t,q}
=\sum_{e\in\mathsf E}
\frac{\mathcal J_e e^{t\mathcal J_e}}
     {e^{t\mathcal J_e}-1}\mathbb E\omega_e.
\]
Differentiating
$\phi_{\mathsf G,t,q}(\omega)=w_t(\omega)/Z_{\mathsf G,t,q}$
therefore yields
\[
\frac{d}{dt}\phi_{\mathsf G,t,q}(\omega)
=\phi_{\mathsf G,t,q}(\omega)
\sum_{e\in\mathsf E}
\frac{\mathcal J_e e^{t\mathcal J_e}}
     {e^{t\mathcal J_e}-1}
\bigl(\omega_e-\mathbb E\omega_e\bigr).
\]
Multiply by $F(\omega)$ and sum over $\omega$.
The identity
$\mathbb E[F(\omega)(\omega_e-\mathbb E\omega_e)]
=\operatorname{Cov}(F,\omega_e)$ then gives
\eqref{eq:exact-derivative}.

If $F$ is increasing, the FKG inequality for $q\geq1$ gives
$\operatorname{Cov}(F,\omega_e)\geq0$ for every edge
\citep{Grimmett2006}.
By \eqref{eq:finite-chi-C},
\[
\frac{\mathcal J_e e^{t\mathcal J_e}}
     {e^{t\mathcal J_e}-1}
\geq\frac{\mathcal J_e}{e^{t\mathcal J_e}-1}
\geq\frac{1}{\mathfrak C_{\mathsf G}(t)}.
\]
Substituting this lower bound for each coefficient in
\eqref{eq:exact-derivative} proves
\eqref{eq:derivative-lower}.
\end{proof}

The following lemma uses the ghost-field argument of
\citet[Proposition~3.1]{Hutchcroft2020} to bound the covariance sum in
\eqref{eq:derivative-lower} for
$F=\mathbf1_{\{|C(v)|\geq n\}}$.

\begin{lemma}[Ghost-field covariance inequality]
\label{lem:ghost-covariance}
For every $v\in\mathsf V$, $t\geq0$, and
$1\leq n\leq|\mathsf V|$,
\begin{equation}
\label{eq:ghost-covariance}
\sum_{e\in\mathsf E}
\operatorname{Cov}\bigl(\mathbf1_{\{|C(v)|\geq n\}},\omega_e\bigr)
\geq\frac12\left(\frac{n}{\chi_{\mathsf G}(t)}-1\right)
\phi_{\mathsf G,t,q}(|C(v)|\geq n).
\end{equation}
The covariance is taken under $\phi_{\mathsf G,t,q}$.
\end{lemma}

\begin{proof}[Proof of \cref{lem:ghost-covariance}]
If $t=0$, every edge is closed and $|C(v)|=1$.
Thus $\chi_{\mathsf G}(0)=1$, and both sides of
\eqref{eq:ghost-covariance} are zero.
We now fix $t>0$ and set
\begin{equation}
\label{eq:ghost-tail-indicator}
F:=\mathbf1_{\{|C(v)|\geq n\}}.
\end{equation}

Fix $\lambda>0$ and sample $\omega$ with law
$\phi_{\mathsf G,t,q}$.
Independently of $\omega$, assign to each vertex $u$ a
Bernoulli variable $\eta_u$ with success probability
$1-e^{-\lambda/n}$.
The variables $(\eta_u)_{u\in\mathsf V}$ are independent.
In this proof, $\mathbb P$, $\mathbb E$, and
$\operatorname{Cov}$ refer to the joint law of $\omega$ and
$\eta$.
For random variables depending only on $\omega$, these agree
with probability, expectation, and covariance under
$\phi_{\mathsf G,t,q}$.

This auxiliary family is called the ghost field.
Call a vertex $u$ green if $\eta_u=1$, and let
\[
B:=\{C(v)\text{ contains at least one green vertex}\}.
\]
Conditional on $\omega$, the probability that none of the
vertices in $C(v)$ is green is $e^{-\lambda|C(v)|/n}$.
Therefore
\begin{align}
\mathbb P(B\mid\omega)
&=1-e^{-\lambda|C(v)|/n},\notag\\
M_\lambda:=\mathbb P(B)
&=\mathbb E\bigl[1-e^{-\lambda|C(v)|/n}\bigr].
\label{eq:ghost-magnetization}
\end{align}
The joint law is invariant under the graph automorphisms
that preserve the edge weights.
Hence the probability that $C(u)$ contains a green vertex
is $M_\lambda$ for every $u\in\mathsf V$.

Fix an ordering of the edges.
For each $u\in\mathsf V$, define a decision tree $T_u$ as follows.
First reveal $\eta_u$, and stop if $\eta_u=0$.
If $\eta_u=1$, explore the open cluster of $u$.
Start with $u$ as the only discovered vertex.
At each step, reveal the first unrevealed edge with a
discovered endpoint.
If this edge is open, add its other endpoint to the
discovered vertices.
Stop when no unrevealed edge has a discovered endpoint.
The discovered vertices then form $C(u)$.

The collection $\mathcal T=(T_u)_{u\in\mathsf V}$ is a
decision forest.
It determines $\mathbf1_B$: the event $B$ occurs exactly when
an exploration from a green vertex discovers $v$.
For a coordinate $i\in\mathsf E\sqcup\mathsf V$, let
$\delta_i(\mathcal T)$ be the probability that at least one
of the trees reveals that coordinate.
An edge coordinate is $\omega_e$, and a vertex coordinate is
$\eta_u$.
If an edge $e=\{x,y\}$ is revealed, an exploration from a
green vertex has reached $x$ or $y$ along open edges.
Consequently, $C(x)$ or $C(y)$ contains a green vertex.
The union bound gives
\begin{align}
\delta_e(\mathcal T)
&\leq\mathbb P(C(x)\text{ contains a green vertex})\notag\\
&\quad+\mathbb P(C(y)\text{ contains a green vertex})
=2M_\lambda.
\label{eq:ghost-revealment}
\end{align}

For $q\geq1$, the FK law and its product with the independent
ghost law are monotonic \citep[Sections~3.1--3.2]{Grimmett2006}.
Here monotonicity means that the conditional probability
of a coordinate being $1$ cannot decrease when the values
of the other coordinates increase.
Both $F$ and $\mathbf1_B$ are increasing and take values in
$\{0,1\}$.
Also, $F$ depends only on $\omega$, so independence gives
$\operatorname{Cov}(F,\eta_u)=0$ for every $u\in\mathsf V$.
The Boolean form of the decision-forest OSSS inequality
\citep[Corollary~2.4 and (2.3)]{Hutchcroft2020} gives
\begin{align}
\operatorname{Cov}(F,\mathbf1_B)
&\leq\sum_{e\in\mathsf E}\delta_e(\mathcal T)
\operatorname{Cov}(F,\omega_e)
+\sum_{u\in\mathsf V}\delta_u(\mathcal T)
\operatorname{Cov}(F,\eta_u)\notag\\
&\leq2M_\lambda\sum_{e\in\mathsf E}
\operatorname{Cov}(F,\omega_e).
\label{eq:OSSS-ghost}
\end{align}
The second inequality uses \eqref{eq:ghost-revealment} and
$\operatorname{Cov}(F,\omega_e)\geq0$, which follows from FKG.

We also obtain a lower bound for
$\operatorname{Cov}(F,\mathbf1_B)$.
On the event $\{F=1\}$, we have $|C(v)|\geq n$, and hence
$1-e^{-\lambda|C(v)|/n}\geq1-e^{-\lambda}$.
Using \eqref{eq:ghost-magnetization} gives
\begin{align}
\operatorname{Cov}(F,\mathbf1_B)
&=\mathbb E\bigl[F(1-e^{-\lambda|C(v)|/n})\bigr]
-M_\lambda\mathbb EF\notag\\
&\geq(1-e^{-\lambda}-M_\lambda)
\phi_{\mathsf G,t,q}(|C(v)|\geq n).
\label{eq:ghost-lower-cov}
\end{align}
Since $|C(v)|\geq1$,
$M_\lambda\geq1-e^{-\lambda/n}>0$.
Combining \eqref{eq:OSSS-ghost} and
\eqref{eq:ghost-lower-cov}, and dividing by $2M_\lambda$, yields
\begin{equation}
\label{eq:ghost-before-limit}
\sum_{e\in\mathsf E}\operatorname{Cov}(F,\omega_e)
\geq\frac12\left(\frac{1-e^{-\lambda}}{M_\lambda}-1\right)
\phi_{\mathsf G,t,q}(|C(v)|\geq n).
\end{equation}

Finally,
\[
0\leq\frac{1-e^{-\lambda|C(v)|/n}}{\lambda}
\leq\frac{|C(v)|}{n}\leq\frac{|\mathsf V|}{n}.
\]
For each edge configuration,
$(1-e^{-\lambda|C(v)|/n})/\lambda$ converges to $|C(v)|/n$
as $\lambda\downarrow0$.
The law $\phi_{\mathsf G,t,q}$ of $\omega$ does not depend on
$\lambda$.
Dominated convergence in \eqref{eq:ghost-magnetization}
therefore gives
\[
\frac{M_\lambda}{\lambda}\longrightarrow
\frac{\chi_{\mathsf G}(t)}{n},
\qquad
\frac{1-e^{-\lambda}}{\lambda}\longrightarrow1.
\]
Letting $\lambda\downarrow0$ in
\eqref{eq:ghost-before-limit} proves
\eqref{eq:ghost-covariance}.
\end{proof}

Combining \cref{lem:weighted-derivative,lem:ghost-covariance}
gives the cluster-size differential inequality.

\begin{lemma}[Cluster-size differential inequality]
\label{prop:cluster-tail-differential}
For every $v\in\mathsf V$, $t>0$, and
$1\leq n\leq|\mathsf V|$ such that
$\phi_{\mathsf G,t,q}(|C(v)|\geq n)>0$,
\[
\mathfrak C_{\mathsf G}(t)\frac{d}{dt}
\log\phi_{\mathsf G,t,q}(|C(v)|\geq n)
\geq\frac12\left(\frac{n}{\chi_{\mathsf G}(t)}-1\right).
\]
\end{lemma}

\begin{proof}[Proof of \cref{prop:cluster-tail-differential}]
Set $F:=\mathbf1_{\{|C(v)|\geq n\}}$.
Applying \eqref{eq:derivative-lower} to $F$ and then using
\eqref{eq:ghost-covariance} gives
\begin{align*}
\frac{d}{dt}\phi_{\mathsf G,t,q}(|C(v)|\geq n)
&\geq\frac{1}{\mathfrak C_{\mathsf G}(t)}
\sum_{e\in\mathsf E}\operatorname{Cov}(F,\omega_e)\\
&\geq\frac{1}{2\mathfrak C_{\mathsf G}(t)}
\left(\frac{n}{\chi_{\mathsf G}(t)}-1\right)
\phi_{\mathsf G,t,q}(|C(v)|\geq n).
\end{align*}
The covariance is taken under $\phi_{\mathsf G,t,q}$.
Dividing by $\phi_{\mathsf G,t,q}(|C(v)|\geq n)>0$ and
multiplying by $\mathfrak C_{\mathsf G}(t)$ proves
\cref{prop:cluster-tail-differential}.
\end{proof}

Integrating the inequality in
\cref{prop:cluster-tail-differential} gives a cluster-size
tail bound in terms of $\chi_{\mathsf G}(t_0)$ at a larger
parameter $t_0$.

\begin{lemma}[Integrated cluster-size tail bound]
\label{lem:integrated-tail}
Let $0\leq t<t_0$.
For every $v\in\mathsf V$ and $n\geq1$,
\begin{equation}
\label{eq:integrated-tail}
\phi_{\mathsf G,t,q}(|C(v)|\geq n)
\leq
\exp\!\left\{\frac{t_0-t}{2\mathfrak C_{\mathsf G}(t_0)}\right\}
\exp\!\left\{-\frac{(t_0-t)n}
{2\mathfrak C_{\mathsf G}(t_0)\chi_{\mathsf G}(t_0)}\right\}.
\end{equation}
\end{lemma}

\begin{proof}[Proof of \cref{lem:integrated-tail}]
For $u\in[0,t_0]$, write
\begin{equation}
\label{eq:cluster-tail-probability}
P_u(n):=\phi_{\mathsf G,u,q}(|C(v)|\geq n).
\end{equation}
If $P_t(n)=0$, then \eqref{eq:integrated-tail} holds.
If $n\leq\chi_{\mathsf G}(t_0)$, then
\[
\exp\!\left\{
\frac{t_0-t}{2\mathfrak C_{\mathsf G}(t_0)}
\left(1-\frac{n}{\chi_{\mathsf G}(t_0)}\right)
\right\}\geq1.
\]
The bound $P_t(n)\leq1$ therefore implies
\eqref{eq:integrated-tail}.

It remains to consider
$P_t(n)>0$ and $n>\chi_{\mathsf G}(t_0)\geq1$.
The positivity of $P_t(n)$ implies $n\leq|\mathsf V|$.
At $t=0$, every cluster has size one, so $P_0(n)=0$ for
$n>1$.
Thus $t>0$ in the remaining case.

Both $\mathbf1_{\{|C(v)|\geq n\}}$ and $|C(v)|$ are
increasing functions of the edge configuration.
By \eqref{eq:exact-derivative} and FKG
\citep[Section~3.2]{Grimmett2006},
$P_u(n)$ and $\chi_{\mathsf G}(u)$ are nondecreasing in
$u>0$.
In particular, $P_u(n)\geq P_t(n)>0$ for every
$u\in[t,t_0]$.
The definition in \eqref{eq:finite-chi-C} also shows that
$\mathfrak C_{\mathsf G}(u)$ is positive and nondecreasing
for $u>0$.
Hence, for $u\in[t,t_0]$,
\[
\frac{1}{\mathfrak C_{\mathsf G}(u)}
\geq\frac{1}{\mathfrak C_{\mathsf G}(t_0)},
\qquad
\frac{n}{\chi_{\mathsf G}(u)}-1
\geq\frac{n}{\chi_{\mathsf G}(t_0)}-1>0.
\]
Applying \cref{prop:cluster-tail-differential} gives
\begin{align}
\frac{d}{du}\log P_u(n)
&\geq\frac{1}{2\mathfrak C_{\mathsf G}(u)}
\left(\frac{n}{\chi_{\mathsf G}(u)}-1\right)\notag\\
&\geq\frac{1}{2\mathfrak C_{\mathsf G}(t_0)}
\left(\frac{n}{\chi_{\mathsf G}(t_0)}-1\right).
\label{eq:uniform-log-tail-derivative}
\end{align}
Integrating \eqref{eq:uniform-log-tail-derivative} over
$[t,t_0]$ and using $P_{t_0}(n)\leq1$, we obtain
\begin{align*}
\log P_t(n)
&\leq\log P_{t_0}(n)
-\frac{t_0-t}{2\mathfrak C_{\mathsf G}(t_0)}
\left(\frac{n}{\chi_{\mathsf G}(t_0)}-1\right)\\
&\leq\frac{t_0-t}{2\mathfrak C_{\mathsf G}(t_0)}
-\frac{(t_0-t)n}
{2\mathfrak C_{\mathsf G}(t_0)\chi_{\mathsf G}(t_0)}.
\end{align*}
Exponentiating proves \eqref{eq:integrated-tail}.
\end{proof}

We now apply \cref{lem:integrated-tail} to prove
\cref{prop:periodic-cluster-tail}.

\begin{proof}[Proof of \cref{prop:periodic-cluster-tail}]
Fix $0\leq\beta<\beta_0<\beta_{\mathrm c}(J)$ and set
\begin{equation}
\label{eq:periodic-tail-constants}
\begin{aligned}
\delta&:=\beta_0-\beta,
&\qquad \mathfrak C_0&:=\beta_0e^{2\beta_0\|J\|_1},\\
A_0&:=\exp\!\left\{\frac{\delta}{2\beta_0}\right\},
& a_0&:=\frac{\delta}{2\mathfrak C_0\chi_J(\beta_0)}.
\end{aligned}
\end{equation}
Since $\beta_0<\beta_{\mathrm c}(J)$, we have
$1\leq\chi_J(\beta_0)<\infty$ by
\eqref{eq:free-susceptibility-critical-point}.
Thus $A_0$ and $a_0$ are positive and finite.

Fix an integer $L\geq1$.
If $E_L=\varnothing$ or $\beta=0$, every open cluster is a
singleton.
The probability in \eqref{eq:periodic-volume-tail-new}
is then $1$ for $m=1$ and $0$ for every integer $m\geq2$.
Since $\mathfrak C_0\geq\beta_0$ and $\chi_J(\beta_0)\geq1$,
\[
a_0\leq\frac{\delta}{2\beta_0}=\log A_0,
\]
so $A_0e^{-a_0}\geq1$.
This proves \eqref{eq:periodic-volume-tail-new} when
$E_L=\varnothing$ or $\beta=0$.
We now assume $E_L\ne\varnothing$ and $\beta>0$.

Let $\mathsf G_L=(\mathbb T_L^d,E_L)$, with $E_L$ defined in
\eqref{eq:periodic-FK-edge-set}, and assign the edge weights
\begin{equation}
\label{eq:torus-tail-edge-weights}
\mathcal J_{\{\bar x,\bar y\}}
:=2J_L(\bar x-\bar y),
\qquad \{\bar x,\bar y\}\in E_L.
\end{equation}
The weights $\mathcal J_e$ are positive for $e\in E_L$.
Torus translations preserve the edge set and the weights,
and act transitively on the vertices.
For every $t\geq0$ and every $\{\bar x,\bar y\}\in E_L$,
\[
1-e^{-t\mathcal J_{\{\bar x,\bar y\}}}
=1-e^{-2tJ_L(\bar x-\bar y)}.
\]
Comparing \eqref{eq:finite-RC-parametrization} with
\eqref{eq:periodic-FK-edge-parameters} and
\eqref{eq:periodic-FK-measure} gives
\[
\phi_{\mathsf G_L,t,2}=\phi^{\mathrm{per}}_{L,t,2},
\qquad t\geq0.
\]
We use $q=2$ in the definition of
$\chi_{\mathsf G_L}$.

By \eqref{eq:periodized-J-new} and the nonnegativity of $J$,
\[
0<\mathcal J_e\leq2\|J\|_1,
\qquad e\in E_L.
\]
For every $e\in E_L$,
\[
\beta_0
\leq\frac{e^{\beta_0\mathcal J_e}-1}{\mathcal J_e}
=\int_0^{\beta_0}e^{u\mathcal J_e}\,du
\leq\beta_0e^{2\beta_0\|J\|_1}.
\]
Taking the maximum over $e\in E_L$ yields
\begin{equation}
\label{eq:C-uniform-bound}
\beta_0\leq\mathfrak C_{\mathsf G_L}(\beta_0)
\leq\beta_0e^{2\beta_0\|J\|_1}
=\mathfrak C_0.
\end{equation}
At parameter $\beta_0$, \eqref{eq:periodic-susceptibility-FK}
and \eqref{eq:periodic-susceptibility-new} give
\begin{equation}
\label{eq:chi-beta0-uniform}
\chi_{\mathsf G_L}(\beta_0)
=\mathbb E|C_L(\bar0)|
\leq\chi_J(\beta_0)<\infty.
\end{equation}
The expectation in \eqref{eq:chi-beta0-uniform} is taken
under $\phi^{\mathrm{per}}_{L,\beta_0,2}$.

The bounds in \eqref{eq:C-uniform-bound} and
\eqref{eq:chi-beta0-uniform} imply
\begin{align*}
\frac{\delta}{2\mathfrak C_{\mathsf G_L}(\beta_0)}
&\leq\frac{\delta}{2\beta_0}=\log A_0,\\
\frac{\delta}
{2\mathfrak C_{\mathsf G_L}(\beta_0)\chi_{\mathsf G_L}(\beta_0)}
&\geq\frac{\delta}{2\mathfrak C_0\chi_J(\beta_0)}=a_0.
\end{align*}
Apply \eqref{eq:integrated-tail} on $\mathsf G_L$ with
$t=\beta$, $t_0=\beta_0$, $q=2$, $v=\bar0$, and $n=m$.
For every integer $m\geq1$, we obtain
\begin{align*}
\phi^{\mathrm{per}}_{L,\beta,2}
\bigl(|C_L(\bar0)|\geq m\bigr)
&\leq
\exp\!\left\{\frac{\delta}{2\mathfrak C_{\mathsf G_L}(\beta_0)}\right\}
\exp\!\left\{-\frac{\delta m}
{2\mathfrak C_{\mathsf G_L}(\beta_0)\chi_{\mathsf G_L}(\beta_0)}\right\}\\
&\leq A_0e^{-a_0m}.
\end{align*}
The constants $A_0$ and $a_0$ in
\eqref{eq:periodic-tail-constants} depend only on
$\beta,\beta_0,J$ and are independent of $L$.
This proves \eqref{eq:periodic-volume-tail-new}.
\end{proof}

\subsubsection{Proof of the finite-range unfolding estimate}
\label{subsubsec:periodic-unfolding}

For the zero-field nearest-neighbor ferromagnet,
\citet[Proposition~1.7]{LiuPanisSlade2025} proved
\eqref{eq:finite-range-unfolding} for $L\geq3$ and $\beta>0$.
We adapt their exploration and switching argument
\citep[Sections~3.2--3.3]{LiuPanisSlade2025}
to the finite-range interaction $J^{(R)}$.
We first construct a lattice graph and a torus graph that retain
the displacement of each edge as a label.
Different displacements can produce torus edges with the same
endpoints, and the labels keep those edges distinct.

Fix $\beta\geq0$, $R\geq1$, and an integer $L\geq1$.
Use the interactions $J^{(R)}$ and $J_L^{(R)}$ defined in
\eqref{eq:periodized-truncation}.
Choose
\[
\mathcal D_R\subset
\{z\in\mathbb Z^d\setminus\{0\}:J^{(R)}(z)>0\}
\]
containing exactly one element of each pair $\{z,-z\}$
for which $J^{(R)}(z)>0$, and define
\begin{equation}
\label{eq:cover-displacement-set}
\mathcal D_{R,L}:=\mathcal D_R\setminus L\mathbb Z^d.
\end{equation}
The displacements in $L\mathbb Z^d$ would produce torus edges
whose endpoints coincide.
We omit them because the torus measure in
\eqref{eq:periodic-measure-new} includes only interactions
between distinct sites.

For an edge $e$, write $\operatorname{end}(e)$ for its
unordered pair of endpoints.
Let $\widetilde H_{R,L}$ have vertex set $\mathbb Z^d$.
For every $y\in\mathbb Z^d$ and $z\in\mathcal D_{R,L}$,
introduce an edge $e(y,z)$ with label $(y,z)$ and define
\begin{equation}
\label{eq:cover-edge-definition}
\operatorname{end}(e(y,z)):=\{y,y+z\},
\qquad
K_{e(y,z)}:=\beta J(z).
\end{equation}
Let $H_{R,L}$ have vertex set $\mathbb T_L^d$.
For every $\bar y\in\mathbb T_L^d$ and
$z\in\mathcal D_{R,L}$, introduce an edge
$\bar e(\bar y,z)$ with label $(\bar y,z)$ and define
\begin{equation}
\label{eq:quotient-edge-definition}
\operatorname{end}(\bar e(\bar y,z))
:=\{\bar y,\bar y+\bar z\},
\qquad
K_{\bar e(\bar y,z)}:=\beta J(z).
\end{equation}
Edges with different labels remain distinct even if their
unordered endpoint pairs agree.
Thus $H_{R,L}$ is a multigraph: it may have several edges
between the same two vertices.
Neither graph has loops.

Define the projection on vertices and edges by
\begin{equation}
\label{eq:periodic-labeled-cover}
\pi_L(y):=\bar y,
\qquad
\pi_L(e(y,z)):=\bar e(\bar y,z).
\end{equation}
The edges incident to $y$ are
\[
\{e(y,z),e(y-z,z):z\in\mathcal D_{R,L}\},
\]
and the edges incident to $\bar y$ are
\[
\{\bar e(\bar y,z),\bar e(\bar y-\bar z,z):
z\in\mathcal D_{R,L}\}.
\]
The map $\pi_L$ gives a bijection between the edges incident
to $y$ and the edges incident to $\bar y$.
It also preserves the edge couplings:
$K_e=K_{\pi_L(e)}$.
We call $\pi_L$ a coupling-preserving graph covering.

We next check that $H_{R,L}$ gives the Ising measure with
interaction $J_L^{(R)}$.
Fix distinct $\bar a,\bar b\in\mathbb T_L^d$ and choose
representatives $a,b\in\mathbb Z^d$.
Each displacement $w=b-a+Lu$ with $J^{(R)}(w)>0$
corresponds to exactly one edge between $\bar a$ and $\bar b$:
it is $\bar e(\bar a,w)$ if $w\in\mathcal D_R$, and
$\bar e(\bar b,-w)$ if $-w\in\mathcal D_R$.
Every edge between $\bar a$ and $\bar b$ arises in this way.
Consequently,
\begin{equation}
\label{eq:quotient-coupling-identity}
\begin{aligned}
\sum_{\substack{\bar e\in E(H_{R,L})\\
\operatorname{end}(\bar e)=\{\bar a,\bar b\}}}
K_{\bar e}
&=\beta\sum_{u\in\mathbb Z^d}J^{(R)}(b-a+Lu)\\
&=\beta J_L^{(R)}(\bar b-\bar a).
\end{aligned}
\end{equation}
All edges between $\bar a$ and $\bar b$ contribute to the
coefficient of $\sigma_{\bar a}\sigma_{\bar b}$ in the
Ising exponent.
Equation~\eqref{eq:quotient-coupling-identity} therefore
shows that the zero-field Ising measure on $H_{R,L}$ is
exactly the truncated torus measure whose expectation is
$\langle\cdot\rangle_{L,\beta}^{\mathrm{per},(R)}$.

We now define random currents on a finite labeled multigraph
$H=(V_H,E_H)$, where $V_H$ is the vertex set and $E_H$ is
the edge set.
Assume that $H$ has no loops and has couplings $K_e\geq0$.
Parallel edges are treated as distinct elements of $E_H$.
A current assigns a nonnegative integer $n(e)$ to each edge:
$n:E_H\to\mathbb Z_{\geq0}$.
Its weight, support, and source set are
\begin{equation}
\label{eq:current-definitions}
\begin{aligned}
w_H(n)&:=\prod_{e\in E_H}\frac{K_e^{n(e)}}{n(e)!},
\qquad
\supp n:=\{e\in E_H:n(e)>0\},\\
\partial n&:=\left\{v\in V_H:
\sum_{e\ni v}n(e)\equiv1\pmod2\right\}.
\end{aligned}
\end{equation}
Thus a source is a vertex where the sum of the incident
current values is odd.
We use $0^0=1$ in the weight, so an edge with $K_e=0$
must have $n(e)=0$ for the current to have positive weight.
For $A\subset V_H$, define
\begin{equation}
\label{eq:current-partition-sum}
Z_H^A:=\sum_{\partial n=A}w_H(n).
\end{equation}
Since every edge has two endpoints, $|\partial n|$ is even.
Hence $Z_H^A=0$ if $|A|$ is odd.
The zero current has weight one, so $Z_H^\varnothing\geq1$.

Write $\langle\cdot\rangle_H$ for expectation under the
zero-field Ising measure on $H$ with couplings $(K_e)$,
with all spins summed over $\{-1,+1\}^{V_H}$.
For an edge $e$ with endpoints $u,v$, expand
\[
e^{K_e\sigma_u\sigma_v}
=\sum_{n(e)=0}^{\infty}
\frac{K_e^{n(e)}}{n(e)!}(\sigma_u\sigma_v)^{n(e)}.
\]
Multiply the edge expansions and the spin product
$\prod_{v\in A}\sigma_v$.
For a fixed current $n$, the sum over all spin configurations
is zero unless $\partial n=A$; when $\partial n=A$, the
spin sum is $2^{|V_H|}$.
Applying the same expansion to the normalizing constant gives
the random-current identity \citep{Aizenman1982}
\begin{equation}
\label{eq:current-correlation}
\left\langle\prod_{v\in A}\sigma_v\right\rangle_H
=\frac{Z_H^A}{Z_H^\varnothing}.
\end{equation}
For two currents $m,n$ on $H$, a vertex is a source of
$m+n$ exactly when it is a source of one of $m,n$ and not
the other.
Thus
\begin{equation}
\label{eq:source-symmetric-difference}
\partial(m+n)=\partial m\mathbin{\triangle}\partial n,
\end{equation}
where $\triangle$ denotes symmetric difference.
For a current $k$ and vertices $o,z\in V_H$, write
$o\xleftrightarrow{k}z$ when $o$ and $z$ are connected by
a path of edges $e$ with $k(e)>0$.

\begin{lemma}[Ordinary source switching]
\label{lem:ordinary-switching}
Let $A,D\subset V_H$ have even cardinality, and let
$o,z\in V_H$ be distinct.
For every nonnegative function
$F:\mathbb Z_{\geq0}^{E_H}\to[0,\infty)$,
\begin{align}
&\sum_{\substack{\partial m=A\\\partial n=D}}
w_H(m)w_H(n)F(m+n)
\mathbf1_{\{o\xleftrightarrow{m+n}z\}}
\notag\\
&\quad=
\sum_{\substack{\partial m=A\triangle\{o,z\}\\
\partial n=D\triangle\{o,z\}}}
w_H(m)w_H(n)F(m+n)
\mathbf1_{\{o\xleftrightarrow{m+n}z\}}.
\label{eq:ordinary-switching}
\end{align}
\end{lemma}

\begin{proof}[Proof of \cref{lem:ordinary-switching}]
All summands in \eqref{eq:ordinary-switching} are nonnegative,
so we may group both sums by the total current $k=m+n$.
If $o$ and $z$ are not connected in $(V_H,\supp k)$,
the connectivity indicator is zero.
Fix a current $k$ for which $o\xleftrightarrow{k}z$.

For every edge $e$, represent $k(e)$ by $k(e)$ distinguishable
units, numbered $1,\ldots,k(e)$.
Color every unit red or blue, and let $m(e)$ and $n(e)$
be the numbers of red and blue units on $e$, respectively.
A fixed pair $(m,n)$ with $m+n=k$ is produced by exactly
\[
\prod_{e\in E_H}\binom{k(e)}{n(e)}
\]
colorings.
By \eqref{eq:current-definitions},
\[
w_H(m)w_H(n)
=w_H(k)\prod_{e\in E_H}\binom{k(e)}{n(e)}.
\]
Assign weight $w_H(k)$ to each coloring.
The sum of the coloring weights over all colorings that
produce $(m,n)$ is then exactly $w_H(m)w_H(n)$.

Choose a simple path from $o$ to $z$ in $(V_H,\supp k)$
by a rule that depends only on $k$.
On every edge of this path, reverse the color of unit $1$.
Let $m'$ and $n'$ be the resulting red and blue currents.
Each interior vertex of the path is incident to two
recolored units, so the parity of its incident red total
and the parity of its incident blue total are unchanged.
At each of $o$ and $z$, exactly one unit is recolored,
so both parities change.
No unit incident to a vertex outside the path is recolored.
Consequently,
\[
\partial m'=\partial m\mathbin{\triangle}\{o,z\},
\qquad
\partial n'=\partial n\mathbin{\triangle}\{o,z\}.
\]

The total current remains $k$.
The chosen path and its units numbered $1$ therefore remain
the same after recoloring.
Applying the color change twice returns the original coloring.
We thus obtain a bijection between colorings with source sets
$(A,D)$ and colorings with source sets
$(A\triangle\{o,z\},D\triangle\{o,z\})$.
Every coloring has the same factor $w_H(k)F(k)$ before
and after the color change.
Hence the contributions from the fixed current $k$ to
the two sides of \eqref{eq:ordinary-switching} are equal.
Summing over $k$ proves \eqref{eq:ordinary-switching}.
\end{proof}

We next use an exploration to apply
\cref{lem:ordinary-switching} to currents on two different graphs.
Let $\pi:\widetilde H\to H$ be a coupling-preserving covering
of labeled graphs without loops.
The graph $H$ is finite, and $\widetilde H$ may be infinite.
The edge couplings are nonnegative.
Parallel edges retain their distinct labels.
At every vertex $v\in V(\widetilde H)$, the map $\pi$
gives a bijection from the edges incident to $v$ onto the
edges incident to $\pi(v)$, with $K_e=K_{\pi(e)}$.
A vertex $v$ is a lift of $\bar v$ if $\pi(v)=\bar v$.
An edge $e$ is a lift of $\bar e$ if $\pi(e)=\bar e$.

Fix $o\in V(\widetilde H)$, write $\bar o:=\pi(o)$,
and set $n_H:=|V(H)|$.
Let $B$ be a finite subgraph of $\widetilde H$ containing
the subgraph induced by
\begin{equation}
\label{eq:exploration-cover-ball}
\mathbb B_{\widetilde H}(o,n_H)
:=\{v\in V(\widetilde H):
\operatorname{dist}_{\widetilde H}(o,v)\leq n_H\}.
\end{equation}
Here graph distance is the minimum number of edges in a path
between the two vertices.
The distance is $\infty$ if no such path exists.
For a current $m$ on $B$ and a current $n$ on $H$, define
\begin{equation}
\label{eq:lift-current}
\widehat n(e):=n(\pi(e)),
\qquad e\in E(B),
\end{equation}
and
\begin{equation}
\label{eq:exploration-open-edge}
\omega_e(m,n)
:=\mathbf1_{\{m(e)+\widehat n(e)>0\}},
\qquad e\in E(B).
\end{equation}
An edge is open when $\omega_e(m,n)=1$ and closed when
$\omega_e(m,n)=0$.
Equation~\eqref{eq:lift-current} copies $n(\bar e)$ to
every edge $e\in E(B)$ with $\pi(e)=\bar e$.
The map $n\mapsto\widehat n$ need not preserve current
weights or source sets.
We use $\widehat n$ only to define the edge states in
\eqref{eq:exploration-open-edge}.

The exploration selects a connected subgraph of $B$.
It keeps at most one lift of each vertex of $H$ and queries
at most one lift of each labeled edge of $H$.
Fix a total order on $V(B)$ and a total order on the labeled
edges of $B$ incident to each vertex.
The orders do not depend on $m$ or $n$.
Each vertex has one of four states: unseen, active, processed,
or gray.
Active and processed vertices are called selected.
Gray vertices cannot be selected.
Once a vertex is gray or processed, its state is permanent.

\begin{enumerate}
\item Declare $o$ active and every other lift of $\bar o$
in $B$ gray.
All remaining vertices are unseen, and no edge has been queried.

\item Let $v$ be the first active vertex in the fixed order.
Scan the edges of $B$ incident to $v$ in their fixed order.
For each edge $e$ in this scan, write
$\operatorname{end}(e)=\{v,w\}$.
Skip $e$ if $w$ is gray or if an edge $e'\in E(B)$
with $\pi(e')=\pi(e)$ has already been queried.
Previous queries include both open and closed edges.
Otherwise, $e$ is eligible, and we query $\omega_e(m,n)$.

If $\omega_e(m,n)=0$, record $e$ as tested closed and
leave all vertex states unchanged.
If $\omega_e(m,n)=1$, record $e$ as tested open.
If $\omega_e(m,n)=1$ and no selected vertex projects to
$\pi(w)$, declare $w$ active and every other still-unseen
lift of $\pi(w)$ in $B$ gray.
An open edge is recorded even when $w$ is already selected.
After scanning all incident edges, declare $v$ processed.

\item If any active vertices remain, return to step~2.
Otherwise, stop.
\end{enumerate}

At termination, let $\mathcal V_o(m,n)$ be the set of
processed vertices, and let $\mathcal E_{\mathrm o}(m,n)$
and $\mathcal E_{\mathrm c}(m,n)$ be the sets of tested
open and tested closed edges, respectively.
Define the explored graph by
\begin{equation}
\label{eq:exploration-output}
\mathcal X_o(m,n)
:=\bigl(\mathcal V_o(m,n),\mathcal E_{\mathrm o}(m,n)\bigr).
\end{equation}
The vertex $o$ belongs to $\mathcal V_o(m,n)$ even if
no edge is tested open.
The tested closed edges are recorded in
$\mathcal E_{\mathrm c}(m,n)$ but are not edges of
$\mathcal X_o(m,n)$.
We omit $(m,n)$ when the currents have been fixed.
A pair $(E_{\mathrm o},E_{\mathrm c})$ of subsets of
$E(B)$ is an admissible exploration output if
\[
(E_{\mathrm o},E_{\mathrm c})
=\bigl(\mathcal E_{\mathrm o}(m,n),
\mathcal E_{\mathrm c}(m,n)\bigr)
\]
for some currents $m$ on $B$ and $n$ on $H$.

A concrete run of the algorithm is shown in
\Cref{fig:cover-exploration-example}.  Take the nearest-neighbor cover
$\Z\to\Z/5\Z$, and let $B$ be the subgraph induced by
$[-5,5]\cap\Z$.  Let $n$ carry one current unit on each edge of the
quotient path
\[
 \bar0-\bar4-\bar3-\bar2,
\]
and no current on the other two quotient edges, and let $m$ carry
one unit on $\{0,1\}$ and vanish on every other edge of $B$.  Choose the
incident-edge order at $0$ so that $\{0,1\}$ is queried before
$\{-1,0\}$.

Starting from $o=0$, the first two queries, both of which reveal open edges, select $1$ as the
representative of $\bar1$ and $-1$ as the representative of $\bar4$.
Choose the vertex order so that $1$ is processed before $-1$.  The query of
$\{1,2\}$ then reveals that this edge is closed.  The exploration subsequently follows the two
open lifts $\{-2,-1\}$ and $\{-3,-2\}$, thereby selecting the
representatives $-2$ of $\bar3$ and $-3$ of $\bar2$.  At that point $2$ and
$-4$ are gray, because their quotient vertices are already represented.  In
particular, the competing lift $\{-4,-3\}$ of the quotient edge
$\{\bar1,\bar2\}$ is not queried.

\begin{figure}[htbp]
\centering
\begin{tikzpicture}[
  selected/.style={
    circle,draw=black,fill=black,text=white,
    minimum size=5.6mm,inner sep=0pt,font=\footnotesize},
  active/.style={
    circle,draw=black,double,
    minimum size=5.6mm,inner sep=0pt,font=\footnotesize},
  pending/.style={
    circle,draw=black,
    minimum size=5.6mm,inner sep=0pt,font=\footnotesize},
  gray/.style={
    circle,draw=black!45,fill=black!12,text=black!55,
    minimum size=5.6mm,inner sep=0pt,font=\footnotesize},
  quotient/.style={
    circle,draw=black,
    minimum size=5.6mm,inner sep=0pt,font=\footnotesize},
  testedopen/.style={draw=black,line width=1.15pt},
  testedclosed/.style={draw=black,dashed,line width=.9pt},
  unqueried/.style={draw=black!45,densely dotted,line width=.8pt},
  projection/.style={draw=black!45,densely dashed},
  every node/.style={font=\footnotesize}
]
\node at (-3.65,1.55) {\textnormal{(a) After processing \(o=0\)}};
\node[pending]  (am2) at (-5.65,.55) {$-2$};
\node[active]   (am1) at (-4.65,.55) {$-1$};
\node[selected] (a0)  at (-3.65,.55) {$0$};
\node[active]   (a1)  at (-2.65,.55) {$1$};
\node[pending]  (a2)  at (-1.65,.55) {$2$};
\draw[unqueried]  (am2)--(am1);
\draw[testedopen] (am1)--node[above,fill=white,inner sep=1pt]
  {\scriptsize 2} (a0);
\draw[testedopen] (a0)--node[above,fill=white,inner sep=1pt]
  {\scriptsize 1} (a1);
\draw[unqueried]  (a1)--(a2);

\node at (3.1,1.55)
  {\textnormal{(b) Final explored graph and projection}};
\node[gray]     (fm4) at (.1,.55) {$-4$};
\node[selected] (fm3) at (1.1,.55) {$-3$};
\node[selected] (fm2) at (2.1,.55) {$-2$};
\node[selected] (fm1) at (3.1,.55) {$-1$};
\node[selected] (f0)  at (4.1,.55) {$0$};
\node[selected] (f1)  at (5.1,.55) {$1$};
\node[gray]     (f2)  at (6.1,.55) {$2$};
\draw[unqueried] (fm4)--(fm3);
\draw[testedopen] (fm3)--node[above,fill=white,inner sep=1pt]
  {\scriptsize 5} (fm2);
\draw[testedopen] (fm2)--node[above,fill=white,inner sep=1pt]
  {\scriptsize 4} (fm1);
\draw[testedopen] (fm1)--node[above,fill=white,inner sep=1pt]
  {\scriptsize 2} (f0);
\draw[testedopen] (f0)--node[above,fill=white,inner sep=1pt]
  {\scriptsize 1} (f1);
\draw[testedclosed] (f1)--node[above,fill=white,inner sep=1pt]
  {\scriptsize 3} (f2);

\node[quotient] (q2) at (1.1,-1.05) {$\bar2$};
\node[quotient] (q3) at (2.1,-1.05) {$\bar3$};
\node[quotient] (q4) at (3.1,-1.05) {$\bar4$};
\node[quotient] (q0) at (4.1,-1.05) {$\bar0$};
\node[quotient] (q1) at (5.1,-1.05) {$\bar1$};
\draw[projection] (fm3)--(q2);
\draw[projection] (fm2)--(q3);
\draw[projection] (fm1)--(q4);
\draw[projection] (f0)--(q0);
\draw[projection] (f1)--(q1);
\draw[testedopen] (q2)--(q3)--(q4)--(q0)--(q1);
\draw[testedclosed] (q1) to[bend left=32] (q2);

\draw[testedopen] (-4.8,-2.05)--(-4.1,-2.05);
\node[anchor=west,font=\scriptsize] at (-4.0,-2.05) {tested open};
\draw[testedclosed] (-1.35,-2.05)--(-.65,-2.05);
\node[anchor=west,font=\scriptsize] at (-.55,-2.05) {tested closed};
\node[gray] at (2.5,-2.05) {\strut};
\node[anchor=west,font=\scriptsize] at (2.88,-2.05) {competing lift};
\end{tikzpicture}
\caption{An exploration over a five-cycle.  The numbers give the query
order.  Thick edges are tested open, dashed edges are tested closed, and
dotted edges have not been queried in the displayed state.
In (a), double circles mark active vertices, and the black vertex
has been processed.  In (b), the black vertices form the final graph $\mathcal X_o$, while graying discards
competing lifts.  The projection lines show that each quotient
vertex is represented once and each quotient edge is queried once.}
\label{fig:cover-exploration-example}
\end{figure}
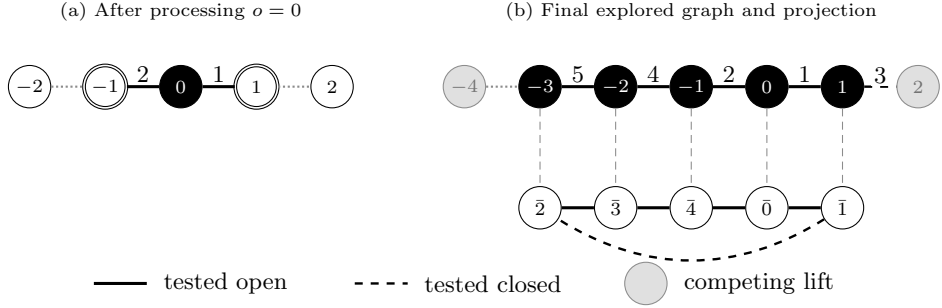

We now establish the properties of the exploration needed for
\cref{lem:mixed-switching}.

\begin{lemma}[Properties of the exploration]
\label{lem:cover-exploration-interfaces}
The exploration terminates and has the following properties.
{\renewcommand{\labelenumi}{\textnormal{(X\arabic{enumi})}}
\begin{enumerate}
\item $\mathcal X_o$ is connected, contains $o$, and contains
at most one lift of each vertex of $H$.

\item The map $\pi$ is injective on
$\mathcal E_{\mathrm o}\cup\mathcal E_{\mathrm c}$.
Its restriction to $\mathcal X_o$ is a coupling-preserving
labeled graph isomorphism onto $\pi(\mathcal X_o)$.
For every admissible output $(E_{\mathrm o},E_{\mathrm c})$,
\begin{equation}
\label{eq:exploration-measurability}
\begin{aligned}
&\{\mathcal E_{\mathrm o}=E_{\mathrm o},
\mathcal E_{\mathrm c}=E_{\mathrm c}\}\\
&\qquad=
\{\omega_e=1\ (e\in E_{\mathrm o}),
\omega_e=0\ (e\in E_{\mathrm c})\}.
\end{aligned}
\end{equation}

\item For every $e\in\mathcal E_{\mathrm c}$,
\[
n(\pi(e))=0.
\]

\item If $\partial n=\{\bar o,\bar x\}$ with
$\bar x\ne\bar o$, then $\mathcal X_o$ contains exactly
one vertex $x^\star$ satisfying $\pi(x^\star)=\bar x$.
\end{enumerate}}
\end{lemma}

\begin{proof}[Proof of \cref{lem:cover-exploration-interfaces}]
\emph{Termination and (X1).}
Initially, $o$ is the only selected vertex.
A new vertex is selected only if its image under $\pi$
has no selected representative.
All other unseen lifts of that image are then made gray,
and gray vertices cannot be selected.
Thus there is at most one selected lift of each vertex
of $H$ at every stage.
At most $n_H$ vertices can therefore be selected.
Every processing step changes one active vertex to processed,
and a processed vertex never becomes active again.
Since $B$ is finite, every scan examines finitely many edges.
The exploration therefore terminates after at most $n_H$
processing steps.

Every selected vertex other than $o$ is reached by a tested
open edge from a previously selected vertex.
At termination, all selected vertices are processed.
Hence $\mathcal X_o$ is connected, contains $o$, and
contains at most one lift of each vertex of $H$.

We also record why the choice of $B$ includes every edge
that may need to be scanned.
Each $v\in\mathcal V_o$ can be joined to $o$ by a simple
path in $\mathcal X_o$.
The path has at most $n_H-1$ edges, since its vertices
have distinct images in $H$.
Thus
\[
\operatorname{dist}_{\widetilde H}(o,v)\leq n_H-1,
\qquad v\in\mathcal V_o.
\]
If an edge $e$ of $\widetilde H$ has endpoints $v,w$
with $v\in\mathcal V_o$, then
$\operatorname{dist}_{\widetilde H}(o,w)\leq n_H$.
Both endpoints of $e$ lie in
$\mathbb B_{\widetilde H}(o,n_H)$.
Since $B$ contains the subgraph induced by
$\mathbb B_{\widetilde H}(o,n_H)$, we have $e\in E(B)$.

\emph{Property (X2).}
Once an edge has been queried, the algorithm skips every
edge with the same image under $\pi$.
Thus $\pi$ is injective on
$\mathcal E_{\mathrm o}\cup\mathcal E_{\mathrm c}$.
Property (X1) gives injectivity on $\mathcal V_o$.
Since the covering preserves endpoints and couplings,
$\pi$ restricts to a coupling-preserving labeled graph
isomorphism from $\mathcal X_o$ onto $\pi(\mathcal X_o)$.

Fix an admissible output $(E_{\mathrm o},E_{\mathrm c})$.
If the exploration produces this output, it has found
$\omega_e=1$ for $e\in E_{\mathrm o}$ and
$\omega_e=0$ for $e\in E_{\mathrm c}$.
This proves the forward inclusion in
\eqref{eq:exploration-measurability}.
For the reverse inclusion, choose a pair of currents whose
exploration produces $(E_{\mathrm o},E_{\mathrm c})$.
Such a pair exists by admissibility.
Consider another pair of currents with the same edge states
on $E_{\mathrm o}\cup E_{\mathrm c}$.
The two runs start from the same vertex states and use
the same vertex and edge orders.
Whenever their previous queries and answers agree,
their vertex states and query records agree.
They therefore skip the same edges and make the same
next query, or both stop.
The next queried edge belongs to
$E_{\mathrm o}\cup E_{\mathrm c}$, so its answer also agrees.
Induction over the queries shows that both runs produce
$(E_{\mathrm o},E_{\mathrm c})$.
This proves \eqref{eq:exploration-measurability}.

\emph{Property (X3).}
For $e\in\mathcal E_{\mathrm c}$,
\eqref{eq:exploration-open-edge} and \eqref{eq:lift-current}
give
\[
m(e)+n(\pi(e))=0.
\]
Both terms are nonnegative, so $n(\pi(e))=0$.

\emph{Property (X4).}
Suppose that $\partial n=\{\bar o,\bar x\}$ and
$\bar x\ne\bar o$.
Count every edge $e$ with multiplicity $n(e)$.
The handshaking lemma shows that each connected component
of $(V(H),\supp n)$ contains an even number of sources.
Thus $\bar o$ and $\bar x$ lie in the same component.
Choose a simple labeled path
\[
\bar o=\bar v_0,\bar e_0,\bar v_1,\ldots,
\bar e_{\ell-1},\bar v_\ell=\bar x
\]
in $(V(H),\supp n)$, where
$\operatorname{end}(\bar e_i)=\{\bar v_i,\bar v_{i+1}\}$
and $n(\bar e_i)>0$.

We prove by induction on $i$ that a lift of $\bar v_i$
is selected.
The lift $o$ of $\bar v_0$ is selected at initialization.
Suppose that a lift $v_i$ of $\bar v_i$ is selected.
The covering property gives a unique edge $e_i$ incident
to $v_i$ with $\pi(e_i)=\bar e_i$.
Let $w_i$ be its other endpoint, so
$\pi(w_i)=\bar v_{i+1}$.
The distance argument in the proof of (X1) gives
$e_i\in E(B)$.

Consider the scan of $e_i$ when $v_i$ is processed.
If a lift of $\bar v_{i+1}$ has already been selected,
there is nothing to prove for this induction step.
Otherwise, $w_i$ is unseen: it is not selected, and
a gray vertex always has a selected lift of the same
vertex of $H$.
If a lift of $\bar e_i$ had been queried earlier,
the query would have been open, since $n(\bar e_i)>0$.
After an open query, both projected endpoints have
selected representatives, including $\bar v_{i+1}$.
This contradicts the assumption that $\bar v_{i+1}$
has no selected lift.
Thus no lift of $\bar e_i$ has been queried earlier.
The edge $e_i$ is therefore eligible and open, and
the algorithm selects $w_i$.
Induction gives a selected lift of $\bar x$.
Property (X1) gives its uniqueness.
\end{proof}

We next use \eqref{eq:exploration-measurability} to identify
changes to the currents that preserve a specified exploration
output.
Fix an admissible output $(E_{\mathrm o},E_{\mathrm c})$
and write $\bar E_{\mathrm o}:=\pi(E_{\mathrm o})$.
Let $m_1,m_2$ be currents on $B$, and let $n_1,n_2$
be currents on $H$.
Suppose that
\begin{gather*}
m_1(e)=m_2(e),
\qquad e\in E(B)\setminus E_{\mathrm o},\\
n_1(\bar e)=n_2(\bar e),
\qquad \bar e\in E(H)\setminus\bar E_{\mathrm o},\\
m_1(e)+n_1(\pi(e))=m_2(e)+n_2(\pi(e)),
\qquad e\in E_{\mathrm o}.
\end{gather*}
The equality of the current sums on $E_{\mathrm o}$ gives
$\omega_e(m_1,n_1)=\omega_e(m_2,n_2)$ for
$e\in E_{\mathrm o}$.
Since the output is admissible, (X2) gives
$\pi(E_{\mathrm c})\cap\bar E_{\mathrm o}=\varnothing$.
For $e\in E_{\mathrm c}$, we have $e\notin E_{\mathrm o}$
and $\pi(e)\notin\bar E_{\mathrm o}$.
The equalities outside $E_{\mathrm o}$ and
$\bar E_{\mathrm o}$ therefore give
$\omega_e(m_1,n_1)=\omega_e(m_2,n_2)$ for
$e\in E_{\mathrm c}$ as well.
By \eqref{eq:exploration-measurability},
$(m_1,n_1)$ produces $(E_{\mathrm o},E_{\mathrm c})$
if and only if $(m_2,n_2)$ produces
$(E_{\mathrm o},E_{\mathrm c})$.

For $C\subset V(B)$, $A\subset V(H)$, and a nonnegative
function $F$ of a current $m$ on $B$ and a current $n$ on $H$,
define
\begin{equation}
\label{eq:mixed-current-functional}
\mathcal Z_{B,H}^{C,A}[F]
:=\sum_{\substack{\partial m=C\\\partial n=A}}
w_B(m)w_H(n)F(m,n).
\end{equation}
When $F\equiv1$, the two current sums separate:
\begin{equation}
\label{eq:product-current-factorization}
\mathcal Z_{B,H}^{C,A}[1]=Z_B^C Z_H^A.
\end{equation}

\begin{lemma}[Mixed switching on an explored lift]
\label{lem:mixed-switching}
Let $C\subset V(B)$ and $A\subset V(H)$ have even
cardinality.
For every $z\in V(B)$ with $z\ne o$ and $\pi(z)\ne\bar o$,
\begin{equation}
\label{eq:mixed-switching}
\begin{aligned}
&\mathcal Z_{B,H}^{C,A}
\bigl[\mathbf1_{\{z\in\mathcal V_o\}}\bigr]\\
&\qquad=
\mathcal Z_{B,H}^{C\triangle\{o,z\},
                    A\triangle\{\bar o,\pi(z)\}}
\bigl[\mathbf1_{\{z\in\mathcal V_o\}}\bigr].
\end{aligned}
\end{equation}
\end{lemma}

\begin{proof}[Proof of \cref{lem:mixed-switching}]
We group the current pairs by their exploration output
and by their currents outside the explored open edges.
For each group, we apply \cref{lem:ordinary-switching}
on the explored graph.

\emph{The graph determined by a fixed output.}
Fix an admissible output $(E_{\mathrm o},E_{\mathrm c})$
whose explored vertex set contains $z$.
Write
\[
E_0:=E_{\mathrm o},
\qquad
V_0:=\{o\}\cup\bigcup_{e\in E_0}\operatorname{end}(e),
\qquad
\bar E_0:=\pi(E_0).
\]
The set $V_0$ is the explored vertex set for this output.
By (X1) and (X2) of
\cref{lem:cover-exploration-interfaces}, the graph
$\mathsf G_0$ below is connected and contains $o$ and $z$,
and
\begin{equation}
\label{eq:explored-chart-isomorphism}
\begin{aligned}
\mathsf G_0&:=(V_0,E_0),
&\overline{\mathsf G}_0&:=(\pi(V_0),\bar E_0),\\
\pi|_{\mathsf G_0}&:\mathsf G_0
\longrightarrow\overline{\mathsf G}_0
\end{aligned}
\end{equation}
is a coupling-preserving labeled graph isomorphism.

Let $m_{\mathrm{in}}$ and $m_{\mathrm{out}}$ be the
restrictions of $m$ to $E_0$ and $E(B)\setminus E_0$.
Let $n_{\mathrm{in}}$ and $n_{\mathrm{out}}$ be the
restrictions of $n$ to $\bar E_0$ and
$E(H)\setminus\bar E_0$.
Extend each restriction by zero to its ambient edge set.
Then
\begin{equation}
\label{eq:current-inside-outside}
m=m_{\mathrm{in}}+m_{\mathrm{out}},
\qquad
n=n_{\mathrm{in}}+n_{\mathrm{out}}.
\end{equation}
Define a current on $\mathsf G_0$ by
\[
\widetilde n_{\mathrm{in}}(e)
:=n_{\mathrm{in}}(\pi(e)),
\qquad e\in E_0.
\]
The isomorphism in \eqref{eq:explored-chart-isomorphism}
preserves couplings and incident edges, so
\begin{equation}
\label{eq:transported-current-data}
\begin{aligned}
w_{\mathsf G_0}(\widetilde n_{\mathrm{in}})
&=w_{\overline{\mathsf G}_0}(n_{\mathrm{in}}),\\
\pi(\partial\widetilde n_{\mathrm{in}})
&=\partial n_{\mathrm{in}}.
\end{aligned}
\end{equation}

\emph{Fixing the exterior currents.}
Fix $m_{\mathrm{out}}$ and $n_{\mathrm{out}}$.
Under the source conditions $\partial m=C$ and
$\partial n=A$, \eqref{eq:source-symmetric-difference}
gives
\begin{equation}
\label{eq:interior-source-sets}
\begin{aligned}
\partial m_{\mathrm{in}}
&=C_0:=C\triangle\partial m_{\mathrm{out}},\\
\partial n_{\mathrm{in}}
&=A_0:=A\triangle\partial n_{\mathrm{out}}.
\end{aligned}
\end{equation}
If $C_0\not\subset V_0$ or $A_0\not\subset\pi(V_0)$,
no interior currents can satisfy these source conditions.
Both contributions to \eqref{eq:mixed-switching} are
then zero: toggling $\{o,z\}$ and
$\{\bar o,\pi(z)\}$ does not remove any source outside
$V_0$ or $\pi(V_0)$.
We may therefore assume that
$C_0\subset V_0$ and $A_0\subset\pi(V_0)$.
The sets $C_0$ and $A_0$ have even cardinality, since
$C$, $A$, and the source set of every current have even
cardinality.
Set
\[
D_0:=\{v\in V_0:\pi(v)\in A_0\}.
\]
The bijection $\pi:V_0\to\pi(V_0)$ implies that $D_0$
has even cardinality and
$\partial\widetilde n_{\mathrm{in}}=D_0$.

With the zero extensions in
\eqref{eq:current-inside-outside}, set
\[
W_{\mathrm{out}}:=w_B(m_{\mathrm{out}})
w_H(n_{\mathrm{out}}).
\]
The edgewise weight formula in
\eqref{eq:current-definitions} and
\eqref{eq:transported-current-data} give
\begin{align*}
w_B(m)w_H(n)
&=W_{\mathrm{out}}
w_{\mathsf G_0}(m_{\mathrm{in}})
w_{\overline{\mathsf G}_0}(n_{\mathrm{in}})\\
&=W_{\mathrm{out}}
w_{\mathsf G_0}(m_{\mathrm{in}})
w_{\mathsf G_0}(\widetilde n_{\mathrm{in}}).
\end{align*}
\emph{Applying ordinary source switching.}
Set
\begin{equation}
\label{eq:interior-total-current}
k_{\mathrm{in}}
:=m_{\mathrm{in}}+\widetilde n_{\mathrm{in}}.
\end{equation}
For $e\in E_0$, \eqref{eq:exploration-open-edge} gives
$\omega_e(m,n)=\mathbf1_{\{k_{\mathrm{in}}(e)>0\}}$.
By (X2) of \cref{lem:cover-exploration-interfaces},
$\pi(E_{\mathrm c})\cap\bar E_0=\varnothing$.
Thus the edge states on $E_{\mathrm c}$ are fixed by
$m_{\mathrm{out}}$ and $n_{\mathrm{out}}$.
Equation~\eqref{eq:exploration-measurability} therefore shows that
the indicator of producing $(E_{\mathrm o},E_{\mathrm c})$
depends on the interior currents only through $k_{\mathrm{in}}$.
Denote this indicator by $F_0(k_{\mathrm{in}})$.
If $F_0(k)=1$, then $k(e)>0$ for every $e\in E_0$.
Since $\mathsf G_0$ is connected and contains $o$ and $z$,
\[
F_0(k)=F_0(k)\mathbf1_{\{o\xleftrightarrow{k}z\}}
\]
for every current $k$ on $\mathsf G_0$.
Apply \eqref{eq:ordinary-switching} on $\mathsf G_0$
with source sets $C_0,D_0$ and function $F_0$.
Writing $p,q$ for currents on $\mathsf G_0$, we obtain
\begin{align*}
&\sum_{\substack{\partial p=C_0\\\partial q=D_0}}
w_{\mathsf G_0}(p)w_{\mathsf G_0}(q)F_0(p+q)\\
&\qquad=
\sum_{\substack{\partial p=C_0\triangle\{o,z\}\\
\partial q=D_0\triangle\{o,z\}}}
w_{\mathsf G_0}(p)w_{\mathsf G_0}(q)F_0(p+q).
\end{align*}
Under $\pi:V_0\to\pi(V_0)$, the source change
$D_0\mapsto D_0\triangle\{o,z\}$ becomes
$A_0\mapsto A_0\triangle\{\bar o,\pi(z)\}$.
After the fixed exterior currents are added back,
the full source sets are therefore changed from
$(C,A)$ to
$(C\triangle\{o,z\},A\triangle\{\bar o,\pi(z)\})$.

Multiply the equality of interior sums by
$W_{\mathrm{out}}$ and sum over all exterior currents.
Then sum over all admissible outputs whose vertex sets
contain $z$.
Every current pair with $z\in\mathcal V_o$ belongs to
exactly one of these outputs.
All summands are nonnegative, so the regrouping is valid.
The resulting identity is \eqref{eq:mixed-switching}.
\end{proof}

We now apply \cref{lem:mixed-switching} to prove
\cref{prop:finite-range-unfolding}.

\begin{proof}[Proof of \cref{prop:finite-range-unfolding}]
If $\bar x=\bar0$, then $x\in L\mathbb Z^d$.
The term $u=-x/L$ on the right-hand side of
\eqref{eq:finite-range-unfolding} is
$\tau_\beta^{(R)}(0)=1$, which equals the left-hand side.
All other terms are nonnegative, so the inequality holds.
If $\beta=0$ and $\bar x\ne\bar0$, the spins are
independent with mean zero, and both sides vanish.
We may therefore assume that $\beta>0$ and
$\bar x\ne\bar0$.

\emph{Choosing a finite graph for the exploration.}
Use the covering
$\pi_L:\widetilde H_{R,L}\to H_{R,L}$ defined in
\eqref{eq:periodic-labeled-cover}.
Choose an integer $M$ such that
\begin{equation}
\label{eq:cover-box-choice}
M>(L^{d}+1)R,
\qquad
\Lambda_M:=[-M,M]^d\cap\mathbb Z^d,
\end{equation}
and let $\mathsf G_{M,R,L}$ be the subgraph of
$\widetilde H_{R,L}$ induced by $\Lambda_M$.
Every edge of $\widetilde H_{R,L}$ has displacement
of $\ell^\infty$ norm at most $R$.
Hence
\[
\mathbb B_{\widetilde H_{R,L}}(0,L^{d})
\subseteq[-L^{d}R,L^{d}R]^d\cap\mathbb Z^d
\subseteq\Lambda_M.
\]
Since $\mathsf G_{M,R,L}$ is an induced subgraph, it
contains the subgraph induced by
$\mathbb B_{\widetilde H_{R,L}}(0,L^{d})$.
Thus the exploration defined using
\eqref{eq:exploration-cover-ball} applies with
\[
B=\mathsf G_{M,R,L},
\qquad H=H_{R,L},
\qquad \pi=\pi_L,
\qquad o=0.
\]

\emph{Applying mixed switching.}
Let $m$ be a current on $\mathsf G_{M,R,L}$ and
$n$ a current on $H_{R,L}$.
If $\partial n=\{\bar0,\bar x\}$, property (X4) of
\cref{lem:cover-exploration-interfaces} gives exactly
one explored vertex whose projection is $\bar x$.
Consequently,
\begin{equation}
\label{eq:explored-lift-partition}
\sum_{\substack{y\in\Lambda_M\\\pi_L(y)=\bar x}}
\mathbf1_{\{y\in\mathcal V_0\}}=1.
\end{equation}
Multiply \eqref{eq:explored-lift-partition} by
$w_{\mathsf G_{M,R,L}}(m)w_{H_{R,L}}(n)$ and sum over
$\partial m=\varnothing$ and
$\partial n=\{\bar0,\bar x\}$.
For every $y\in\Lambda_M$ with $\pi_L(y)=\bar x$,
\eqref{eq:mixed-switching} applies with
$C=\varnothing$, $A=\{\bar0,\bar x\}$, and $z=y$.
Using \eqref{eq:product-current-factorization} and
bounding the exploration indicators by one gives
\begin{align*}
Z_{\mathsf G_{M,R,L}}^\varnothing
Z_{H_{R,L}}^{\{\bar0,\bar x\}}
&=\sum_{\substack{y\in\Lambda_M\\\pi_L(y)=\bar x}}
\mathcal Z_{\mathsf G_{M,R,L},H_{R,L}}
^{\varnothing,\{\bar0,\bar x\}}
\bigl[\mathbf1_{\{y\in\mathcal V_0\}}\bigr]\\
&=\sum_{\substack{y\in\Lambda_M\\\pi_L(y)=\bar x}}
\mathcal Z_{\mathsf G_{M,R,L},H_{R,L}}
^{\{0,y\},\varnothing}
\bigl[\mathbf1_{\{y\in\mathcal V_0\}}\bigr]\\
&\leq\sum_{\substack{y\in\Lambda_M\\\pi_L(y)=\bar x}}
Z_{\mathsf G_{M,R,L}}^{\{0,y\}}
Z_{H_{R,L}}^\varnothing.
\end{align*}
The current weights are nonnegative, so the order of
summation may be exchanged.
Divide by the positive quantity
$Z_{\mathsf G_{M,R,L}}^\varnothing Z_{H_{R,L}}^\varnothing$.
By \eqref{eq:current-correlation} and
\eqref{eq:quotient-coupling-identity},
\begin{equation}
\label{eq:finite-box-unfolding}
\bigl\langle\sigma_{\bar0}\sigma_{\bar x}\bigr\rangle_
{L,\beta}^{\mathrm{per},(R)}
\leq
\sum_{\substack{y\in\Lambda_M\\\pi_L(y)=\bar x}}
\bigl\langle\sigma_0\sigma_y\bigr\rangle_{\mathsf G_{M,R,L}}.
\end{equation}

\emph{Comparing with the lattice two-point function.}
The Ising model on $\mathsf G_{M,R,L}$ is obtained from
the zero-field free-boundary model on $\Lambda_M$ with
interaction $J^{(R)}$ by deleting the couplings whose
displacements lie in $L\mathbb Z^d$.
In the FK representation, this deletion is equivalent to
conditioning the edges with displacements in
$L\mathbb Z^d$ to be closed.
Positive association and \eqref{eq:two-point-connectivity}
show that deleting these couplings cannot increase
the two-point function.
Free-boundary domain monotonicity, proved in the proof of
\cref{prop:critical-equivalence}, applies to $J^{(R)}$.
Enlarging the finite volume and taking the free-boundary
limit therefore gives, for every $y\in\Lambda_M$,
\begin{equation}
\label{eq:cover-correlation-comparison}
\bigl\langle\sigma_0\sigma_y\bigr\rangle_{\mathsf G_{M,R,L}}
\leq\mathbb E[\sigma_0\sigma_y]
\leq\tau_\beta^{(R)}(y).
\end{equation}
Here $\mathbb E$ is taken under the zero-field
free-boundary Gibbs measure on $\Lambda_M$ with
interaction $J^{(R)}$ and inverse temperature $\beta$.
The limiting function $\tau_\beta^{(R)}$ is defined in
\eqref{eq:truncated-two-point-functions}.

Combining \eqref{eq:finite-box-unfolding} and
\eqref{eq:cover-correlation-comparison} yields
\[
\bigl\langle\sigma_{\bar0}\sigma_{\bar x}\bigr\rangle_
{L,\beta}^{\mathrm{per},(R)}
\leq
\sum_{\substack{y\in\Lambda_M\\\pi_L(y)=\bar x}}
\tau_\beta^{(R)}(y)
\leq
\sum_{u\in\mathbb Z^d}\tau_\beta^{(R)}(x+Lu).
\]
The last inequality uses nonnegativity and the identity
$\{y\in\mathbb Z^d:\pi_L(y)=\bar x\}
=\{x+Lu:u\in\mathbb Z^d\}$.
This proves \eqref{eq:finite-range-unfolding}.
\end{proof}

\subsection{Infinite-volume ferromagnetic model}
\label{sec:infinite-long-range}

We first prove a limit lemma for the total spin in a fixed finite
set. We will also use the lemma for the nearest-neighbor model.

\begin{lemma}[Passing to infinite volume on a fixed observation set]
\label{lem:fixed-set-limit}
Fix a nonempty finite set $\Lambda\subset\mathbb Z^d$.
Let $S_\Lambda^{(m)}$ and $S_\Lambda$ denote total spins in
$\Lambda$ under possibly different probability measures.
Assume that
\[
S_\Lambda^{(m)}\xrightarrow[m\to\infty]{\mathrm d}S_\Lambda,
\qquad
\mathbb E S_\Lambda^{(m)}=0\quad(m\geq1).
\]
Write
\[
s_{\Lambda,m}:=\sqrt{\operatorname{Var}(S_\Lambda^{(m)})},
\qquad
s_\Lambda:=\sqrt{\operatorname{Var}(S_\Lambda)}.
\]
Suppose that $s_{\Lambda,m}^2\geq v|\Lambda|$ for every $m\geq1$,
where $v>0$ is independent of $m$.
Then
\[
\mathbb E S_\Lambda=0,
\qquad
s_{\Lambda,m}^2\longrightarrow s_\Lambda^2\geq v|\Lambda|>0.
\]
If a constant $0\leq C<\infty$ satisfies
\begin{equation}
\label{eq:fixed-set-prelimit-bounds}
\begin{aligned}
d_{\mathrm K}\!\left(
\frac{S_\Lambda^{(m)}}{s_{\Lambda,m}},G\right)
&\leq C|\Lambda|^{-1/2},\quad
d_{\mathrm W}\!\left(
\frac{S_\Lambda^{(m)}}{s_{\Lambda,m}},G\right)
\leq C|\Lambda|^{-1/2}, \qquad(m\geq1),
\end{aligned}
\end{equation}
then
\begin{equation}
\label{eq:fixed-set-limit-bounds}
\begin{aligned}
d_{\mathrm K}\!\left(\frac{S_\Lambda}{s_\Lambda},G\right)
&\leq C|\Lambda|^{-1/2},\quad
d_{\mathrm W}\!\left(\frac{S_\Lambda}{s_\Lambda},G\right)
\leq C|\Lambda|^{-1/2}.
\end{aligned}
\end{equation}
\end{lemma}

\begin{proof}[Proof of \cref{lem:fixed-set-limit}]
The total spins take values in the same finite set
\[
\mathcal T_\Lambda:=\{-|\Lambda|,-|\Lambda|+2,\ldots,|\Lambda|\}.
\]
Convergence in distribution therefore gives
\[
\mathbb P(S_\Lambda^{(m)}=t)
\longrightarrow\mathbb P(S_\Lambda=t),
\qquad t\in\mathcal T_\Lambda.
\]
The first and second moments converge, so
$\mathbb E S_\Lambda=0$ and
$s_{\Lambda,m}^2\to s_\Lambda^2\geq v|\Lambda|>0$.
Consequently,
\[
\frac{S_\Lambda^{(m)}}{s_{\Lambda,m}}
\xrightarrow[m\to\infty]{\mathrm d}
\frac{S_\Lambda}{s_\Lambda}.
\]

Let $F_m$ and $F$ be the distribution functions of
$S_\Lambda^{(m)}/s_{\Lambda,m}$ and $S_\Lambda/s_\Lambda$,
respectively.
At every continuity point $t$ of $F$, we have
$F_m(t)\to F(t)$.
The Kolmogorov bound in \eqref{eq:fixed-set-prelimit-bounds}
therefore gives
\[
|F(t)-\Phi(t)|\leq C|\Lambda|^{-1/2}.
\]
For an arbitrary $t\in\mathbb R$, choose continuity points
$t_j$ of $F$ with $t_j\downarrow t$.
The right continuity of $F$ and the continuity of $\Phi$
give the same inequality at $t$.
Taking the supremum over $t$ proves the Kolmogorov bound in
\eqref{eq:fixed-set-limit-bounds}.

For the Wasserstein distance, it suffices to consider
$1$-Lipschitz functions $f$ with $f(0)=0$.
Indeed, subtracting $f(0)$ does not change the difference
of expectations.
For $t\in\mathcal T_\Lambda$, we have
\begin{align*}
|f(t/s_{\Lambda,m})|
&\leq\frac{|\Lambda|}{s_{\Lambda,m}}\leq\sqrt{|\Lambda|/v},\\
|f(t/s_{\Lambda,m})-f(t/s_\Lambda)|
&\leq |\Lambda|\left|s_{\Lambda,m}^{-1}-s_\Lambda^{-1}\right|.
\end{align*}
Expanding the expectations over $\mathcal T_\Lambda$ gives
\begin{align*}
&\left|\mathbb E f\!\left(\frac{S_\Lambda^{(m)}}{s_{\Lambda,m}}\right)
-\mathbb E f\!\left(\frac{S_\Lambda}{s_\Lambda}\right)\right|\\
&\quad\leq\sqrt{|\Lambda|/v}\sum_{t\in\mathcal T_\Lambda}
\left|\mathbb P(S_\Lambda^{(m)}=t)
-\mathbb P(S_\Lambda=t)\right|\\
&\qquad+|\Lambda|\left|s_{\Lambda,m}^{-1}-s_\Lambda^{-1}\right|.
\end{align*}
The bound is independent of $f$ and tends to zero because
$\mathcal T_\Lambda$ is fixed and $s_{\Lambda,m}\to s_\Lambda>0$.
Taking the supremum over $f$ yields
\[
d_{\mathrm W}\!\left(
\frac{S_\Lambda^{(m)}}{s_{\Lambda,m}},
\frac{S_\Lambda}{s_\Lambda}\right)\longrightarrow0.
\]
The triangle inequality and the Wasserstein bound in
\eqref{eq:fixed-set-prelimit-bounds} now prove the Wasserstein
bound in \eqref{eq:fixed-set-limit-bounds}.
\end{proof}

\begin{proof}[Proof of \cref{thm:infinite-long-range-new}]
Let $J$ satisfy \eqref{eq:intro-summable-J}, and fix
$0\leq\beta<\beta_{\mathrm c}(J)$.
Fix a nonempty finite set $\Lambda\subset\mathbb Z^d$.
Choose an integer $m_0\geq1$ such that $\Lambda\subseteq B_{m_0}$,
and set $D_m:=B_{m_0+m}$ for $m\geq1$.
The set $\Lambda$ remains fixed as $D_m$ grows.
Under $\mu^{\mathrm f}_{D_m,\beta}$ with interaction $J$, write
\[
S_\Lambda^{(m)}:=\sum_{x\in\Lambda}\sigma_x,
\qquad
s_{\Lambda,m}:=\sqrt{\operatorname{Var}(S_\Lambda^{(m)})}.
\]
The variance is taken under $\mu^{\mathrm f}_{D_m,\beta}$.
Spin-flip symmetry gives $\mathbb E S_\Lambda^{(m)}=0$.

Use the Edwards--Sokal coupling on the graph $(D_m,E_{D_m})$
from \eqref{eq:finite-free-edge-set}, with
$K_{xy}=\beta J(x-y)$, and take
\[
a_x=\mathbf1_{\{x\in\Lambda\}},\qquad x\in D_m.
\]
Then $T_a=S_\Lambda^{(m)}$, $s_a=s_{\Lambda,m}$,
$A=1$, $H=\Lambda$, and $M=|\Lambda|$.
By \cref{prop:cluster-moments},
\[
\mathbb E|C_{D_m}(x)|^r\leq M_r(\beta,J)<\infty,
\qquad x\in\Lambda,\quad r=2,3.
\]
The expectations are taken under
$\phi^{\mathrm f}_{D_m,\beta,2}$, and $C_{D_m}(x)$ is the
open cluster containing $x$ in the entire graph $(D_m,E_{D_m})$.
The constants $M_2(\beta,J)$ and $M_3(\beta,J)$ are independent
of $m$ and $\Lambda$.

For the weights $a_x=\mathbf1_{\{x\in\Lambda\}}$, we have
\[
V_a=\sum_{C\in\mathcal C(\omega)}|C\cap\Lambda|^2
\geq\sum_{C\in\mathcal C(\omega)}|C\cap\Lambda|=|\Lambda|.
\]
Here $\mathcal C(\omega)$ consists of the open clusters in
$(D_m,E_{D_m})$.
Thus \eqref{eq:ES-conditional-moments} gives
\begin{equation}
\label{eq:infinite-free-local-variance}
s_{\Lambda,m}^2=\mathbb EV_a\geq |\Lambda|.
\end{equation}
The expectation is taken under $\phi^{\mathrm f}_{D_m,\beta,2}$.
Applying \cref{cor:cluster-moment-normal-approximation} with
$b_2=M_2(\beta,J)$, $b_3=M_3(\beta,J)$, and $c=1$ gives
\begin{equation}
\label{eq:infinite-free-prelimit-bounds}
\begin{aligned}
d_{\mathrm K}\!\left(
\frac{S_\Lambda^{(m)}}{s_{\Lambda,m}},G\right)
&\leq C_{\beta,J}^{\infty}|\Lambda|^{-1/2},\\
d_{\mathrm W}\!\left(
\frac{S_\Lambda^{(m)}}{s_{\Lambda,m}},G\right)
&\leq C_{\beta,J}^{\infty}|\Lambda|^{-1/2}.
\end{aligned}
\end{equation}
Both bounds hold under $\mu^{\mathrm f}_{D_m,\beta}$, with
$C_{\beta,J}^{\infty}<\infty$ independent of $m$ and $\Lambda$.

By \eqref{eq:intro-infinite-volume}, the distribution of the
spins in $\Lambda$ under $\mu^{\mathrm f}_{D_m,\beta}$ converges
to their distribution under $\mu^\infty_\beta$ with interaction
$J$. Hence
\[
S_\Lambda^{(m)}\xrightarrow[m\to\infty]{\mathrm d}S_\Lambda.
\]
Applying \cref{lem:fixed-set-limit} with $v=1$, using
\eqref{eq:infinite-free-local-variance} and
\eqref{eq:infinite-free-prelimit-bounds}, proves
\eqref{eq:infinite-long-range-BE-new} and
\eqref{eq:infinite-long-range-W-new}.
\end{proof}

\subsection{Classical nearest-neighbor model}
\label{sec:nearest-neighbor}

Throughout this subsection, $\beta_{\mathrm c}$ denotes
$\beta_{\mathrm c}(J_{\mathrm{nn}})$, as in
\cref{subsec:main-classical}.
We first bound the variance of $S_\Lambda$ under
$\mu^{\mathrm f}_{D,\beta}$ for finite sets
$\varnothing\ne\Lambda\subseteq D$.
The estimate applies both when $D=\Lambda$ and when
$D$ grows with $\Lambda$ fixed.

\begin{lemma}[Variance lower bound]
\label{lem:nearest-neighbor-variance-lower}
Let $D\subset\mathbb Z^d$ be finite, let
$\varnothing\ne\Lambda\subseteq D$, and let $\beta\in\mathbb R$.
Then
\begin{equation}
\label{eq:nearest-neighbor-variance-lower}
\operatorname{Var}(S_\Lambda)
\geq\frac{|\Lambda|}{2}
\operatorname{sech}^2(2d|\beta|)>0.
\end{equation}
Here $\operatorname{sech}(t)=1/\cosh(t)$.
All probabilities, expectations, and variances in this
lemma and its proof are taken under the zero-field
nearest-neighbor Gibbs measure $\mu^{\mathrm f}_{D,\beta}$
with free boundary conditions.
\end{lemma}

\begin{proof}[Proof of \cref{lem:nearest-neighbor-variance-lower}]
Partition $\Lambda$ into the two parity classes
\begin{equation}
\label{eq:parity-classes}
I_j:=\{x\in\Lambda:x_1+\cdots+x_d\equiv j\pmod2\},
\qquad j\in\{0,1\}.
\end{equation}
Choose $I$ to be a class with $|I|\geq|\Lambda|/2$,
and set
\[
\mathcal F:=\sigma(\sigma_y:y\in D\setminus I).
\]
No nearest-neighbor edge joins two vertices of $I$.
After conditioning on $\mathcal F$, every term in the
Gibbs exponent that contains a spin $\sigma_x$ with
$x\in I$ is therefore linear in that spin.
Set
\[
h_x:=\beta\sum_{\substack{y\in D\\\|x-y\|_1=1}}
\sigma_y,
\qquad x\in I.
\]
For every $\xi\in\{-1,+1\}^{I}$, the conditional law is
\[
\mathbb P(\sigma_x=\xi_x\text{ for all }x\in I
\mid\mathcal F)
=\prod_{x\in I}\frac{e^{h_x\xi_x}}{2\cosh(h_x)}.
\]
Thus the spins in $I$ are conditionally independent, and
\[
\mathbb E[\sigma_x\mid\mathcal F]=\tanh(h_x),
\qquad
\operatorname{Var}(\sigma_x\mid\mathcal F)
=1-\tanh^2(h_x)=\operatorname{sech}^2(h_x).
\]
Every vertex has at most $2d$ nearest neighbors, so
$|h_x|\leq2d|\beta|$.
Consequently,
\begin{equation}
\label{eq:single-site-conditional-variance-lower}
\operatorname{Var}(\sigma_x\mid\mathcal F)
\geq\operatorname{sech}^2(2d|\beta|),
\qquad x\in I.
\end{equation}

Conditional on $\mathcal F$, the sum over
$\Lambda\setminus I$ is fixed.
The law of total variance and conditional independence
therefore give
\begin{align*}
\operatorname{Var}(S_\Lambda)
&\geq\mathbb E\bigl[
\operatorname{Var}(S_\Lambda\mid\mathcal F)\bigr]\\
&=\sum_{x\in I}\mathbb E\bigl[
\operatorname{Var}(\sigma_x\mid\mathcal F)\bigr]\\
&\geq |I|\operatorname{sech}^2(2d|\beta|)\\
&\geq\frac{|\Lambda|}{2}
\operatorname{sech}^2(2d|\beta|).
\end{align*}
This proves \eqref{eq:nearest-neighbor-variance-lower}.
\end{proof}

We next prove normal approximation bounds for $S_\Lambda$
under $\mu^{\mathrm f}_{D,\beta}$, with constants independent
of both $D$ and $\Lambda$.

\begin{proposition}[Total spin in a subset of a finite system]
\label{prop:nearest-neighbor-local-approximation}
Fix $|\beta|<\beta_{\mathrm c}$.
Let $D\subset\mathbb Z^d$ be finite, and let
$\varnothing\ne\Lambda\subseteq D$.
Write
\[
s_\Lambda:=\sqrt{\operatorname{Var}(S_\Lambda)}.
\]
The variance is taken under the zero-field nearest-neighbor
Gibbs measure $\mu^{\mathrm f}_{D,\beta}$ with free boundary
conditions.
Spin-flip symmetry gives $\mathbb E S_\Lambda=0$.
Then $s_\Lambda^2>0$, and there exists $C_{\beta,d}<\infty$
such that
\begin{equation}
\label{eq:nearest-neighbor-local-bounds}
\begin{aligned}
d_{\mathrm K}\!\left(\frac{S_\Lambda}{s_\Lambda},G\right)
&\leq C_{\beta,d}|\Lambda|^{-1/2},\\
d_{\mathrm W}\!\left(\frac{S_\Lambda}{s_\Lambda},G\right)
&\leq C_{\beta,d}|\Lambda|^{-1/2}.
\end{aligned}
\end{equation}
Both bounds hold under $\mu^{\mathrm f}_{D,\beta}$.
The constant depends only on $\beta$ and $d$.
\end{proposition}

\begin{proof}[Proof of \cref{prop:nearest-neighbor-local-approximation}]
Set $b:=|\beta|$.
By \cref{lem:nearest-neighbor-variance-lower},
\[
s_\Lambda^2\geq c|\Lambda|,
\qquad
c:=\frac12\operatorname{sech}^2(2db)>0.
\]
For $\beta\geq0$, take
\[
\tau_x:=\sigma_x,
\qquad
a_x:=\mathbf1_{\{x\in\Lambda\}},
\qquad x\in D.
\]
For $\beta<0$, define
\begin{equation}
\label{eq:gauge-transformation}
\eta_x:=(-1)^{x_1+\cdots+x_d},
\qquad
\tau_x:=\eta_x\sigma_x,
\qquad x\in D.
\end{equation}
Nearest neighbors satisfy $\eta_x\eta_y=-1$, so
\[
\beta\sigma_x\sigma_y
=(-b)\eta_x\eta_y\tau_x\tau_y
=b\tau_x\tau_y.
\]
The map $\sigma\mapsto\tau$ is a bijection on
$\{-1,+1\}^{D}$.
For $\beta<0$, take
\[
a_x:=\eta_x\mathbf1_{\{x\in\Lambda\}},\qquad x\in D.
\]
Thus, for both signs of $\beta$, the spins $\tau$ have law
$\mu^{\mathrm f}_{D,b}$, and
\[
T_a=\sum_{x\in D}a_x\tau_x=S_\Lambda,
\qquad
s_a=s_\Lambda,\quad A=1,\quad H=\Lambda,\quad M=|\Lambda|.
\]
Here $s_a^2$ is the variance of $T_a$ under
$\mu^{\mathrm f}_{D,b}$; it equals the variance of
$S_\Lambda$ under $\mu^{\mathrm f}_{D,\beta}$.

Apply the Edwards--Sokal coupling to $\tau$ on $D$, with
nearest-neighbor couplings $K_{xy}=b$.
Since $b<\beta_{\mathrm c}$, \cref{prop:cluster-moments} gives
\[
\mathbb E|C_D(x)|^r\leq M_r(b,J_{\mathrm{nn}})<\infty,
\qquad x\in\Lambda,\quad r=2,3.
\]
The expectations are taken under $\phi^{\mathrm f}_{D,b,2}$.
The cluster $C_D(x)$ contains $x$ and is formed using the
open edges in the entire set $D$.
The constants $M_2(b,J_{\mathrm{nn}})$ and
$M_3(b,J_{\mathrm{nn}})$ are independent of $D$ and $\Lambda$.

Apply \cref{cor:cluster-moment-normal-approximation} with
$b_2=M_2(b,J_{\mathrm{nn}})$, $b_3=M_3(b,J_{\mathrm{nn}})$,
and $c=\frac12\operatorname{sech}^2(2db)$.
This gives the two bounds in
\eqref{eq:nearest-neighbor-local-bounds} for $T_a/s_a$
under $\mu^{\mathrm f}_{D,b}$.
The identity $T_a=S_\Lambda$ transfers both bounds to
$S_\Lambda/s_\Lambda$ under $\mu^{\mathrm f}_{D,\beta}$.
\end{proof}

\subsubsection{Finite-volume Gibbs measures}
\label{subsubsec:proof-classical-finite}

\begin{proof}[Proof of \cref{thm:classical-ising}]
Fix $|\beta|<\beta_{\mathrm c}$ and a nonempty finite set
$\Lambda\subset\mathbb Z^d$.
Apply \cref{prop:nearest-neighbor-local-approximation} with
$D=\Lambda$.
The variance $s_\Lambda^2$ is then taken under
$\mu^{\mathrm f}_{\Lambda,\beta}$, and
$W_\Lambda=S_\Lambda/s_\Lambda$.
Hence $s_\Lambda^2>0$, and
\eqref{eq:nearest-neighbor-local-bounds} gives
\eqref{eq:classical-ising-BE} and
\eqref{eq:classical-ising-W-bound}.
\end{proof}

\subsubsection{Infinite-volume Gibbs measures}
\label{subsubsec:proof-classical-infinite}

\begin{proof}[Proof of
\cref{thm:classical-ising-infinite-wasserstein}]
Fix $|\beta|<\beta_{\mathrm c}$ and a nonempty finite set
$\Lambda\subset\mathbb Z^d$.
Choose an integer $m_0\geq1$ such that
$\Lambda\subseteq B_{m_0}$, and set
$D_m:=B_{m_0+m}$ for $m\geq1$.
The set $\Lambda$ remains fixed as $D_m$ grows.
Under $\mu^{\mathrm f}_{D_m,\beta}$, write
\[
S_\Lambda^{(m)}:=\sum_{x\in\Lambda}\sigma_x,
\qquad
s_{\Lambda,m}:=\sqrt{\operatorname{Var}(S_\Lambda^{(m)})}.
\]
The variance is taken under $\mu^{\mathrm f}_{D_m,\beta}$.
Spin-flip symmetry gives $\mathbb E S_\Lambda^{(m)}=0$.
By \cref{lem:nearest-neighbor-variance-lower} with $D=D_m$,
\[
s_{\Lambda,m}^2\geq v|\Lambda|,
\qquad
v:=\frac12\operatorname{sech}^2(2d|\beta|)>0.
\]
Applying \cref{prop:nearest-neighbor-local-approximation}
with $D=D_m$ gives
\begin{equation}
\label{eq:classical-uniform-prelimit-bounds}
\begin{aligned}
d_{\mathrm K}\!\left(
\frac{S_\Lambda^{(m)}}{s_{\Lambda,m}},G\right)
&\leq C_{\beta,d}|\Lambda|^{-1/2},\\
d_{\mathrm W}\!\left(
\frac{S_\Lambda^{(m)}}{s_{\Lambda,m}},G\right)
&\leq C_{\beta,d}|\Lambda|^{-1/2}.
\end{aligned}
\end{equation}
Both bounds hold under $\mu^{\mathrm f}_{D_m,\beta}$, with
$C_{\beta,d}$ independent of $m$ and $\Lambda$.

For $\beta\geq0$, the distribution of the spins in $\Lambda$
under $\mu^{\mathrm f}_{D_m,\beta}$ converges to their
distribution under $\mu^\infty_\beta$ by
\eqref{eq:intro-infinite-volume}.
For $\beta<0$, the transformation in
\eqref{eq:gauge-transformation} maps
$\mu^{\mathrm f}_{D_m,\beta}$ to
$\mu^{\mathrm f}_{D_m,|\beta|}$.
The signs $\eta_x$ do not depend on $m$.
The free-boundary limit at parameter $|\beta|$ therefore
also gives the free-boundary limit at parameter $\beta$.
Thus, for both signs of $\beta$,
\[
S_\Lambda^{(m)}\xrightarrow[m\to\infty]{\mathrm d}S_\Lambda,
\]
where $S_\Lambda$ is the total spin under $\mu^\infty_\beta$.
Applying \cref{lem:fixed-set-limit} to
\eqref{eq:classical-uniform-prelimit-bounds} gives
$s_\Lambda^2\geq v|\Lambda|>0$ and proves
\eqref{eq:classical-infinite-K-bound} and
\eqref{eq:classical-infinite-W-bound}, with
$C_{\beta,d}^{\infty}=C_{\beta,d}$.
\end{proof}

\bmhead{Acknowledgements}
Liu S.H. was partially supported by the Fundamental Research Funds for the Central Universities DUT25RC(3)133.
Shao Q.M. was partially supported by National Nature Science Foundation of China NSFC 12031005 and Shenzhen Outstanding Talents Training Fund, China.

\vspace{1cm}

\bmhead{AI Statement}
During the preparation of this manuscript, the authors used ChatGPT for language editing, LaTeX formatting, bibliographic verification, and checks of notation and cross-references, as well as for assistance with graph-covering arguments in the torus-to-lattice comparison for the fully periodized model. The authors reviewed the manuscript, verified all mathematical arguments, and take full responsibility for its content.





\end{document}